\documentclass{article}

\usepackage{enumitem}
\usepackage{microtype}
\usepackage{graphicx}
\usepackage{subcaption}
\usepackage{booktabs, caption, threeparttable}
\usepackage[labelfont=bf]{caption}
\usepackage{multirow}
\usepackage{hyperref}

\graphicspath{{images/}} 

\usepackage[accepted]{icml2023}
\usepackage{tabularray}

\usepackage{amsmath}
\usepackage{amssymb}
\usepackage{mathtools}
\usepackage{amsthm}
\newtheorem{example}{Example}

\newcommand{\R}{\mathbb{R}}
\newcommand{\E}{\mathbb{E}}
\newcommand{\I}{\mathbb{I}}
\usepackage{bm}

\usepackage[capitalize,noabbrev]{cleveref}

\theoremstyle{plain}
\newtheorem{theorem}{Theorem}[section]

\newtheorem{lemma}[theorem]{Lemma}
\newtheorem{corollary}[theorem]{Corollary}
\theoremstyle{definition}
\newtheorem{definition}[theorem]{Definition}
\newtheorem{assumption}[theorem]{Assumption}
\theoremstyle{remark}

\usepackage[textsize=tiny]{todonotes}

\icmltitlerunning{Generalized Polyak Step Size with Momentum}

\begin{document}

\twocolumn[
\icmltitle{Generalized Polyak Step Size for First Order Optimization with Momentum}



\icmlsetsymbol{equal}{*}

\begin{icmlauthorlist}
\icmlauthor{Xiaoyu Wang}{hkust}
\icmlauthor{Mikael Johansson}{kth}
\icmlauthor{Tong Zhang}{hkust}
\end{icmlauthorlist}

\icmlaffiliation{kth}{Royal Institute of Technology (KTH)}
\icmlaffiliation{hkust}{The Hong Kong University of Science and Technology}
\icmlcorrespondingauthor{Xiaoyu Wang}{maxywang@ust.hk}

\icmlkeywords{Machine Learning, ICML}

 \vskip 0.3in
]



\printAffiliationsAndNotice{} 

\begin{abstract}
In machine learning applications, it is well known that carefully designed learning rate (step size) schedules can significantly improve the convergence of commonly used first-order optimization algorithms. Therefore how to set step size adaptively becomes an important research question. A popular and effective method is the Polyak step size, which sets step size adaptively for gradient descent or stochastic gradient descent without the need to estimate the smoothness parameter of the objective function. However, there has not been a principled way to generalize the Polyak step size for algorithms with momentum accelerations. This paper presents a general framework to set the learning rate adaptively for first-order optimization methods with momentum, motivated by the derivation of Polyak step size. It is shown that the resulting techniques are 
 much less sensitive to the choice of momentum parameter and may avoid the oscillation of the heavy-ball method on ill-conditioned problems. These adaptive step sizes are further extended to the stochastic settings, which are attractive choices for stochastic gradient descent with momentum. 
 Our methods are demonstrated to be more effective for stochastic gradient methods than prior adaptive step size algorithms in large-scale machine learning tasks. 

\end{abstract}

\section{Introduction}
\label{sec:intro}

We consider stochastic optimization problems on the form
\begin{align}\label{P1}
    \min_{x \in \R^d} f(x):= \E_{\xi \sim \Xi }[f(x;\xi)]
\end{align}
where $\xi$ is a random variable with probability distribution $\Xi$ and $f(x;\xi)$ is the instantaneous realization of $f$ with respect to $\xi$. We use $X^{\ast}$ to denote the set of minimizers of (\ref{P1}), which we assume is non-empty. In other words, there is at least one  $x^{\ast} \in \R^d$ such that $f^{\ast}=f(x^{\ast})=\min f(x)$. 

Stochastic gradient descent (SGD)~\citep{SGD-1951} has been the workhorse for training machine learning models. To accelerate its practical performance, one often adds a momentum term to SGD, leading to algorithms such as SGDM~\citep{sutskever2013importance}. SGDM has been widely used in deep neural networks due to its empirical success and is a default choice in machine learning libraries (PyTorch and TensorFlow). However, its practical performance relies heavily on the choice of the step size (learning rate) that controls the rate at which the model learns. 


In the traditional optimization literature, Polyak's heavy ball (momentum) method~\cite{polyak1964some}
is a well-known technique to accelerate gradient descent.  
By accounting for the history of the iterates, it achieves a linear convergence that is substantially faster than gradient descent on ill-conditioned problems. 
However, to achieve its optimal performance, the heavy-ball method relies on 
a 
specific combination of momentum parameter $\beta$ and step size $\eta$ adapted to the condition number of the problem. One downside of the method is that its empirical performance is very sensitive to the momentum factor $\beta\in (0,1)$ (see Figure \ref{fig:ls:beta}), which makes the method difficult to use when the condition number is unknown.  For the heavy-ball method, each component of the decision vector is updated independently and shares the same step size. Even in very simple examples,  
the method exhibits a zigzag phenomenon in the dimension with large curvature~\citep{polyak1964some}, leading to a slow and oscillatory convergence.
Thus, an adaptive step size is important for the best practical behavior of the heavy-ball method. 


 For convex functions whose optimal value is known a priori, the Polyak step size, which depends on the current function value and the magnitude of the gradient (subgradient), is optimal in a certain sense 
~\citep{polyak_book,CFM,brannlund1995generalized}. \citet{hazan2019revisiting} revisited the Polyak step size and proved near-optimal convergence even when the optimal value is unknown. Still, it is impractical to use the deterministic Polyak step size due to the computation of exact function values and gradients in each iteration.
Recently, there has been a strong interest in developing adaptive step size policies that are inspired by the classical Polyak step size~\citep{L4, SGD_polyak, ALIG, SPS, pmlr-v134-sebbouh21a}. This line of research has been particularly successful on overparameterized models. \citet{SPS} extended the Polyak step size to the stochastic setting and proposed a stochastic Polyak step size (SPS). \citet{ALIG} made explicit use of the interpolation property to design a step size policy for SGD in a closed form (called ALI-G) and incorporated regularization as a constraint to promote generalization. The experiments in \citet{ALIG,berrada2021comment}  used momentum without and theoretical guarantee and demonstrated that it could significantly improve the practical performance. This highlights the importance of adaptive step size policies for momentum methods. However, the
existing research on Polyak step sizes has focused on SGD, and rarely designed adaptive step sizes for heavy-ball and momentum algorithms.

\subsection{Motivation}
To demonstrate the challenges that arise in adapting the Polyak step size to momentum algorithms, we consider the approach that underpins L$^4$Mom~\citep{L4}. The key idea is to linearize the loss function at the current iterate, 
$f(x_k - \eta d_k) \approx f(x_k) - \eta \left\langle \nabla f(x_k), d_k\right\rangle$ and then choose $\eta_k$  so that the linearized  prediction of $f$ at the next iterate equals $f^{\ast}$.  
To account for the inaccuracy of the linear approximation, L$^4$Mom introduces a hyperparameter $\alpha>0$ and uses
$\eta = \alpha \frac{f(x_k) - f^{\ast}}{\left\langle \nabla f(x_k), d_k \right\rangle}$.  
However, we have found that this algorithm is quite unstable in practice, and fails on standard experiments such as 
the CIFAR100 experiments in Section \ref{sec:numerical:dnn}. This sensitivity is also observed in~\citep{ALIG}. One reason is that $\left\langle \nabla f(x_k), d_k \right\rangle$ is not always guaranteed to be positive. We experience difficulties with the algorithm even in a simple least-squares problem with condition number $\kappa=10^4$  and $f^{\ast}=0$. 
 \begin{figure}[ht]
\begin{center}
 \vskip -0.1in
\includegraphics[width=0.23\textwidth,height=1.4in]{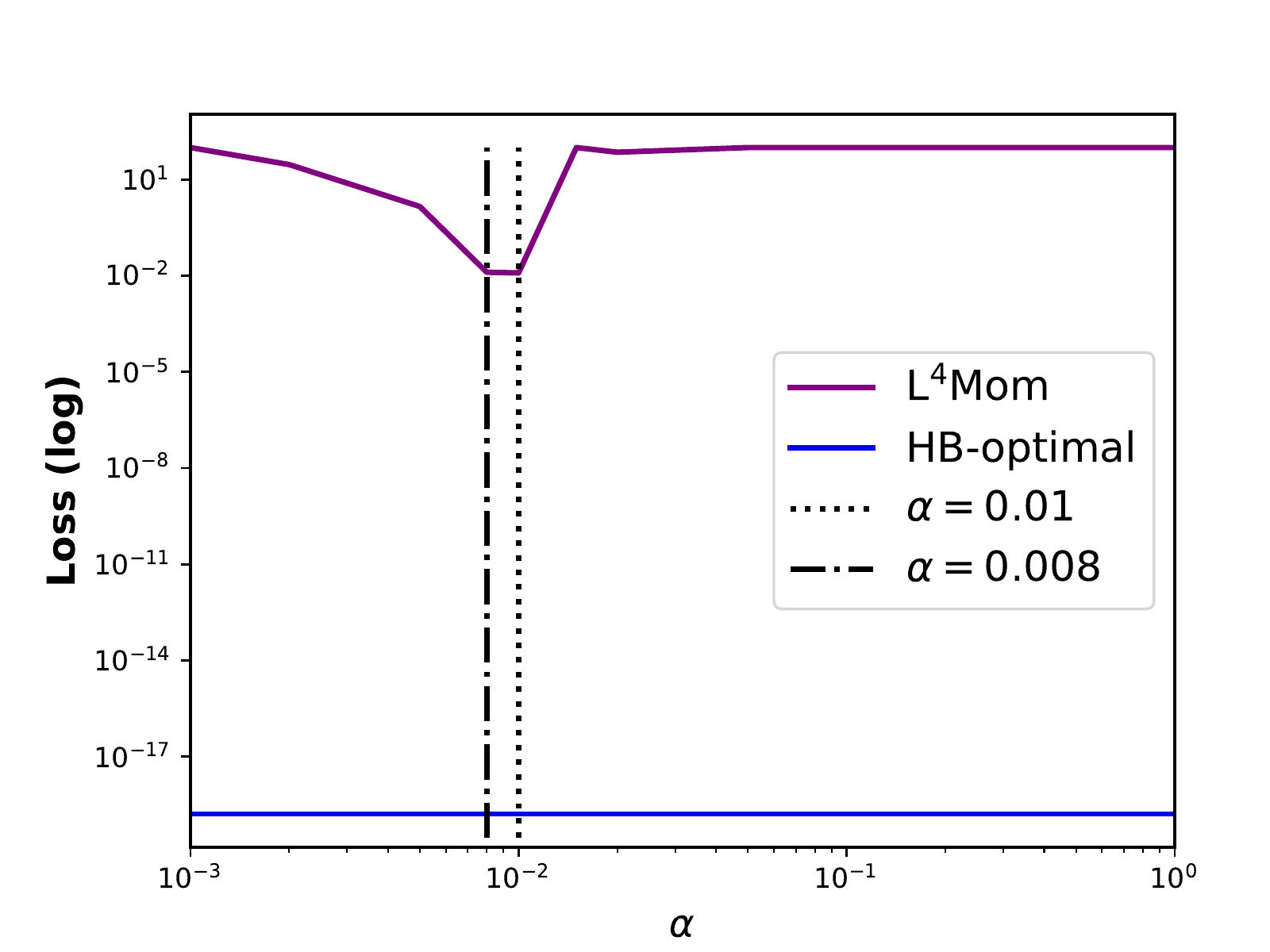}
\includegraphics[width=0.23\textwidth,height=1.4in]{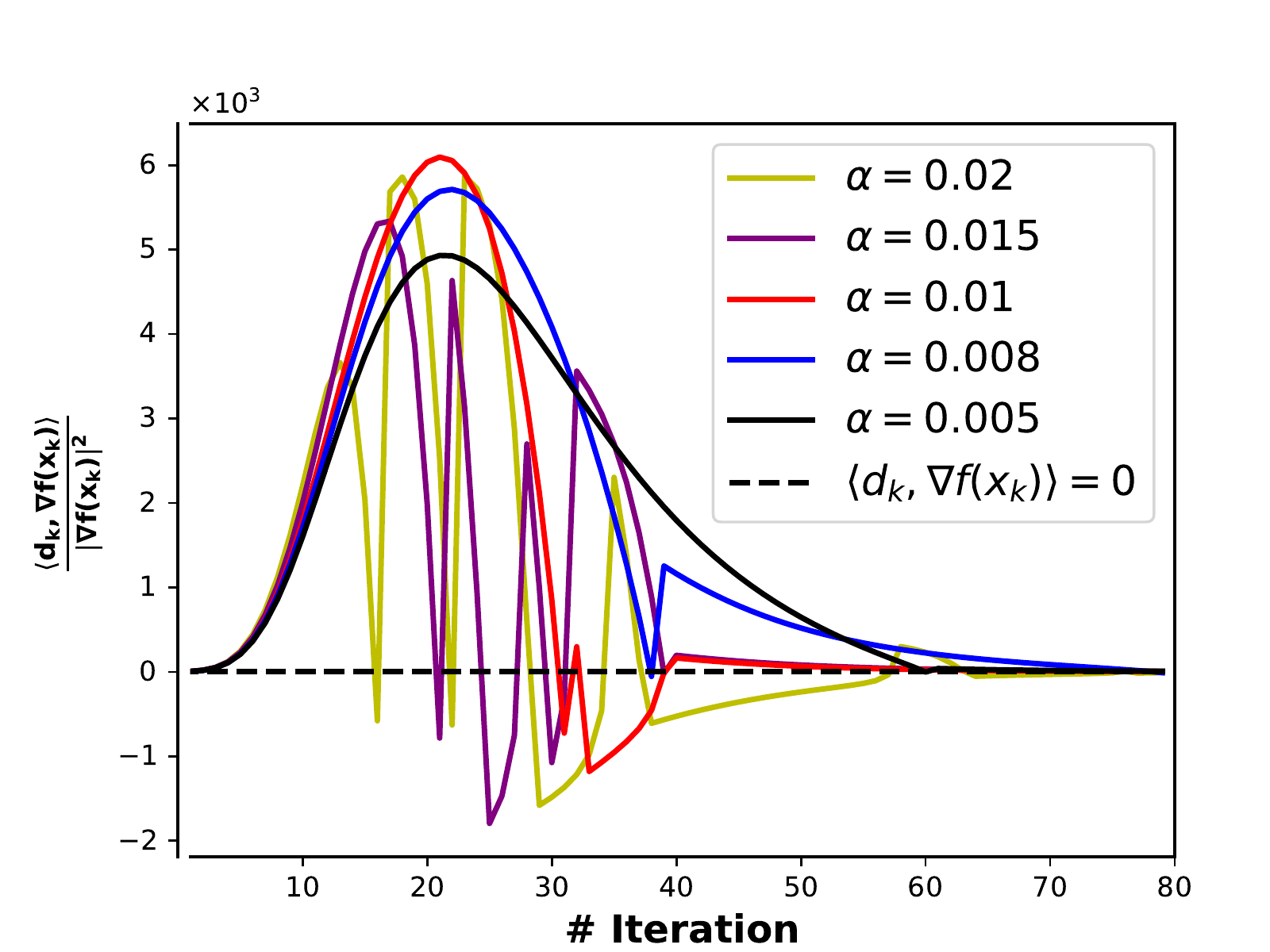}
\vskip -0.1in
\label{fig:l4:mom}
\end{center}
\end{figure}
 The parameter $\alpha$ is crucial for the empirical convergence: for large values of $\alpha$, the algorithm easily explodes, while small values of $\alpha$ result in slow convergence. In brief, $L^4$ is not an ideal approach for finding an adaptive step size for heavy-ball or momentum acceleration. Besides, there is no theoretical guarantee for the L$^4$Mom algorithm. Our goal is to find adaptive step sizes for the momentum acceleration algorithms that are more stable and efficient in practice.

\subsection{Contribution}
{\bf A new perspective on adaptive step sizes for momentum.} Inspired by the success of the Polyak step size for subgradient methods and SGD, and the absence and insufficiency of adaptive Polyak step sizes for momentum accelerations, we propose a generic Adaptive Learning Rate (ALR) framework for two variants of momentum methods: heavy-ball (HB) and moving averaged gradient (MAG). We call corresponding adaptive algorithms ALR-HB and ALR-MAG and make the following contributions:
\begin{enumerate}[topsep=0.5pt,itemsep=0.5ex,partopsep=1ex,parsep=1ex]
    \item[(i)] We prove global linear convergence of ALR-MAG on semi-strongly convex and smooth functions, improving the results for modified subgradient methods in  \citep{brannlund1995generalized}, under less restrictive assumptions.
    \item[(ii)] For least-squares problems, we demonstrate that ALR-HB and ALR-MAG are less sensitive to the choice of $\beta$ than the original heavy-ball method. 
    Our algorithms are significantly better than heavy-ball, gradient descent with Polyak step size, and L$^4$Mom if the condition number is unknown a priori. 
    \item[(iii)] The proposed framework is also applicable to Nesterov accelerated gradient (NAG)~\citep{nesterov1983} and performs better than the original Nesterov momentum under optimal parameters (see Appendix~\ref{sec:model:nag}).
\end{enumerate}

{\bf Stochastic extensions of ALR-HB and ALR-MAG.} We extend ALR-HB and ALR-MAG to the stochastic setting and call them ALR-SHB and ALR-SMAG, respectively. We make the following contributions:
\begin{enumerate}[topsep=0.5pt,itemsep=0.5ex,partopsep=1ex,parsep=1ex]
 \item[(i)] Under the assumption of interpolation (overparameterized models),  we prove a linear convergence rate for ALR-SMAG on semi-strongly convex and smooth functions. Such a result did not exist for SGD with momentum under this class of step sizes.
 \item[(ii)] We demonstrate the superiority of ALR-SHB and ALR-SMAG over state-of-the-art adaptive methods and the popular step-decay step size~\citep{ge2019step}
    on logistic regression and deep neural network training. By incorporating a warmup technique into the upper bound of the step size, the performance of ALR-SHB and ALR-SMAG can be improved further and is better than step-decay.
    \item[(iii)] We incorporate weight-decay into the update of ALR-SMAG to improve the generalization. The algorithm performs better than ALI-G with momentum and step-decay step size and is comparable to cosine step size without restart~\citep{loshchilov2016sgdr}.
\end{enumerate}

\section{Adaptive Step Sizes}\label{sec:model}

 Consider a general first-order method with momentum acceleration on the form 
\begin{align}\label{unified:mom}
     x_{k+1}  = x_k - \eta_k d_k + \gamma (x_k - x_{k-1})
 \end{align}
where $-d_k$ is a descent direction. A natural question that arises is how far we should move in this direction to converge quickly. In theoretical analyses, the quantity $\left\|x - x^{\ast} \right\|^2$ is often used to measure the convergence of the algorithms. We therefore propose to optimize $\eta_k$ to ensure that $x_{k+1}(\eta_k)$ minimizes this quantity, \emph{i.e.},  \begin{align}\label{main:inequ:lr} 
 \min_{\eta_k} \left\| x_{k+1}(\eta_k) - x^{\ast} \right\|^2.
    \end{align}
 Minimizing (\ref{main:inequ:lr}) w.r.t $\eta_k $ suggests that
\begin{align}\label{mom:lr:original}
     \eta_k = \frac{\left\langle d_k, x_k -x^{\ast}\right\rangle}{\left\|d_k \right\|^2} + \gamma \frac{\left\langle d_k, x_k -x_{k-1}\right\rangle}{\left\|d_k \right\|^2}.
 \end{align}
In general, the minimizer $x^{\ast}$ is not accessible. However, when $f$ is convex, we can often evaluate a lower bound of $\left\langle d_k, x_k -x^{\ast}\right\rangle$ and minimize an upper bound of (\ref{main:inequ:lr}) 
\vskip -0.2in 
{\small
\begin{align}\label{main:inequ:lr:upper}
 \hspace{-0.1in} \left\|x_{k+1} -x^{\ast} \right\|^2 & =
    \Vert x_k-x^{\ast}\Vert_2^2 
    + \eta_k^2 \Vert d_k\Vert_2^2  - 2\eta_k \gamma \langle d_k, x_k-x_{k-1}  \rangle \notag\\ 
    &
    -2\eta_k \left\langle d_k, x_k -x^{\ast}\right\rangle + \mathcal{O}(1).
\end{align}
}\hspace{-0.05in}
For example,
if $d_k \in \partial f(x_k)$ and $\gamma=0$, the method of (\ref{unified:mom}) reduces to the subgradient method. By the convexity of $f$, $\left\langle \nabla f(x_k), x_k -x^{\ast} \right\rangle \geq f(x_k) - f^{\ast} $, and minimizing the upper bound of (\ref{main:inequ:lr:upper}) results in the Polyak step size $\eta_k = (f(x_k) - f^{\ast})/\left\| d_k \right\|^2$~\citep{subgradient_mohktar}. Whatever other model we may have that provides a lower bound on the inner product $\left\langle d_k, x_k -x^{\ast} \right\rangle$ will also work in this framework.
In the rest of this paper, we focus on two popular variants of momentum acceleration.

\subsection{Adaption for Heavy-Ball}\label{sec:model:hb}
 We first consider the heavy-ball (HB) method~\citep{polyak1964some, global-hb} given by
 \begin{align}\label{alg:HB}
  x_{k+1} = x_k - \eta_k\nabla f(x_k) + \beta (x_k - x_{k-1}) 
 \end{align}
 where $\beta \in (0,1)$ is a constant. Clearly, heavy ball is a special case of (\ref{unified:mom}) where $\gamma=\beta$ and $d_k = \nabla f(x_k)$.
By the convexity of $f$, we have a lower bound for $\left\langle \nabla f(x_k), x_k -x^{\ast} \right\rangle $ by $f(x_k) - f^{\ast} $ and minimizing the upper bound of  (\ref{main:inequ:lr:upper}) yields the adaptive learning rate for heavy-ball (\emph{ALR-HB})
\begin{align}\label{mad:hb:lr:1}
\eta_k = \frac{f(x_k) - f^{\ast}}{\left\| \nabla f(x_k) \right\|^2} + \beta \frac{\left\langle \nabla f(x_k), x_k -x_{k-1} \right\rangle}{\left\|\nabla f(x_k) \right\|^2}. 
\end{align}

If the objective function $f$ is also $L$-smooth then $\left\langle \nabla f(x_k), x_k -x^{\ast} \right\rangle \geq f(x_k) - f(x^{\ast}) + \frac{1}{2L}\left\|\nabla f(x_k) \right\|^2,$ which is a tighter lower bound for $\left\langle \nabla f(x_k), x_k -x^{\ast} \right\rangle$. It results in the formula (\ref{mad:hb:lr:L}) below, named ALR-HB(v2), which has an additional constant term $1/(2L)$ compared to (\ref{mad:hb:lr:1}):
\begin{align}\label{mad:hb:lr:L}
\eta_k = \frac{1}{2L} + \frac{f(x_k) - f^{\ast}}{\left\| \nabla f(x_k) \right\|^2} + \beta \frac{\left\langle \nabla f(x_k), x_k -x_{k-1} \right\rangle}{\left\|\nabla f(x_k) \right\|^2}.
\end{align}
The ALR-HB algorithms are shown in Algorithm \ref{alg:mad:hb}. Our next example shows that ALR-HB (v2) can find the exact solution for a simple least-squares problem in a single step.  
\begin{example}
Consider one-dimensional least-squares problem $f(x) = \frac{1}{2}hx^2$. For Polyak with gradient descent and $L^4$Mom, we need at least $k = \log_2(x_0/\epsilon)$ steps for an $\epsilon$-accurate solution $(|x-x^{\ast}|\leq \epsilon)$.  For ALR-HB(v2), given $x_0, x_1$, we only need one step to find the exact solution.
\end{example}
\begin{proof}
The step size of ALR-HB(v2) can be written as 
\begin{align*}
    \eta_k & = \frac{1}{2L} + \frac{f(x_k) - f^{\ast}}{\left\|\nabla f(x_k) \right\|^2} + \beta \frac{\left\langle \nabla f(x_k), x_k -x_{k-1} \right\rangle}{\left\|\nabla f(x_k) \right\|^2} \notag \\
    & = \frac{1}{h} + \beta\frac{1}{h}\left( 1- \frac{x_{k-1}}{x_k} \right)
\end{align*}
Applying the step size to the iterate of HB gives 
\begin{align*}
   x_{k+1}  & = x_k - \eta_k h x_k + \beta(x_k -x_{k-1}) = 0.
\end{align*} 
\vskip -0.2in
\end{proof}
\vskip -0.1in
Thus, we believe that the model (\ref{main:inequ:lr}) is a good choice for designing adaptive step sizes for the heavy-ball method. 
\begin{algorithm}[t]
\caption{ALR-HB}\label{alg:mad:hb}
\begin{algorithmic}[1]
   \STATE {\bfseries Input:} initial point $x_1$, $\beta \in (0,1)$, $v_0=\bm{0}$
\WHILE{$x_k$ does not converge do}
\STATE{$ k \leftarrow k+1 $}
 \STATE $\eta_k \leftarrow  \frac{f(x_k) - f(x^{\ast})}{\left\| \nabla f(x_k) \right\|^2} + \beta \frac{\left\langle \nabla f(x_k), x_k -x_{k-1} \right\rangle}{\left\|\nabla f(x_k) \right\|^2}$ {(v1)}; \\  $\eta_k \leftarrow \frac{1}{2L} + \frac{f(x_k) - f^{\ast}}{\left\|\nabla f(x_k) \right\|^2} + \beta \frac{\left\langle \nabla f(x_k), x_k -x_{k-1} \right\rangle}{\left\|\nabla f(x_k) \right\|^2}$ (v2)
      \STATE $v_k \leftarrow  - \eta_k \nabla f(x_k) + \beta v_{k-1}$
		\STATE $x_{k+1} \leftarrow x_k + v_k $
\ENDWHILE
\end{algorithmic}
\end{algorithm}

\subsection{Adaptive Step Size for MAG}\label{sec:model:mag}

Next, we consider the  moving averaged gradient (MAG), another widely used momentum variant for deep learning 
\begin{align}\label{alg:mag}
       d_{k} = \nabla f(x_k) + \beta d_{k-1}, \,\,\,  x_{k+1} = x_k - \eta_k d_{k} 
    \end{align}
where $\beta \in (0,1)$. 
Note that if the step size is constant,  $\eta_k = \eta$, then the formulas (\ref{alg:HB}) and (\ref{alg:mag}) are equivalent. However, we consider adaptive step sizes that change with $k$, and in this case, the two methods are  
different variants of momentum.

If the search direction $d_k$ is defined by (\ref{alg:mag}) and $\gamma=0$, the update of (\ref{unified:mom}) reduces to the MAG algorithm. 
 By the convexity of $f$,  if $\eta_{i} \leq \frac{f(x_{i}) - f^{\ast}}{\left\| d_{i} \right\|^2}$ for all $i\leq k-1$, Lemma \ref{lem:mag:convex} in our subsequent theoretical analysis shows that $\left\langle d_{k-1}, x_k -x^{\ast}\right\rangle \geq 0$. We therefore provide a lower bound for $\left\langle d_k, x_k -x^{\ast}\right\rangle = \left\langle \nabla f(x_k) + \beta d_{k-1}, x_k -x^{\ast} \right\rangle \geq \left\langle \nabla f(x_k), x_k -x^{\ast} \right\rangle  \geq f(x_k) - f^{\ast}$. Minimizing the upper bound of (\ref{main:inequ:lr:upper}) results in step size:
\begin{align}\label{mag:lr}
        \eta_k = \frac{f(x_k) - f^{\ast}}{\left\|d_k \right\|^2}. 
     \end{align}
     We refer to this adaptive momentum version, detailed in  Algorithm \ref{alg:mad:mag}, as ALR-MAG. 
    Lemma \ref{lem:mag:xxstar} in Section \ref{sec:thm:dc} shows that the iterates of ALR-MAG decrease monotonically  w.r.t. the distance $\left\|x -x^{\ast}\right\|^2$. This guarantees that the iterates come closer and closer to the optimum. Our next example in Section~\ref{example:mag} demonstrates that the step size of ALR-MAG is able to avoid oscillations of the heavy-ball method.
\begin{algorithm}[ht]
\caption{ ALR-MAG}\label{alg:mad:mag}
\begin{algorithmic}[1]
   \STATE {\bfseries Input:} initial point $x_1$, $\beta \in (0,1)$, $d_0=\bm{0}$
\WHILE{$x_k$ does not converge}
\STATE{$ k \leftarrow k+1 $}
\STATE $d_k \leftarrow \beta d_{k-1} + \nabla f(x_k)$
 \STATE $
\eta_k \leftarrow  \frac{f(x_k) - f^{\ast}}{\left\| d_k \right\|^2} $
		\STATE $x_{k+1} \leftarrow x_k - \eta_k d_k $
\ENDWHILE
\end{algorithmic}
\end{algorithm}

\subsection{Justification of ALR-MAG}\label{example:mag}
To demonstrate the advantages of ALR-MAG, we consider a simple two-dimensional least-squares problem $f(x, y) = \frac{1}{2}(x-1)^2 +  \frac{\kappa}{2}(y+1)^2$ with $x \in \R$ and $y \in \R$. We set $\kappa=100$ and use the initial point $(x_0, y_0)=(48, -28)$. For the classic heavy-ball method, the iterates can be re-written as $ x_{k+1}  = x_k - \eta_k (x_k-1) + \beta (x_k - x_{k-1});
    y_{k+1}  = y_k - \eta_k \kappa (y_k+1) + \beta (y_k - y_{k-1})$. Note that the variables $x$ and $y$ are updated independently and share the same step size. Our baseline is the optimal parameters for heavy-ball from~\cite{polyak1964some}, 
 $\beta^{\ast}=(\sqrt{\kappa}-1)^2/(\sqrt{\kappa}+1)^2$ and $\eta^{\ast}=(1+\sqrt{{\beta^{\ast}}})^2/L$ (called HB-optimal).
    From Figure~\ref{fig:mag:ls:lr}(left), we observe a pronounced zigzag behavior in the $y$-dimension for HB-optimal. The step size $\eta^{\ast}$ is large and results in an undamped and slow convergence in the dimension with large curvature (\emph{i.e.,} $y$). We apply ALR-MAG with the same $\beta^{\ast}$ and $f^{\ast}=0$.  ALR-MAG adapts the step size to start from a small value to avoid the instability in the $y$-dimension and finally reaches a value that is comparable to $\eta^{\ast}$.


\begin{figure}[ht]
 \includegraphics[width=0.23\textwidth]{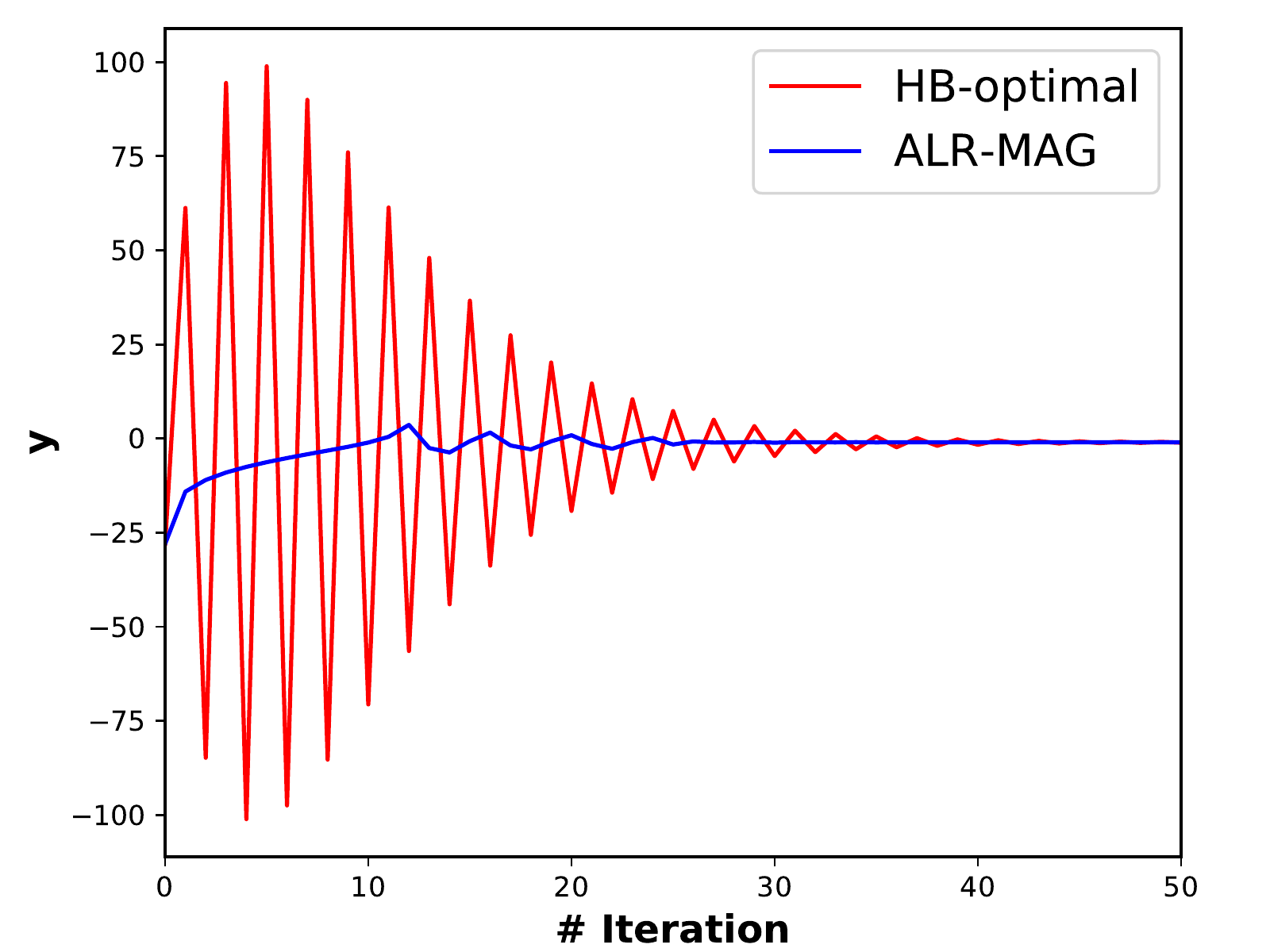}
\hfill 
\includegraphics[width=0.23\textwidth]{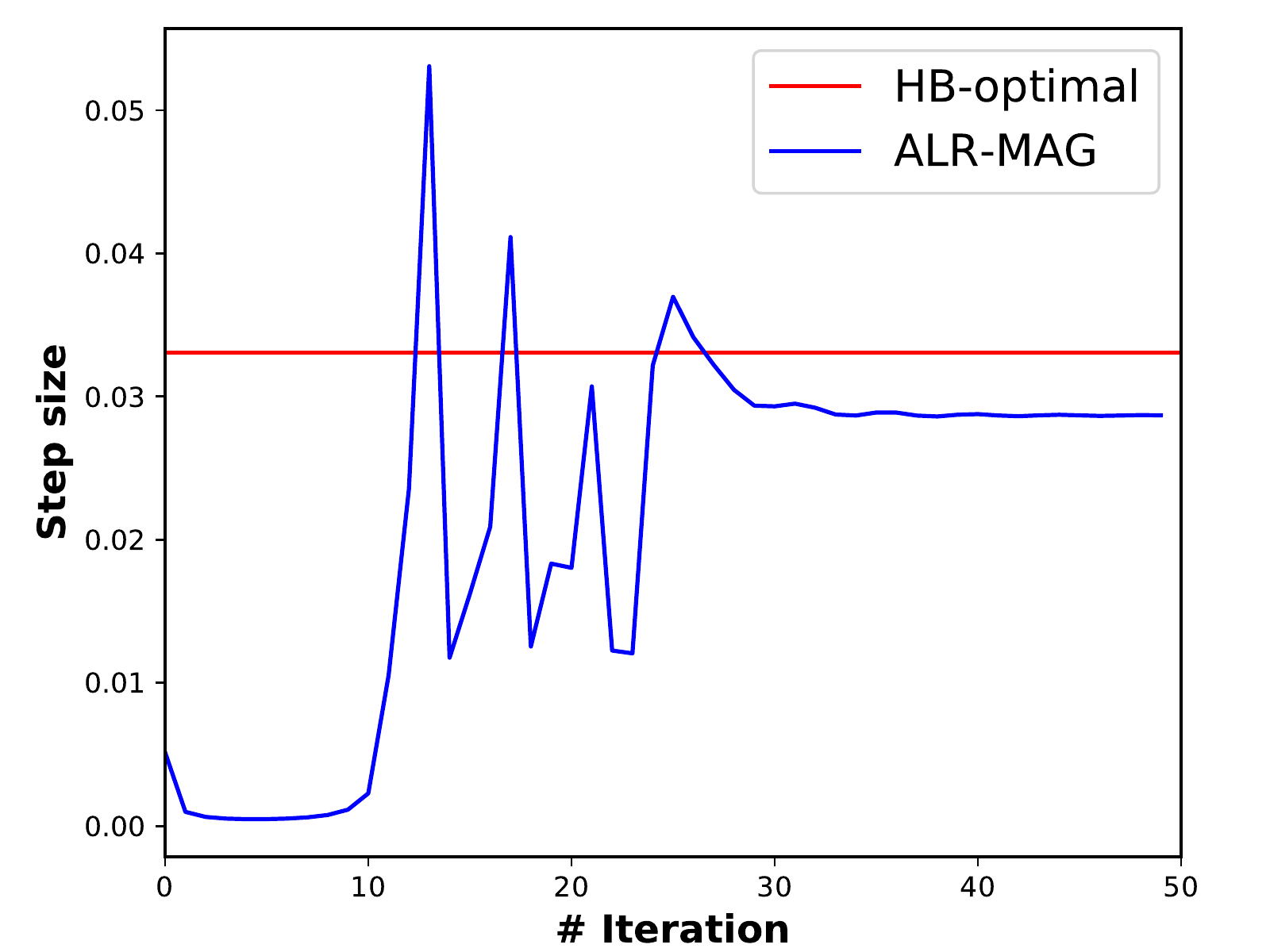}
\vskip -0.1in
\caption{The trace of variable $y$ (left) and the step sizes (right) }
     \label{fig:mag:ls:lr}
     \vskip -0.15in
\end{figure}

\section{Related Work}



 {\bf Adaptive methods for deterministic momentum.} For deterministic problems with $\mu$-strongly convex and $L$-smooth objective functions, \citet{polyak1964some} demonstrated that the fastest local convergence of the heavy-ball method is attained for the optimal parameters  $\beta^{\ast}=(\sqrt{\kappa}+1)^2/(\sqrt{\kappa}+1)^2$ and $\eta^{\ast} =(1+\sqrt{\beta^{\ast}})^2/L$ where $\kappa=L/\mu$ is the condition number. 
 Fast linear convergence is also achieved for Nesterov's accelerated gradient method with $\beta=(\sqrt{\kappa}-1)/(\sqrt{\kappa}+1)$ and $\eta=1/L$~\citep{nesterov2003}.
However, while $L$ is relatively easy to estimate on-line, $\mu$ (and therefore $\kappa$) 
is often inaccessible. A number of recent contributions suggest ideas for approximating these optimal hyper-parameters at each iteration.
\citet{AGM} adaptively estimate the strong convexity
constant by the inverse of the Polyak step size and use this estimate in place of the true $\mu$ in the momentum parameter 
for Nesterov momentum. \citet{AHB} approximate the Lipschitz and strongly convexity constants employing
the absolute differences of current and previous model parameters and their gradients. 
However, its empirical performance, at least in the least-squares problem in Figure \ref{fig:ls:with} (labeled~\emph{AHB}), is poor. \citet{AdSGD} estimate the Lipschitz constant similarly to~\citep{AHB} and the strong convexity constant $\mu$ by the inverse smoothness of the conjugate function. They also add a conservation bound for the estimators of $\mu, L$, leading to a method with four hyperparameters that need to be tuned. \citet{goujaud2022quadratic} develop momentum-based Polyak step sizes for convex quadratic functions. 


{\bf Adaptive step sizes for stochastic algorithms.} \citet{SLS}
 extend line search methods to the stochastic setting (called SLS) using the function and
gradient of a mini-batch and guarantee linear convergence under interpolation. However, many hyper-parameters make it difficult to use in practice.
\citet{AdSGD}
use their estimation technique for $L$ (discussed above) to develop an adaptive step size for SGD (called \emph{AdSGD}). Under interpolation,  the iteration complexity is $\kappa$ times higher than SGD.

{\bf Adaptive gradient methods.} Adaptive gradient methods, such as AdaGrad~\citep{AdaGrad}, RMSProp~\citep{RMSProp}, Adam~\citep{Adam}, and AdamW~\citep{adamw} are very popular in practice. However, adaptive gradient methods have poor generalization compared to SGD in supervising learning tasks~\citep{wilson2017marginal}. \citet{adam_warmup} suggest a learning rate warmup heuristic in the early stage of training that can improve the generalization of adaptive methods.


\section{Preliminaries and Convergence Analysis}\label{sec:theory}

Before presenting our theoretical results, we introduce a few key concepts and the notation used throughout the paper.

\begin{definition}
    $f:\mathbb{R}^d\mapsto \mathbb{R}$ is semi-strongly convex if there exists a constant $\hat{\mu} > 0$ such that $    \frac{\hat{\mu}}{2}\left\|x - x^{\ast} \right\|^2 \leq f(x) - f^{\ast}, \forall x \in \R^d.$
\end{definition}
This condition 
is also called the quadratic growth property of $f$. If the function is convex and smooth,  semi-strong convexity is equivalent to the Polyak-\L{}ojasiewicz (PL) condition~\citep{karimi2016linear}. This is a weaker condition than the strong convexity. The definitions of convexity, strong convexity, and $L$-smoothness are provided in Appendix~\ref{appendix:determinstic}.

{\bf Interpolation.}  We say that the interpolation condition holds if there exists $x^{\ast} \in \mathcal{X}^{\ast}$ such that  individual functions $\min_{x} f(x;\xi) = f(x^{\ast}; \xi)$ for all $\xi \in \Xi$. All loss functions $f(x;\xi)$ meet with a common minimizer $x^{\ast}$. The interpolation property is satisfied in many machine learning models, including linear classifiers with separable data, over-parameterized deep neural networks~\citep{ma2018power,zhang_interpolation}, non-parametric regression~\citep{liang2020just}, and boosting~\citep{bartlett1998boosting}.  


\subsection{ALR-MAG in Deterministic Optimization} \label{sec:thm:dc}
We first provide convergence guarantees for the ALR-MAG method on deterministic convex optimization problems where the exact gradient and function values are available.

Our first lemma shows that the direction $d_{k-1}$ forms an acute angle with the direction from $x_k$ to the minimizer $x^{\ast}$.
\begin{lemma}\label{lem:mag:convex}
Let $f$ be convex and assume that $\{x_i\}_{i=0}^k$ has been generated by MAG with $\eta_i \leq (f(x_i) - f^{\ast})/\left\|d_i \right\|^2$ for all $i\leq k-1$. Then,  
$\left\langle d_{k-1}, x_k -x^{\ast}\right\rangle \geq 0$.
\end{lemma}

The next lemma establishes that the MAG iterates under the step size (\ref{mag:lr}) are monotone decreasing with respect to the distance $\left\|x - x^{\ast} \right\|^2$. This guarantees that the next iterate $x_{k+1}$ is closer to the minimizer $x^{\ast}$ than the current $x_k$.
 \begin{lemma}\label{lem:mag:xxstar}
 Let $\{x_k\}$ be generated by MAG with the step size defined in (\ref{mag:lr}). Then, if $f$ is convex $$ 
\left\| x_{k+1} -x^{\ast} \right\|^2 \leq \left\|x_{k} - x^{\ast} \right\|^2 -  \eta_k \left(f(x_k) - f^{\ast} \right).$$
If, in addition, $f$ is $L$-smooth, then  
   \small{
\begin{align*}
\hspace{-0.1in} 
\left\| x_{k+1} -x^{\ast} \right\|^2 \leq  \left\|x_{k} - x^{\ast} \right\|^2 - \left(\eta_k  +  \frac{(1-\beta)}{L} \right)\left(f(x_k) - f^{\ast} \right).
    \end{align*}}
     \end{lemma}

From Lemma \ref{lem:mag:xxstar}, we can note that if $f$ is smooth, there is an extra decrease compared to when $f$ is only convex. If the function is also semi-strongly convex, then we have the following linear convergence result.
\begin{theorem}
\label{thm:mag:smooth}
Suppose that function $f$ is convex and $L$-smooth and consider the ALR-MAG algorithm under the step size (\ref{mag:lr}).  
If $f$ is semi-strongly convex with $\hat{\mu} >0$, then
\begin{align}
    \left\|x_k -x^{\ast} \right\|^2 \leq (1-\rho)^k \left\|x_1-x^{\ast} \right\|^2
\end{align}
where $\rho = (1-\beta)(2\kappa)^{-1}$ and $\kappa = L/\hat{\mu}$. 
\end{theorem}
\citet{brannlund1995generalized} generalizes the subgradient method to use a convex combination of previous subgradients. Hence, the MAG algorithm (\ref{alg:mag}) can be seen as a special case of \citet{brannlund1995generalized}.
Since the sequence $\left\| x_{k+1} -x^{\ast} \right\|^2 $ is monotone decreasing, it is enough to require $L$-smoothness for all $x$ with $\Vert x-x^{\ast}\Vert \leq \Vert x_0-x^{\ast}\Vert$, which matches the smoothness assumption of Theorem 2.6 in \citep{brannlund1995generalized}. However, compared to~\citep{brannlund1995generalized},  we do not have the extra restriction $\left\|d_k \right\|\leq \left\|\nabla f(x_k) \right\|$. In fact, this requirement on $d_k$ is, in general, not satisfied for $\beta \in (0,1)$. Note that Theorem \ref{thm:mag:smooth} improves the dependence of the condition number $\kappa$ from $\kappa^2$ to $\kappa$ in the convergence of   \citep{brannlund1995generalized,shor2012minimization} and only requires semi-strong convexity (while Theorem 2.12 in \citep{shor2012minimization} assumes strong convexity). More results of ALR-MAG for functions without smoothness are provided in Appendix~\ref{appendix:determinstic}.


\subsection{Convergence of ALR-MAG in Stochastic Settings}\label{sec:thm:sc}
Next, we extend the adaptive step size for HB and MAG to the case when gradients and function values are sampled from an underlying (data) distribution. It is typically not practical to compute the exact function value and gradient in every step. Instead, we evaluate a mini-batch $S_k$ of gradient and function value samples in each iteration
\begin{align*}
 \hspace{-0.1in}   f_{S_k}(x) = \frac{1}{|S_k|}\sum_{i \in S_k} f(x; \xi_i),  \nabla f_{S_k}(x) =  \frac{1}{|S_k|} \sum_{i \in S_k}\nabla f(x; \xi_i)
\end{align*}
and propose to use the following adaption of (\ref{mag:lr}):
{\small \begin{align}\label{mag:lr:sc}
    \eta_k = \min \left\lbrace \frac{f_{S_k}(x_k) - f_{S_k}^{\ast}}{c\left\|d_k \right\|^2}, \eta_{\max} \right\rbrace 
\end{align}}
\hspace{-0.05in}Here, $d_k = \beta d_{k-1} + \nabla f_{S_k}(x_k)$ and $f_{S_k}^{\ast} = \inf_x f_{S_k}(x)$. We refer to the stochastic version of MAG  as \emph{SMAG}, and as \emph{ALR-SMAG} when we use the adaptive step size (\ref{mag:lr:sc}). 

Their step size (\ref{mag:lr:sc}) has three modifications compared to (\ref{mag:lr}). First, while the immediate extension of (\ref{mag:lr}) would replace $f^{\ast}=f(x^{\ast})$ by $f_{S_k}(x^{\ast})$, we suggest to use $f_{S_k}^{\ast}$ instead. For example, in many machine learning problems with unregularized surrogate loss functions, we have $f_{S_k}^{\ast}=0$~\citep{peter_convexity_classifi}. For the loss with regularization for example $\ell_2$ regularization, when the mini-batches contain a single data, then  $f_{S_k}^{\ast}$ can be computed in a closed form for some standard loss function~\citep{SPS,peter_convexity_classifi}. Second, we introduce a hyper-parameter $c >0$ that controls the scale of the step size to account for the inaccuracy in function and gradients. 
Third, 
due to the convergence reason and to make it applicable to wide applications even nonconvex problems, we may restrict the step size to be upper bounded by $\eta_{\max} > 0$.

In the following results, we assume the finite optimal objective function difference which has been used in the analysis of stochastic Polyak step size~\citep{SPS}. 
\begin{assumption}({\bf Finite optimal objective difference})\label{assumpt:noise:fstar}
\begin{align*}
        \sigma^2 = \E[f_{S_k}(x^{\ast}) - f_{S_k}^{\ast}] = f(x^{\ast}) - \E[f_{S_k}^{\ast}] < +\infty
\end{align*}
where $f_{S_k}^{\ast} = \inf f_{S_k}(x)$.
\end{assumption}
Under interpolation, each individual function $f(x;\xi)$ attains its optimum at $x^{\ast}$ which implies that $\sigma = 0$. We focus on semi-strongly convex and smooth functions.

\begin{theorem}
\label{thm:semi:sc}
Suppose that the individual function $f(x;\xi)$ is convex and $L$-smooth for any $\xi \in \Xi$ and that Assumption \ref{assumpt:noise:fstar} holds. Consider ALR-SMAG with $c > 1$, if $f$ is semi-strongly convex with $\hat{\mu}$, then 
\begin{align*}
  \E[\left\|x_{K+1} - x^{\ast} \right\|^2] \leq \left(1- \rho_1\right)^K\left\|x_{1} - x^{\ast} \right\|^2  +\frac{2\eta_{\max} \sigma^2}{\rho_1(1-\beta)}
\end{align*}
where $\rho_1 = \min\left\lbrace \frac{(1-\beta)\left(c-1 \right)\hat{\mu}}{2c^2L} , \frac{(2c-1)\hat{\mu} \eta_{\max}}{2c}\right\rbrace$.
\end{theorem}


When $\beta=0$, step size (\ref{mag:lr:sc}) reduces to SPS\_${\max}$. Our result in Theorem \ref{thm:semi:sc} is comparable to theorem 3.1 of SPS\_${\max}$ for strongly convex functions. However, the numerical results show the superior performance of ALR-SMAG compared to SPS in a wide range of machine-learning applications. In Theorem \ref{thm:semi:sc}, we assume $c > 1$ to ensure that the step size is not too aggressive. For example, in the experiments on logistic regression in Section \ref{sec:experiment:logistic}, we will use $c=5$.  For the deep learning tasks (nonconvex), we suggest that $c < 1$. This coincides with parameter $c$ from SPS\_${\max}$ (they set $c=0.2$)~\citep{SPS}.

An important property of SPS~\citep{SPS} is that the step size is lower and upper-bounded. This is not the case for our step size. Since $d_k$ is a convex combination of all previous stochastic gradients, the scale of $d_k$ is controlled by the previous stochastic gradients. In general, it is not clear how $\left\|d_k\right\|$ is related to $\left\|\nabla f_{S_k}(x_k)\right\|$, which makes it challenging to analyze the convergence of SMAG under (\ref{mag:lr:sc}). A key step in our analysis is to establish the inequality (\ref{dk:nablaf}) to handle the moving averaged gradient. The main novelty of ALR-SMAG is that it provides a principled way to adapt the step size for SGD with momentum and guarantees linear convergence, which earlier techniques were unable to do~\citep{L4,ALIG,berrada2021comment}.

The constant term in the inequality in Theorem~\ref{thm:semi:sc}  can not be made arbitrarily small by decreasing the upper bound $\eta_{\max}$. We also observe this limitation in the stochastic Polyak step size; see Theorem 3.1 and Corollary 3.3 in SPS~\citep{SPS}.  Theorem \ref{thm:semi:sc} suggests that $\beta=0$ achieves the best result in theory. This is also an issue for the stochastic momentum analysis~\citep{ unified_momentum,liu2020improved}. 

Our next corollary provides a stronger convergence result if the model is expressive enough to interpolate the data. In this setting, we use no maximal learning rate.
\begin{corollary}\label{thm:interpolation:ssc}
Assume interpolation ($\sigma = 0$) and
suppose that all assumptions of Theorem \ref{thm:semi:sc} hold. Consider the step size (\ref{mag:lr:sc}) and $\eta_{\max}=\infty$.
Then 
\begin{align*}
\E[\left\|x_{K+1} - x^{\ast} \right\|^2 ]\leq \left(1- \rho_1^{'}\right)^K\left\|x_1 - x^{\ast} \right\|^2
\end{align*}
where $\rho_1^{'} = \frac{(1-\beta)\left(c-1 \right)\hat{\mu}}{2c^2L} $.
\end{corollary}
Under interpolation, ALR-SMAG  can converge to the optimal solution $x^{\ast}$ and achieves the fast linear convergence rate $\mathcal{O}\left((1-(1-\beta)\hat{\mu}/L)^k\right)$ under semi-strong convexity. We also provide the convergence results of ALR-SMAG for general convex functions in Theorem~\ref{thm:convex} (see Appendix~\ref{appendix:determinstic}).

In the end, we will compare the analysis above with other adaptive step sizes and stochastic momentum methods. The iterate complexity of AdSGD~\citep{AdSGD} is $\kappa$ higher compared to SGD
for adaptive estimation of the stepsize. Clearly, the complexity of ALR-SMAG under interpolation is better than that of AdSGD. SMAG under constant step size is equivalent to stochastic heavy-ball (SHB)~\citep{unified_momentum} and SGDM~\citep{liu2020improved}. 
In proposition 2 \citep{liu2020improved}, the constant step size is restricted to be smaller than a small number $(1-\beta)/(5L)$ when the common choice $\beta=0.9$ is applied.  While Corollary \ref{thm:interpolation:ssc} does not have any restriction for $\eta_{\max}$. 

\subsection{Stochastic Extension of ALR-HB.}

The same idea to ALR-SMAG in Section~\ref{sec:thm:sc}, 
we consider applying the mini-batch of the function
$f_{S_k} = \frac{1}{|S_k|}\sum_{i\in S_k}f(x;\xi_i)$
to the framework (\ref{main:inequ:lr}). A natural extension of ALR-HB to the stochastic setting is  
{\small
\begin{align}\label{mad:shb:lr}
\hspace{-0.05in} \eta_k = \min\left\lbrace \frac{f_{S_k}(x_k) - f_{S_k}^{\ast}}{c\left\|\nabla f_{S_k}(x_k) \right\|^2} + \frac{\beta \left\langle \nabla f_{S_k}(x_k), x_k -x_{k-1}\right\rangle }{\left\|\nabla f_{S_k}(x_k) \right\|^2}, \eta_{\max} \right\rbrace.
\end{align}
}We call the stochastic version of HB as SHB, and the algorithm SHB with step size (\ref{mad:shb:lr}) as ALR-SHB. 
Three changes are made compared to the direct generalization of (\ref{mad:hb:lr:1}). A similar discussion can be found in Section~\ref{sec:thm:sc}, which we omit here.  
In Appendix~\ref{sec:hb:ls}, we provide a theoretical guarantee for truncated ALR-HB on least-squares but leave other possible results of ALR-HB and ALR-SHB for the future.

\section{Numerical Evaluations}\label{sec:numerical:exp}
In this section, we evaluate the practical performance of the proposed adaptive step sizes.
We start with experiments on the least-squares problems for ALR-MAG and ALR-HB,  and continue by exploring the performance of the stochastic versions, ALR-SMAG and ALR-SHB, on large-scale convex optimization problems and deep neural networks training. For space concerns, the experiments in the convex interpolation setting are reported in appendix~\ref{sec:experiment:logistic}.
\subsection{Empirical Results on Ill-Conditioned Least-Squares}\label{sec:least:squares}
    We use the procedure described in~\citep{least_square_test} to generate test problems with
        $ f(x) = \frac{1}{2}\left\| A x - b \right\|^2$
    where $A \in \R^{d_1\times d}$ is positive definite, $b \in \R^{d_1}$ is a random vector, and the optimum $f^{\ast} = 0$. We report results for problems with $d_1=d=1000$ for which the condition number $\kappa$ of $A^{T}A$ is $10^{4}$. The strong convexity constant $\mu$  and smoothness constant $L$ are the smallest and largest eigenvalues of the matrix $A^{T}A$, respectively.

    We test \emph{ALR-HB} and \emph{ALR-MAG} against with several important methods including (1) gradient descent with Polyak step size (named GD-Polyak); (2) heavy-ball with the optimal parameters, i.e., $\beta^{\ast} = (\sqrt{\kappa}-1)^2/(\sqrt{\kappa}+1)^2$ and step size $\eta^{\ast} = (1+\sqrt{\beta^{\ast}})^2/L$~\citep{polyak1964some} (named HB-optimal); (3) L$^4$Mom~\citep{L4}; (4) AGM (variant II)~\citep{AGM}; (5) AHB~\citep{AHB}; (6)AdGD-accel~\citep{AdSGD}. 
    
    If $\mu$ and $L$ are known a priori, we set $\beta=\beta^{\ast}$ for ALR-HB and ALR-MAG, as HB-optimal. We perform a grid search for parameters that are not specified (see  Appendix~\ref{append:least:squares}). The results are shown in Figure \ref{fig:ls:with}. ALR-HB (v2) performs the best among these methods.  For the case that $\mu$ and $L$ are not known, momentum parameter $\beta$ is tuned from $\left\lbrace 0.5, 0.9, 0.95, 0.99 \right\rbrace$ for HB, L$^4$Mom, ALR-HB, and ALR-MAG. The results in Figure \ref{fig:ls:without} show that ALR-HB and ALR-MAG are much better than HB with best-tuned constant step size, L$^4$Mom,  and the other algorithms.

To illustrate the behavior of different step sizes, we plot ALR-HB and ALR-MAG with the theoretically optimal step size $\eta^{\ast}$ for HB in Figure~\ref{fig:ls:with}(right). We can see how  ALR-HB (v2) oscillates in a small range around $\eta^{\ast}$ and ALR-MAG  converges to a value that is slightly different than $\eta^{\ast}$. Without knowledge of $\mu$ and $L$, ALR-HB still varies around $\eta^{\ast}$ and captures the function's curvature (see Figure~\ref{fig:ls:without}(right)). We also plot the final loss of the algorithms in Figure \ref{fig:ls:beta} on different $\beta$ selected from the interval $[0.9, 1)$, which includes $\beta^{\ast}$. 
Clearly, ALR-HB is less sensitive to $\beta$ than the original heavy-ball method while L$^4$Mom is far worse.
\begin{figure}[ht]
\vskip -0.1in
\includegraphics[width=0.235\textwidth,height=1.5in]{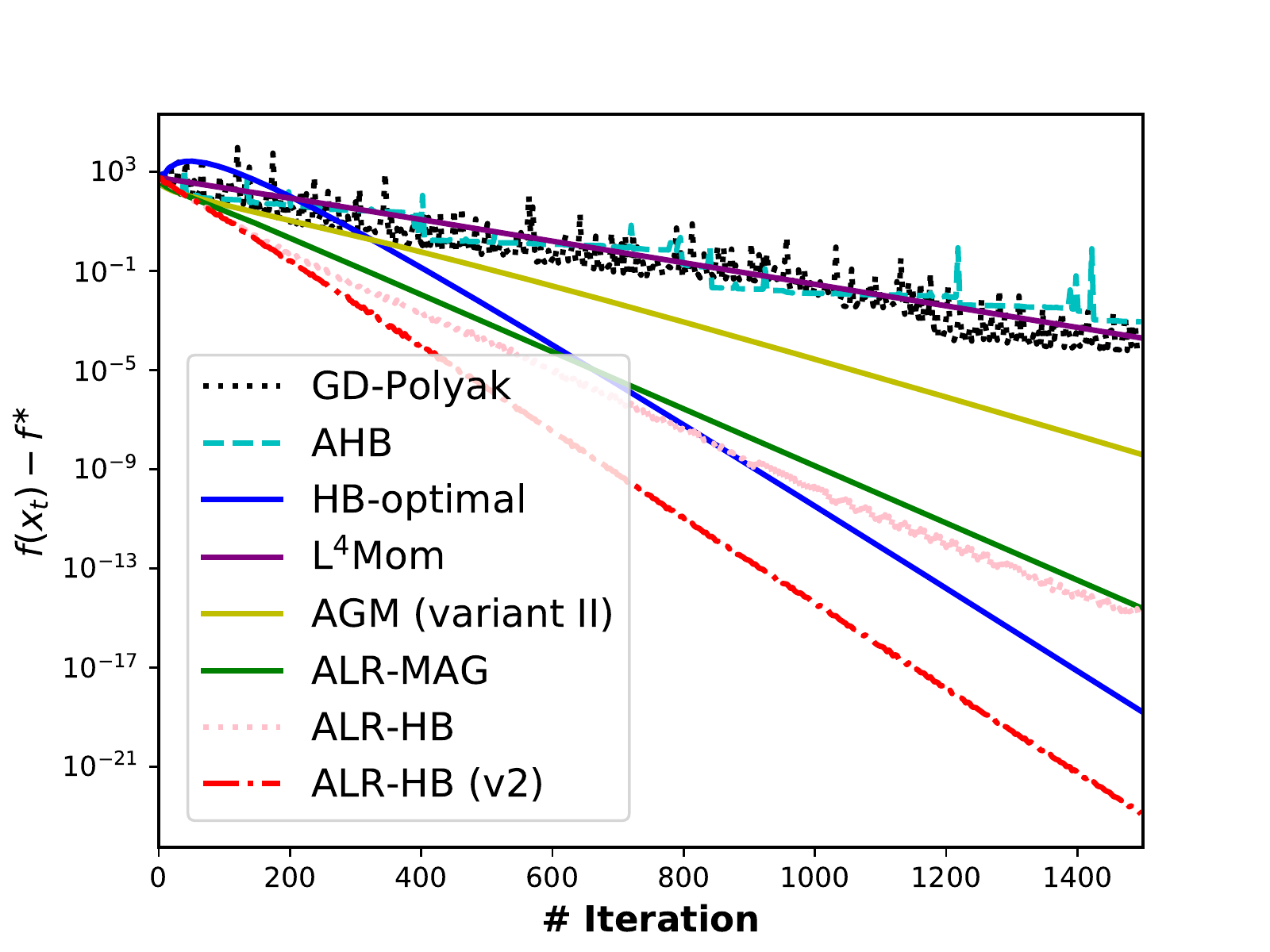}
\hfill 
\includegraphics[width=0.235\textwidth,height=1.5in]{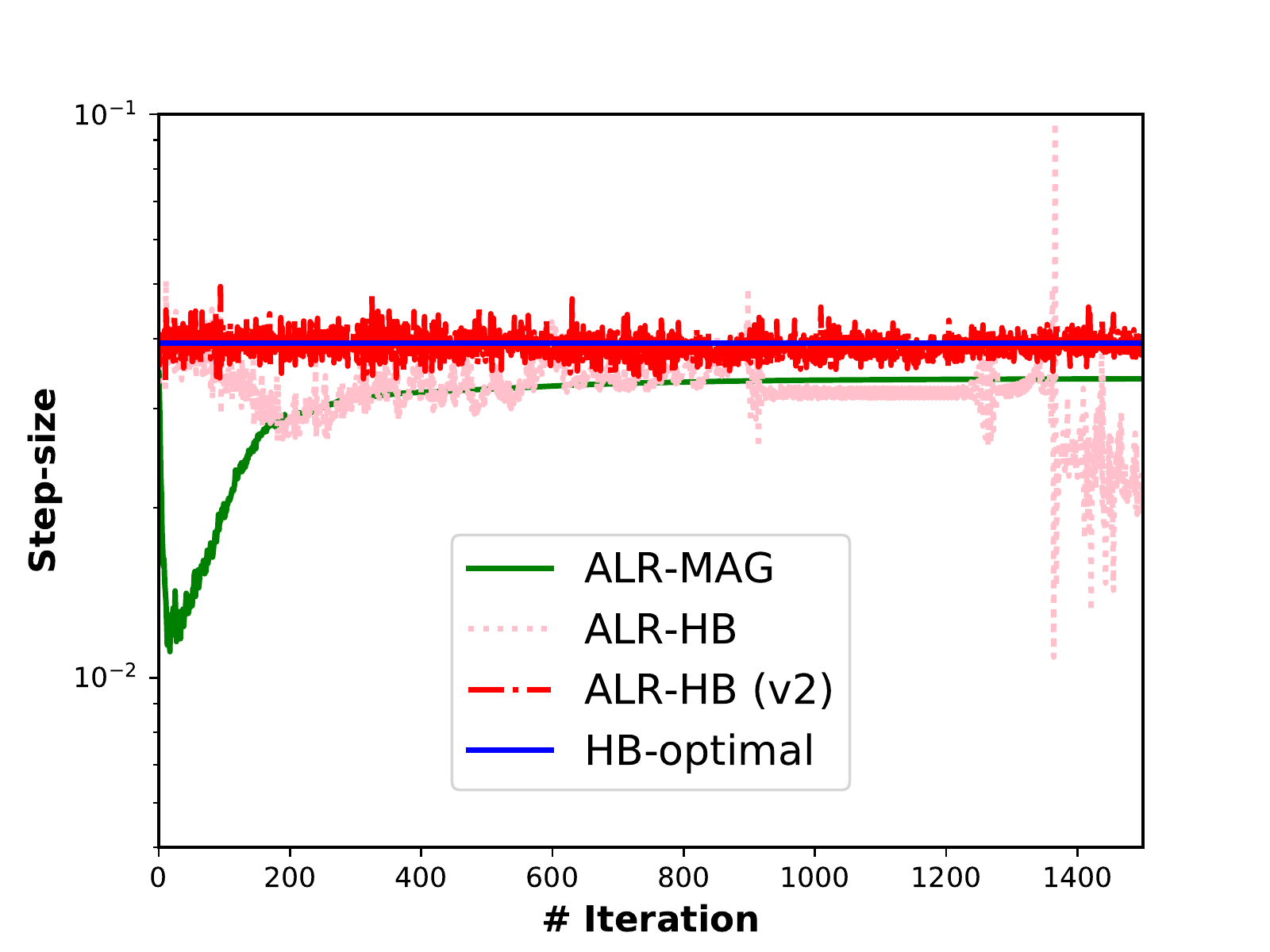}
 \vskip -0.1in
\captionsetup{justification=centering}
         \caption{Least-squares with knowledge of $\mu$ and $L$ (left: sub-optimality; right: step size)}
         \label{fig:ls:with}
         \vskip -0.1in
     \end{figure}
     \hfill
     \begin{figure}[ht]
     \vskip -0.1in       \includegraphics[width=0.235\textwidth,height=1.5in]{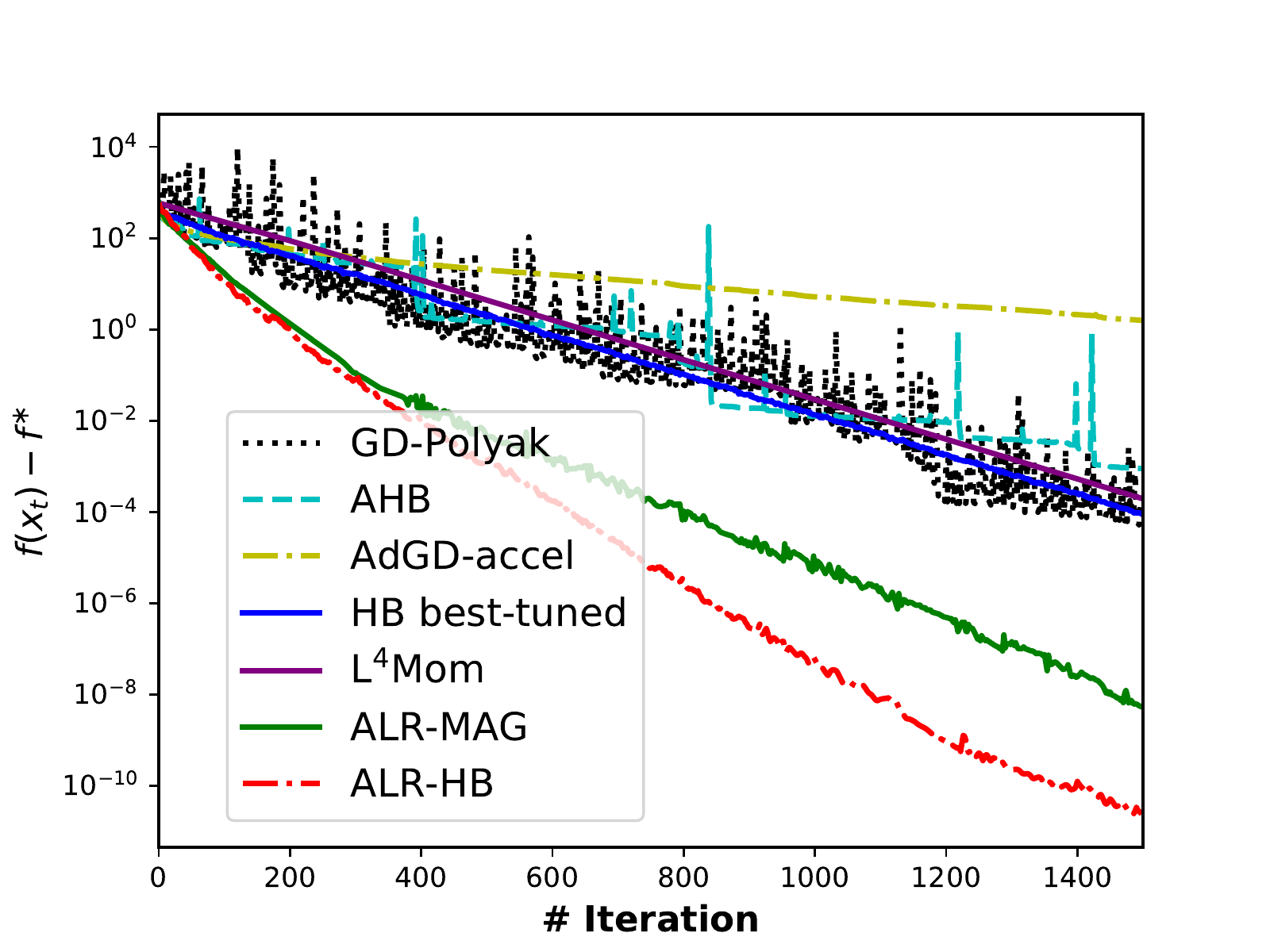}
         \hfill
\includegraphics[width=0.235\textwidth,height=1.5in]{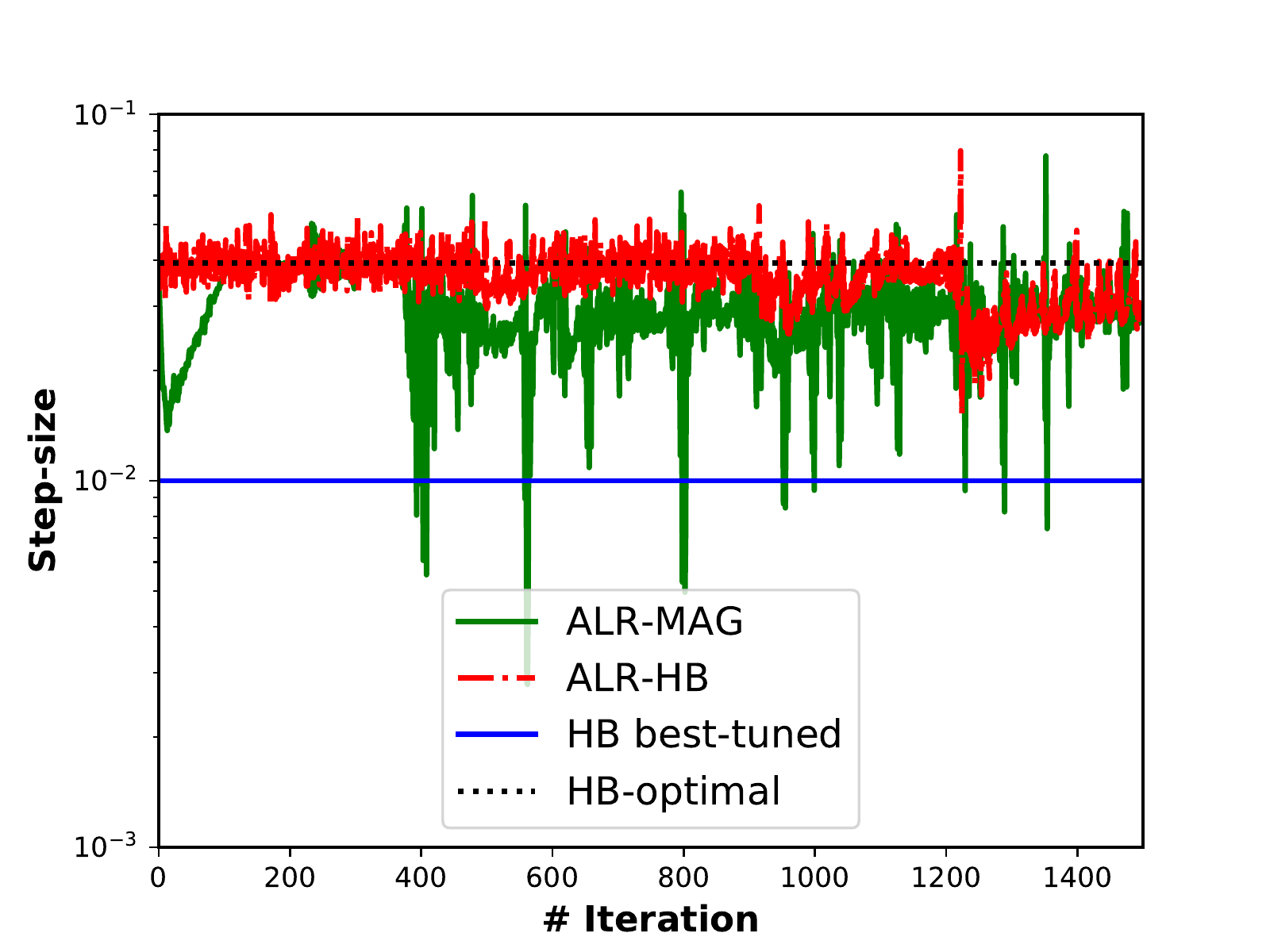}
\captionsetup{justification=centering} 
       \caption{Least-squares without knowledge of $\mu$ and $L$ (left: sub-optimality; right: step size)}
\label{fig:ls:without}
\vskip -0.1in
     \end{figure}

\begin{figure}[t]
\vskip -0.1in
\begin{center}
\includegraphics[width=0.4\textwidth,height=2in]{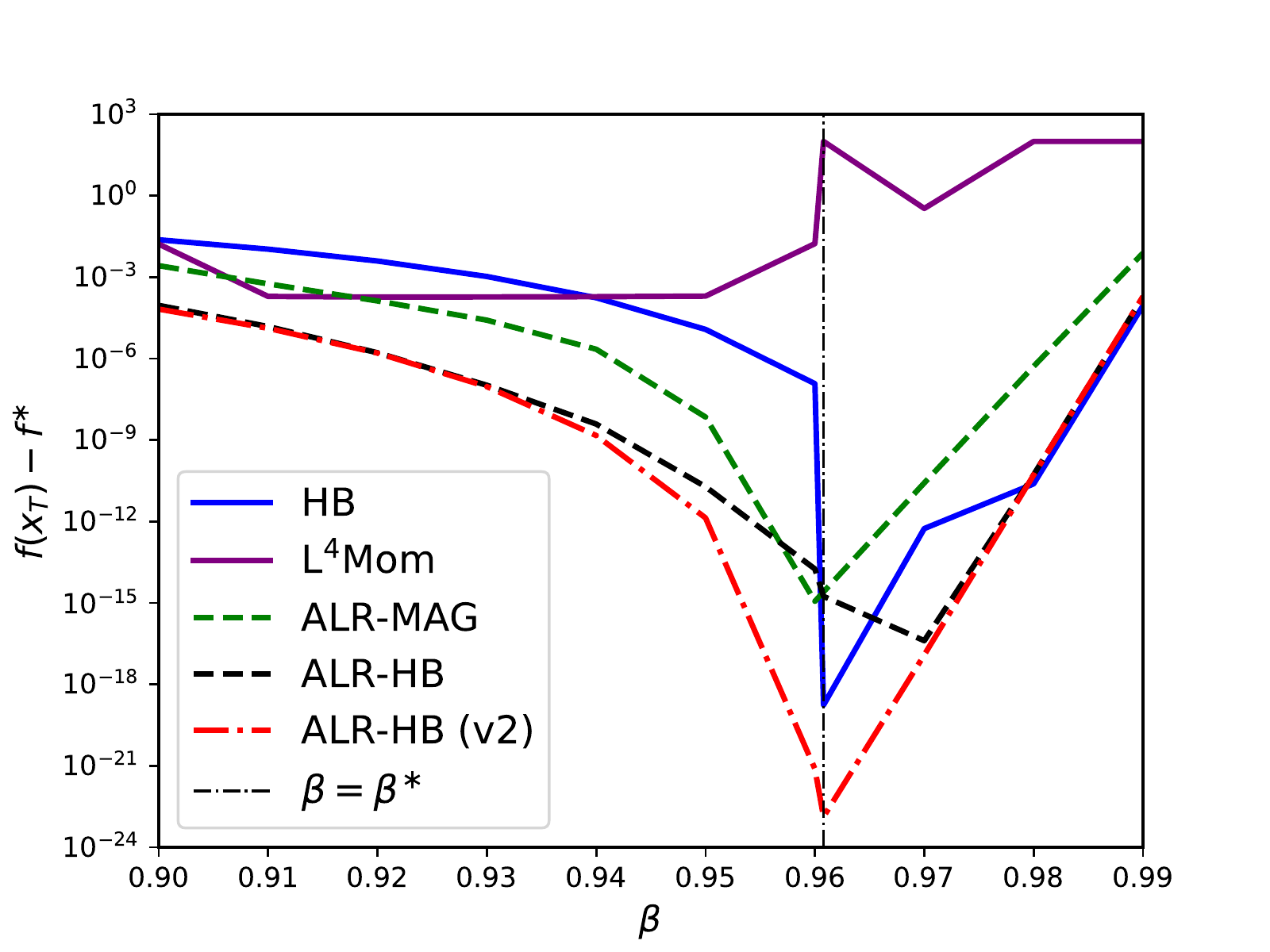}
\vskip -0.1in
\caption{Least-squares under different $\beta$ }
\label{fig:ls:beta}
\end{center}
\end{figure}

\subsection{Experimental Results on Deep Neural Networks}\label{sec:numerical:dnn}

To show the practical implications of ALR-SMAG and ALR-SHB, we conduct experiments with deep neural network training on the CIFAR~\citep{krizhevsky2009learning} and Tiny-ImageNet200~\citep{le2015tiny} datasets. We compare ALR-SMAG and ALR-SHB against SGD with momentum (SGDM) under: constant step sizes; step-decay~\citep{ge2019step,Wang-Step-Decay}, where the step size is divided by 10 after the same number of iterations, and cosine decay without restart~\citep{loshchilov2016sgdr};  the adaptive step size methods  SPS\_${\max}$~\citep{SPS} and SLS with acceleration (SLS-acc)~\citep{SLS};  L$^4$Mom~\citep{L4}; AdSGD~\citep{AdSGD}; and Adam~\citep{Adam}. To eliminate the influence of randomness, we repeat the experiments 5 times with different seeds and report the averaged results. The over-parameterized deep neural networks satisfy interpolation~\citep{zhang_interpolation}. In all Polyak-based algorithms, we use $f_{S_k}^{\ast}=0$ throughout. 

\subsubsection{Results on CIFAR10 and CIFAR100}
We consider the benchmark experiments for CIFAR10 and CIFAR100 with two standard image-classification architectures: 28$\times$10 wide residual network (WRN)~\citep{zagoruyko2016wide} and  DenseNet121~\citep{huang2017densely}, without implementation of weight-decay. The maximum epochs call is 200 and the batch size is 128. For the space concern, the details of the parameters are shown in Appendix~\ref{append:cifar}. The results on CIFAR10 and Tiny-ImageNet are presented in Appendix~\ref{append:cifar} and ~\ref{append:tinyimagenet}, respectively.

First, we report the results of CIFAR100 on WRN-28-10 and DenseNet121 in Figure \ref{fig:lr:cifar100} and Table \ref{tab:cifar100:polyak-grad}. From Figure \ref{fig:lr:cifar100}, we observe that ALR-SMAG and ALR-SHB result in the best training loss and achieves the highest accuracy. Table \ref{tab:cifar100:polyak-grad} shows that our algorithms ALR-SMAG and ALR-SHB perform better than the adaptive step size methods SPS\_${\max}$, L$^4$Mom, SLS-acc and AdSGD, and are comparable to SGDM with step-decay step size (denoted by SGDM-step). Note that L$^4$Mom failed in one run of the experiment but we still report the averaged results from the 4 successful runs. 

In this experiment, we borrow the idea of warmup from~\citep{vaswani2017attention} to update the upper bound  $\eta_{\max}$ of ALR-SHB and ALR-SMAG as $\eta_{\max} = \eta_0 \min(10^{-4} k, 1)$. The warmup heuristic has been used to mitigate the issue of converging to bad local minima for many optimization methods. The averaged result of test accuracy for each algorithm is reported in Table \ref{tab:cifar100:polyak-grad:wp}. 
When we incorporate the warmup (WP) technique for the maximal learning rate, our algorithms outperform SGDM with step-decay.
The performance of ALR-SMAG shown in Figure~\ref{fig:mag:c} is insensitive to the hyper-parameter $c$, see Appendix~\ref{append:cifar}.

 \begin{figure}[t]
 \vskip -0.1in
\begin{center}
\begin{subfigure}[b]{0.4\textwidth}
\includegraphics[width=\textwidth, height=2in]{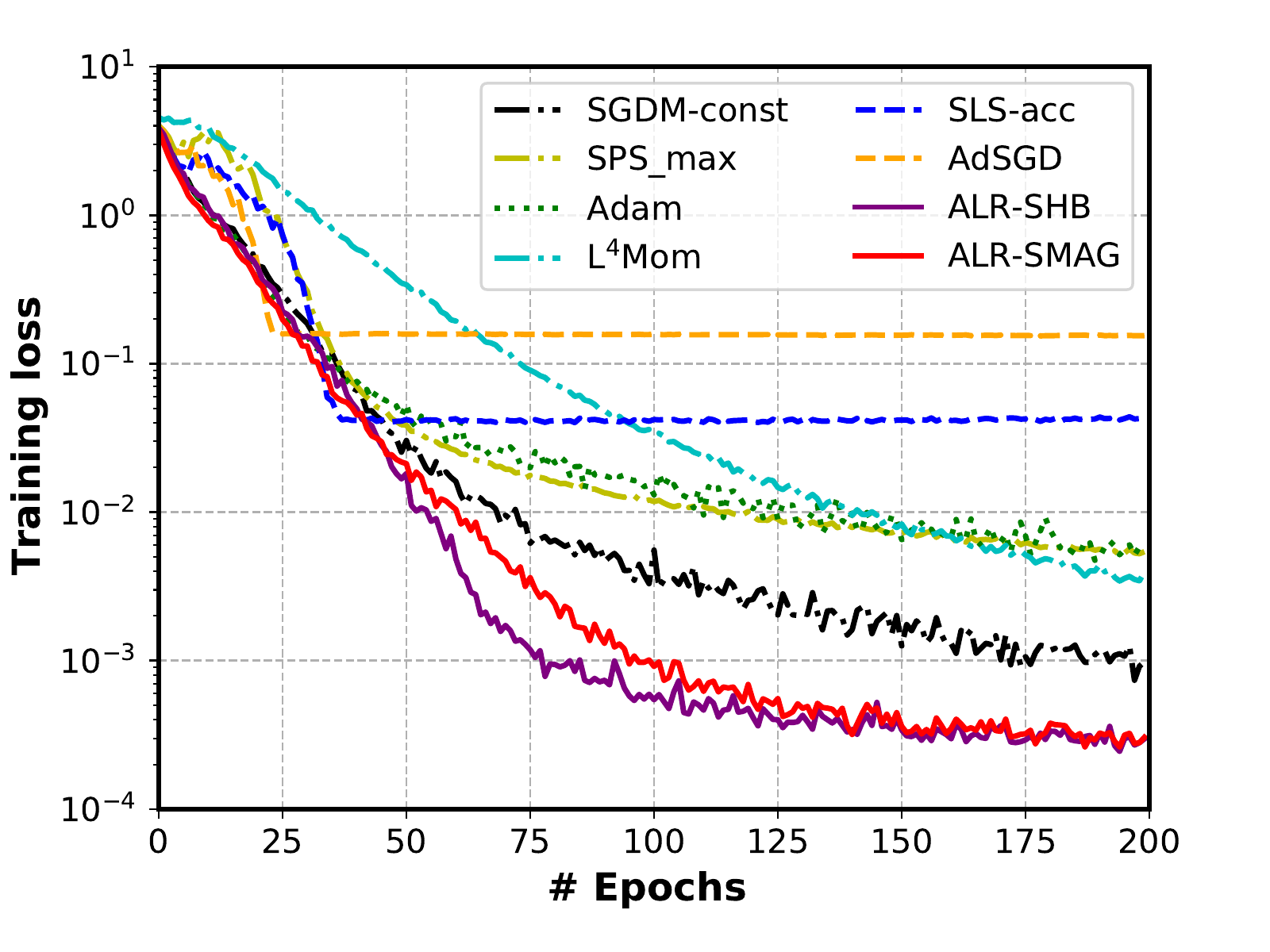}
\vskip -0.1in
\label{fig:cifar100:train:loss}
     \end{subfigure}
      \hfill
     \begin{subfigure}[b]{0.4\textwidth}
\includegraphics[width=\textwidth, height=2in]{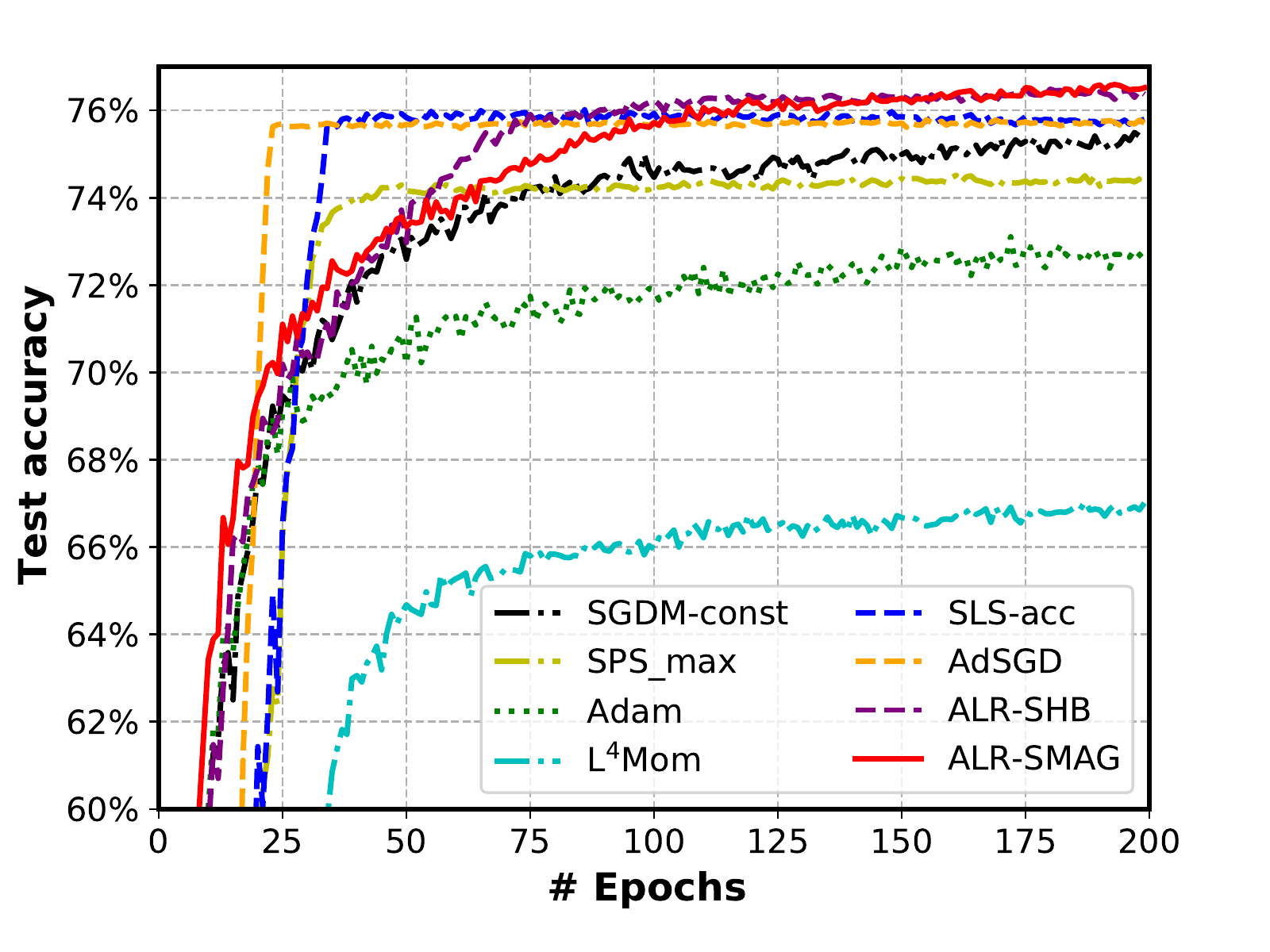} \vskip -0.1in
\label{fig:cifar100:accu}
     \end{subfigure}
        \caption{CIFAR100 - WRN-28-10: training loss (left) and test accuracy (right)}
\label{fig:lr:cifar100}
 \end{center}
 \vskip -0.1in
\end{figure}

\begin{table}[t] 
\noindent\caption{The results of CIFAR100 on WRN and DenseNet}
\label{tab:cifar100:polyak-grad} 
\begin{center}
\begin{small}
\begin{sc}
\begin{tabular}{lcc}
\toprule  \cline{1-3}
 \multirow{2}{*}{Method}  & WRN-28-10 & DenseNet121\\ 
 & \multicolumn{2}{c}{Test accuracy (\%)}  \\\midrule
 SGDM-const &  75.36 $\pm$ 0.28 & 74.20 $\pm$ 1.56 \\ 
ADAM & 72.70 $\pm$ 0.19 & 72.08 $\pm$ 0.15\\ 
 L4Mom & 66.84 $\pm$  0.60 & 67.55 $\pm$ 0.32 \\ 
   SPS\_${\max}$ &  74.39 $\pm$
   0.50 &  73.97 $\pm$ 0.87\\ 
  SLS-acc &  75.74 $\pm$  0.19  & 74.81 $\pm$ 0.26 \\ 
AdSGD &  75.71 $\pm$ 0.29  & 74.72 $\pm$ 0.40\\ 
  ALR-SHB  & 76.36 $\pm$ 0.15 & 74.50 $\pm$ 0.95 \\ 
    ALR-SMAG &  \textbf{76.51 $\pm$ 0.32}  & \textbf{75.25 $\pm$ 0.49} \\
    SGDM-step  & 76.49 $\pm$ 0.37 & 75.12 $\pm$ 0.32 \\ 
 \bottomrule
\end{tabular}
\end{sc}
\end{small}
\end{center} 
\vskip -0.1in
\end{table}

\begin{table}[h] 
\noindent\caption{Results of different step size policies  under warmup}
\label{tab:cifar100:polyak-grad:wp} 
\begin{center}
\begin{small}
\begin{sc}
\begin{tabular}{lcc}
\toprule  \cline{1-3}
 \multirow{2}{*}{Method}  & WRN-28-10 & DenseNet121\\ 
 & \multicolumn{2}{c}{Test accuracy (\%)}  \\\midrule
 SGDM-const + WP &  75.77 $\pm$ 0.48 & 74.4 $\pm$ 0.47 \\ 
    ALR-SHB + WP & { 77.57 $\pm$ 0.42}  & { 77.03 $\pm$ 0.35} \\ 
 ALR-SMAG + WP  & {\bf 77.63 $\pm$ 0.21 } & \bf{77.24 $\pm$ 0.16} \\ 
    SGDM-step + WP  & 77.27 $\pm$ 0.26 & 76.89 $\pm$ 0.28 \\ 
 \bottomrule
\end{tabular}
\end{sc}
\end{small}
\end{center} 
\vskip -0.1in
\end{table}

\subsection{Enabling Weight-Decay to Improve Generalization}
\label{sec:wd}

In neural network training, it is often desirable to incorporate weight-decay ($\ell_2$-regularization) to improve generalization. {It is, therefore, important to make our step sizes efficient also in this setting.} 
However, the typical way of adding $\ell_2$ regularization ($f + \frac{\lambda}{2}\left\|\cdot \right\|^2$) to the objective function is not applicable for Polyak-based algorithms because the corresponding $f_{S_k}^{\ast}$ is often inaccessible or expensive to compute. {ALI-G~\citep{ALIG} incorporates regularization as a constraint on the feasible domain. However, promoting regularization as a constraint does not work well for our step sizes.}
Instead, we use a similar idea as \citet{loshchilov2016sgdr} and decouple the loss and regularization terms. In ALR-SMAG, this is done by adding $\lambda x_k$ to the updated direction $d_{k}$ and use the search direction $d_k+\lambda x_k$; see Algorithm \ref{alg:mag:wd} of Appendix~\ref{append:wd} where $\lambda>0$ is the parameter of weight-decay.  In this way, we still set $f_{S_k}^{\ast}$ to be zero because nothing changes in the networks.

We test the performance of ALR-SMAG with \emph{weight-decay} on CIFAR100 with WRN-28-10 and compare with other state-of-the-art algorithms: AdamW under step-decay step size (denoted by AdamW-step) with $\lambda=0.0001$; SGDM under warmup (SGDM + WP), step-decay (SGDM-step), and cosine (SGDM-cosine) step sizes with $\lambda=0.0005$; and ALI-G~\citep{ALIG} with and without Nesterov momentum. For ALR-SMAG with weight-decay, we set $\lambda=0.0005$ and $c=0.3$. We train for 200 epochs and use batch size  128. More details are given in Appendix~\ref{append:wd}.

The results are shown in Table~\ref{tab:cifar100:polyak:wd}. In addition to the best test accuracy, we also record the results at 60, 120, and 180 epochs. We observe that ALR-SMAG is able to reach a relatively high accuracy at 120 epochs. But after 120 epochs, the training process is basically saturated and the accuracy does not improve much (it even drops a little). Since $c$ is less than 1, the step size of ALR-SMAG is still aggressive.
As a result, the iteration oscillates locally, and it is difficult to converge to a certain point. In order to converge, in the final training phase (last 20\% of iterations) of ALR-SMAG, we introduce a fine-tuning phase that increases $c$ exponentially. 
 From the last row in Table~\ref{tab:cifar100:polyak:wd} we see that fine-tuning (FT) does improve the solution accuracy. 

\begin{table}[h] 
\vskip -0.1in
\noindent\caption{CIFAR100 - WRN-28-10 with weight-decay (WD)}
\label{tab:cifar100:polyak:wd}
\vspace{-0.1in}
\begin{center}
\begin{small}
\begin{sc}
\hspace{-0.1in}
\begin{tabular}{lcccl}
\toprule  \cline{1-5}
\hspace{-0.1in} \multirow{2}{*}{Method}  & \multicolumn{4}{c}{Test accuracy (\%)} \\
&  \#60 & \#120 & \#180  & Best \\ \midrule
\hspace{-0.1in} AdamW-step  &   70.60  & 75.28  &  75.93 & 76.19 $\pm$ 0.13\\ 
\hspace{-0.1in} ALI-G   &  {\bf 72.27}  & 72.83 & 73.12 & 73.40 $\pm$ 0.22\\
\hspace{-0.1in} ALI-G + Mom   & { 67.06}   & 78.99 & 79.88 & 80.21 $\pm$ 0.14\\ 
\hspace{-0.1in} SGDM + WP   & 69.88 & 72.02 & 72.32 & 73.47 $\pm$ 0.40\\ 
\hspace{-0.1in} SGDM-step   & 60.15 & 75.38 & {80.91} & 81.22 $\pm$ 0.16\\
\hspace{-0.1in} SGDM-cosine   & { 64.52} & 71.17 & {81.42} & 81.85 $\pm$ 0.19\\
 \hspace{-0.1in}  ALR-SMAG  &    63.64  &  {\bf 80.20} & 80.69 & 81.18 $\pm$ 0.28\\ 
\hspace{-0.1in}    ALR-SMAG +  FT & 63.64 & {\bf 80.20}  & {\bf 81.64} & {\bf 81.89} $\pm$ 0.20\\
\bottomrule
\end{tabular}
\end{sc}
\end{small}
\end{center} 
\vskip -0.2in
\end{table}


\section{Conclusion}
We proposed a novel approach for generalizing the popular Polyak step size to first-order methods with momentum. The resulting algorithms are significantly better than the original heavy-ball method and gradient descent with Polyak step size if the condition number is inaccessible. We demonstrated our methods are less sensitive to the choice of $\beta$ than the original heavy-ball method and may avoid the instability of heavy-ball on the ill-conditioned problems. Furthermore, we extended our step sizes to the stochastic settings and demonstrated superior performance in logistic regression and deep neural network training compared to the state-of-the-art adaptive methods. In Appendix~\ref{sec:model:nag}, we extend our framework to Nesterov accelerated gradient (NAG) and provide preliminary experiments on least-squares problems.  It will be interesting to study how this step size performs on a wider range of applications. Another interesting extension would be to develop techniques for adjusting the learning rate in second-order adaptive gradient methods, such as  AdaGrad~\citep{AdaGrad} and Adam~\citep{Adam} which are widely used in deep learning.




\section*{Acknowledgements}

This work was supported partially by  
GRF 16310222, GRF 16201320, and VR 2019-05319.
Xiaoyu Wang is supported by Innovation and Technology Commission of Hong Kong China under the project PRP/074/19FX.

\bibliography{spolyak}
\bibliographystyle{icml2023}

\newpage
\appendix
\onecolumn


\section{Application to Nesterov Accelerated Gradient (NAG)}\label{sec:model:nag}

In Sections \ref{sec:model:hb} and \ref{sec:model:mag}, we have shown how our framework applies to the heavy-ball and moving averaged gradient methods. However, the magic does not stop there. For instance, we further incorporate Nesterov accelerated gradient (NAG)~\citep{nesterov1983} 
\begin{align*}
v_{k+1} & = \beta v_k - \eta_k \nabla f(x_k + \beta v_k) \notag \\
x_{k+1} & = x_k + v_{k+1} 
\end{align*}
into the proposed framework in Section~\ref{sec:model}. In (\ref{unified:mom}), let $d_k = \nabla f(x_k + \beta (x_k-x_{k-1}))$ and $\gamma = \beta$, then it reduces to the NAG algorithm. By using the convexity of $f$ at $x_k + \beta v_k$, that is $ \left\langle \nabla f(x_k + \beta v_k),  x_k + \beta v_k -  x^{\ast} \right\rangle \geq f(x_k + \beta v_k) - f^{\ast} $, 
we optimize the upper bound of $\left\|x_{k+1} - x^{\ast} \right\|^2$ with respect to the step size variable $\eta_k$, it results in the adaptive step size for NAG:
\begin{align}\label{inequ:mad:nag}
\eta_k = \frac{f(x_k + \beta v_k) - f^{\ast} }{\left\|\nabla f(x_k + \beta v_k)\right\|^2}.
\end{align}
We refer to the NAG algorithm with the step size (\ref{inequ:mad:nag}) as ALR-NAG. One intuitive interpretation behind the ALR-NAG algorithm is that first, you move a momentum step $\beta(x_k-x_{k-1})$ at the current point $x_k$, and then you stand at this new point $\tilde{x}_k=x_k + \beta(x_k -x_{k-1})$ and perform the Polyak step size along $-\nabla f(\tilde{x}_k)$. 

Next, we conduct preliminary experiments to test ALR-NAG on a least-squares problem where the condition number $\kappa=10^4$. The first interesting result on the least-squares problem shows that ALR-NAG can obtain a more accurate solution than the original NAG~\citep{nesterov1983} under optimal parameters and the accelerated gradient method (AGM)~\citep{AGM}. \citet{AGM} evaluate the strong convexity
constant $\hat{\mu}$ by the inverse of the Polyak step size and set the momentum parameter $\beta=(\sqrt{L}-\sqrt{\hat{\mu}})/(\sqrt{L}+\sqrt{\hat{\mu}})$ for Nesterov momentum with the knowledge of the smoothness parameter $L$. Besides, for each $\beta \in (0,1)$, ALR-NAG automatically adjusts the step size which makes it less sensitive to $\beta$ than NAG with $1/L$ step size where $L$ is the largest eigenvalue of the least-squares problem. Another observation is that the optimal parameter $\beta$ for ALR-NAG is not consistent with the original NAG of which the theoretical optimal momentum parameter is $\beta^{\ast}=(\sqrt{\kappa}-1)/(\sqrt{\kappa}+1)$~\citep{nesterov1983} where $\kappa$ is the condition number of the problem.
\begin{figure}[ht]
\begin{center}
\includegraphics[width=0.4\textwidth]{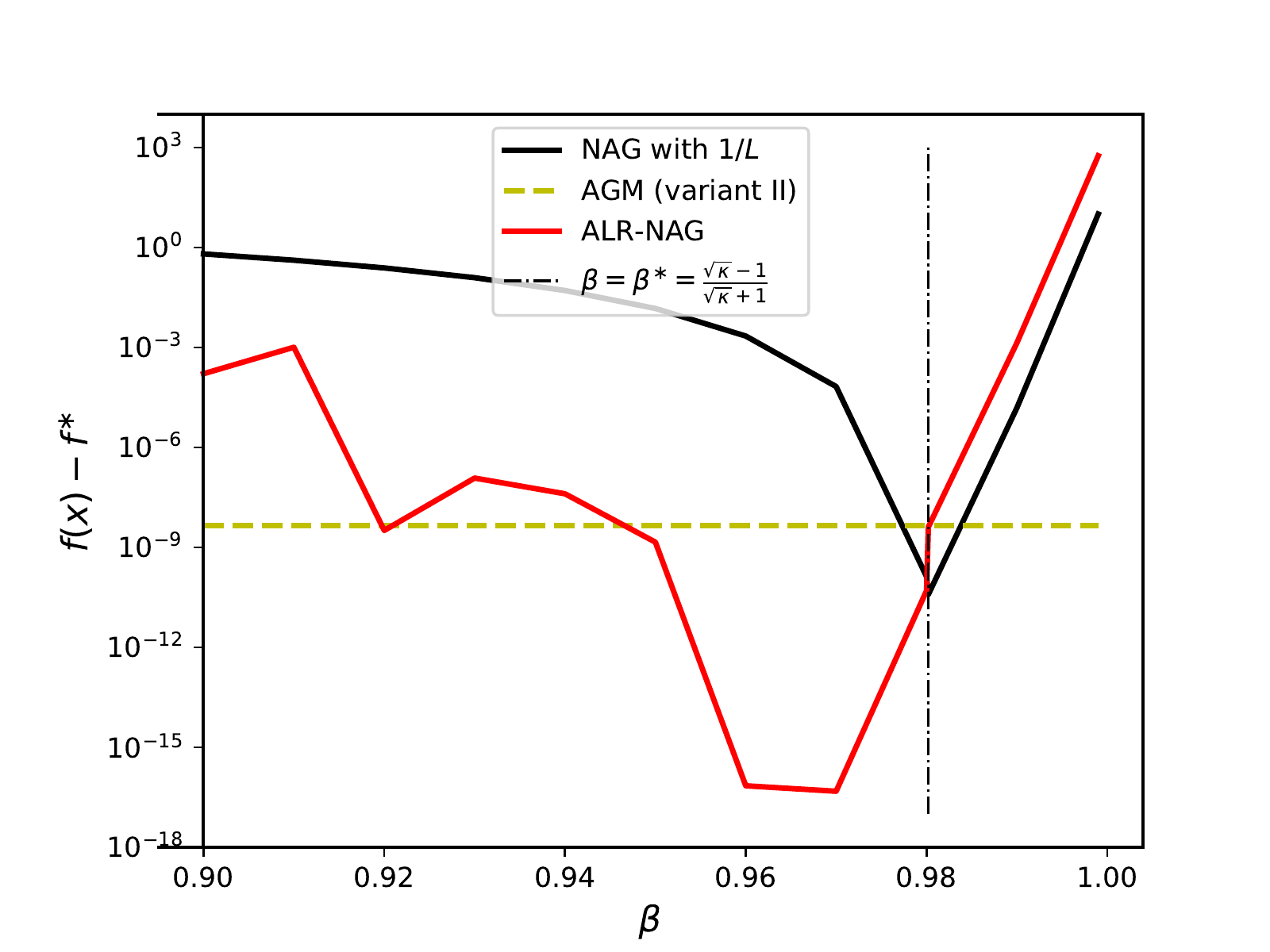}
\caption{Results of NAG and ALR\_NAG with different $\beta$ on least-squares}
\label{fig:ls:nag:beta}  
\end{center}
\end{figure}

\section{Proofs of Section \ref{sec:theory}}\label{appendix:determinstic}
In this section, before describing the details of the proofs, we first provide the basic definitions that omit in the main content.
\begin{definition}(Strongly convex)
We say that a function $f$ is $\mu$-strongly convex on $\R^d$ if $
f(y) \geq f(x) + \left\langle g, y-x\right\rangle + \frac{\mu}{2}\left\|y-x\right\|^2, \, \forall  x, y\in \R^d, g \in \partial f(x),$
with $\mu > 0$.
\end{definition}
\begin{definition}(Convex)
The function $f$ is convex on $\R^d$ if $f(y) \geq f(x) + \left\langle g, y-x \right\rangle $  for any $g \in \partial f(x)$ and $x,y\in \R^d$.
\end{definition}

\begin{definition}(Quasar convexity~\citep{gower2021sgd})
  Let $\zeta \in (0,1]$ and $x^{\ast} \in \mathcal{X}^{\ast}$. A function $f$ is $\zeta$-quasar convex with respect to $x^{\ast}$ if $ f^{\ast} \geq f(x) + \frac{1}{\zeta}\left\langle g, x^{\ast} - x \right\rangle $ for any $x \in \R^d$ and $g \in \partial f(x)$.
\end{definition}
In general, a $\zeta$-quasar convex function $f$ does not need to be convex. The parameter $\zeta$ controls the non-convexity of the function. If $\zeta = 1$, the quasar convex is reduced to the well-known star convexity \citep{nesterov2006cubic}, which is a generalization of convexity. For example, $f(x) = (x^2+\frac{1}{4})^{\frac{1}{4}}$ is quasar-convex with $\zeta=1/2$. Learning linear dynamical systems is the practical example of a quasar-convex function and is nonconvex~\citep{hardt2018gradient}.

In Section~\ref{sec:theory}, we have provided the definition of the semi-strongly convex functions. Here we give one simple example to clarify that the semi-strongly convex function is not necessarily strongly convex. For example, $f(x)=x^2 + 2\sin(x)^2$ is semi-strongly convex with $\hat{\mu}=2$, while the second-order derivative $\nabla^2 f(x)$ can be negative.

\begin{definition}($L$-smooth)
When the function $f$ is differentiable on ${\R^d}$, we say that $f$ is $L$-smooth on 
$\R^d$ if there exists a constant $L > 0$ such that $\left\|\nabla f(x) - \nabla f(y) \right\| \leq L\left\|x-y\right\|$. 
This also implies that $f(y) \leq f(x) + \left\langle \nabla f(x), y-x \right\rangle + \frac{L}{2}\left\|x-y\right\|^2$ for any $x, y \in \R^d$.
\end{definition}

In this paper, we also consider a special family of non-smooth and non-strongly convex problems, whose epigraph is a polyhedron~\citep{SG_nonsmooth_yang}. 
\begin{definition}(Polyhedral convex)
 For a convex minimization problem (\ref{P1}), suppose that the epigraph of $f$ over $\mathcal{X}$ is a polyhedron.
\end{definition}
The convex polyhedral problem implies the \emph{polyhedral error condition}~\citep{SG_nonsmooth_yang}: there exists a constant $\kappa_1 >0 $ such that
\begin{align*}
    \left\|x -x^{\ast} \right\| \leq \frac{1}{\kappa_1} (f(x) - f^{\ast})
\end{align*}
for all $x \in \mathcal{X}$.
Some interesting applications for example $\ell_{1}$ and $\ell_{\infty}$-constrained or regularized piece-wise linear minimization, and a submodular function minimization are polyhedral convex~\citep{SG_nonsmooth_yang}. 
\begin{assumption}\label{assumpt:polyhed}
For the problem (\ref{P1}), we assume that 
(i) $\forall x_0 \in \R^d$, we know there exists $\delta > 0$ such that $f(x_0) - \min_{x \in \R^d}f(x) \leq \delta$; (ii) there exists a constant $G>0$ such that $\max_{g \in \partial f(x)} \left\|g\right\|^2 \leq G^2$ for any $x \in \R^d$.
\end{assumption}
The first assumption of Assumption \ref{assumpt:polyhed} implies that there is a lower bound for $f^{\ast}$, which is also made in \citep{freund2018new}. This is satisfied in most machine learning applications for which we have $f^{\ast} \geq 0$. Assumption \ref{assumpt:polyhed}(ii) is a standard assumption to be made in the non-smooth optimization~\citep{boyd2003subgradient,Shamir-Zhang2013}.

\subsection{Proofs of Theorems and Lemmas in Section \ref{sec:thm:dc}}
In this part, we provide detailed proofs for the results of ALR-MAG in deterministic optimization. 
\begin{proof}({\bf of Lemma \ref{lem:mag:convex}})

For $k = 1$, we have $\left\langle d_{k-1}, x_k -x^{\ast}\right\rangle = \left\langle d_0, x_1 -x^{\ast}\right\rangle = 0$. 
    For $k > 1$, suppose that $\left\langle d_{k-2}, x_{k-1} -x^{\ast}\right\rangle \geq 0$ holds, we have
       \begin{align*}
        \left\langle d_{k-1}, x_k -x^{\ast}\right\rangle &   = \left\langle d_{k-1}, x_k - x_{k-1} + x_{k-1} 
        -x^{\ast}\right\rangle = \left\langle d_{k-1}, -\eta_{k-1}d_{k-1} + x_{k-1} 
        -x^{\ast}\right\rangle  \notag \\
        & =  - \eta_{k-1}\left\|d_{k-1}\right\|^2 + \left\langle \nabla f(x_{k-1}) + \beta d_{k-2},  x_{k-1} 
        -x^{\ast}\right\rangle \notag \\
        & = - \eta_{k-1}\left\|d_{k-1}\right\|^2 + \left\langle \nabla f(x_{k-1}),  x_{k-1} 
        -x^{\ast}\right\rangle + \beta \left\langle d_{k-2}, x_{k-1} 
        -x^{\ast}\right\rangle \notag \\
        & \mathop{\geq}^{(a)} - \eta_{k-1}\left\|d_{k-1}\right\|^2 +  \left(f(x_{k-1}) - f(
        x^{\ast}) \right)+ \beta \left\langle d_{k-2}, x_{k-1} 
        -x^{\ast}\right\rangle \geq 0.
     \end{align*}
     where $(a)$ follows from the convexity of $f$ that $\left\langle \nabla f(x_{k-1}), x_{k-1} -x^{\ast} \right\rangle \geq f(x_{k-1}) - f(x^{\ast}) $ and $\eta_{k-1} \leq \frac{f(x_{k-1}) - f(x^{\ast})}{\left\|d_{k-1} \right\|^2}$. By induction, we claim that $\left\langle d_{k-1}, x_k -x^{\ast}\right\rangle \geq 0$ for all $k\geq 1$.
     \end{proof}

\begin{proof}({\bf of Lemma \ref{lem:mag:xxstar}})

First, we consider the general convex functions without the smoothness assumption, then
    \begin{align}\label{inequ:lem:1}
    \left\langle  d_k, x_k -  x^{\ast} \right\rangle = \left\langle \nabla f(x_k) + \beta d_{k-1}, x_k -  x^{\ast} \right\rangle \geq \left\langle  \nabla f(x_k), x_k -  x^{\ast} \right\rangle \geq \left(f(x_k) - f^{\ast} \right)
    \end{align}
   where $\left\langle d_{k-1}, x_k -x^{\ast}\right\rangle \geq 0$ holds by Lemma \ref{lem:mag:convex}.
 By applying Inequality (\ref{inequ:lem:1}), the distance $\left\|x_{k+1} - x^{\ast} \right\|^2$ can be estimated as:
\begin{align}\label{inequ:mag:x:xk}
    \left\|x_{k+1} - x^{\ast} \right\|^2 & = \left\|x_{k} - \eta_k d_k -  x^{\ast} \right\|^2 = \left\|x_{k} - x^{\ast} \right\|^2 -  2\eta_k \left\langle  d_k, x_k -  x^{\ast} \right\rangle + \eta_k^2 \left\| d_k \right\|^2  \notag \\
    & \leq \left\|x_{k} - x^{\ast} \right\|^2 -  2\frac{(f(x_k) - f(x^{\ast}))^2}{ \left\| d_k \right\|^2}  + \frac{(f(x_k) - f(x^{\ast}))^2}{ \left\| d_k \right\|^2} \notag \\
    & = \left\|x_{k} - x^{\ast} \right\|^2 -  \frac{(f(x_k) - f(x^{\ast}))^2}{ \left\| d_k  \right\|^2}   = \left\|x_{k} - x^{\ast} \right\|^2 -  \eta_k\left(f(x_k) - f(x^{\ast}) \right).
\end{align}
 As we can see, this choice of step size leads to a decrease of $\left\|x_{k+1} - x^{\ast} \right\|^2$.    

Next, if the function is also $L$-smooth, by \citep[Theorem 2.1.5]{nesterov2003}, we have
\begin{align*}
\left\langle \nabla f(x_k), x_k - x^{\ast} \right\rangle \geq f(x_k) - f^{\ast} + \frac{1}{2L}\left\|\nabla f(x_k) \right\|^2.
\end{align*}
In this case, we claim that for all $k \geq 2$
\begin{align}\label{inequ:claim:smooth}
\left\langle d_{k-1}, x_k - x^{\ast} \right\rangle  \geq \frac{1}{2L}\sum_{i=1}^{k-1}\beta^{k-1-i}\left\|\nabla f(x_{i}) \right\|^2. 
\end{align}
When $k=2$, we can see that 
\begin{align*}
    \left\langle d_{1}, x_2 - x^{\ast} \right\rangle & =  \left\langle \nabla f(x_1), -\eta_1 \nabla f(x_1) + x_1 - x^{\ast} \right\rangle = - \eta_1 \left\|\nabla f(x_1) \right\|^2 +  \left\langle \nabla f(x_1), x_1 - x^{\ast} \right\rangle \notag \\
    & \geq  - \left( f(x_1) - f^{\ast} \right) + \left\langle \nabla f(x_1), x_1 - x^{\ast} \right\rangle \geq  \frac{1}{2L}\left\|\nabla f(x_{1}) \right\|^2.
\end{align*} 
Then the claim (\ref{inequ:claim:smooth}) holds at $k=2$. 
For $k > 2$, if $\left\langle d_{k-2}, x_{k-1} - x^{\ast} \right\rangle \geq \frac{1}{2L}\sum_{i=1}^{k-2}\beta^{k-2-i}\left\|\nabla f(x_{i}) \right\|^2$, we have
\begin{align*}
\left\langle d_{k-1}, x_k - x^{\ast} \right\rangle  & = \left\langle d_{k-1}, x_k - x_{k-1} + x_{k-1} - x^{\ast} \right\rangle = - \eta_{k-1} \left\| d_{k-1} \right\|^2 + \left\langle d_{k-1}, x_{k-1} - x^{\ast} \right\rangle \notag \\
& = - \eta_{k-1} \left\| d_{k-1} \right\|^2 + \left\langle \nabla f(x_{k-1}), x_{k-1} - x^{\ast} \right\rangle + \left\langle \beta d_{k-2} , x_{k-1} - x^{\ast} \right\rangle \notag \\
& = - \frac{f(x_{k-1}) - f^{\ast}}{\left\|d_{k-1}\right\|^2} \left\| d_{k-1} \right\|^2 + \left(f(x_{k-1}) - f^{\ast} + \frac{1}{2L}\left\|\nabla f(x_{k-1}) \right\|^2 \right) + \left\langle \beta d_{k-2} , x_{k-1} - x^{\ast} \right\rangle \notag \\
& \geq \frac{1}{2L}\left\|\nabla f(x_{k-1}) \right\|^2 + \beta \frac{1}{2L}\sum_{i=1}^{k-2}\beta^{k-2-i}\left\|\nabla f(x_{i}) \right\|^2 = \frac{1}{2L}\sum_{i=1}^{k-1}\beta^{k-1-i}\left\|\nabla f(x_{i}) \right\|^2.
\end{align*}
By induction, the claim (\ref{inequ:claim:smooth}) is correct for all $k \geq 2$. Then applying this claim (\ref{inequ:claim:smooth}), we can get that 
\begin{align*}
\left\langle  d_k, x_k -  x^{\ast} \right\rangle & = \left\langle \nabla f(x_k) + \beta d_{k-1}, x_k -  x^{\ast} \right\rangle \notag \\
& \geq \left(f(x_k) - f^{\ast} + \frac{1}{2L}\left\| \nabla f(x_k) \right\|^2\right) + \beta \frac{1}{2L} \sum_{i=1}^{k-1}\beta^{k-1-i} \left\|\nabla f(x_i) \right\|^2 \notag \\
& = \left(f(x_k) - f^{\ast}\right)+ \frac{1}{2L} \sum_{i=1}^{k}\beta^{k-i} \left\| \nabla f(x_i) \right\|^2.
\end{align*}
The distance $\left\|x_{k+1} - x^{\ast} \right\|^2$ can be evaluated as 
\begin{align*}
  \left\|x_{k+1} - x^{\ast} \right\|^2 & = \left\|x_{k} - \eta_k d_k -  x^{\ast} \right\|^2 = \left\|x_{k} - x^{\ast} \right\|^2 -  2\eta_k \left\langle  d_k, x_k -  x^{\ast} \right\rangle + \eta_k^2 \left\| d_k \right\|^2  \notag \\
  & \leq \left\|x_{k} - x^{\ast} \right\|^2 - \eta_k \left(f(x_k) - f^{\ast} \right) - 2 \eta_k \frac{1}{2L} \sum_{i=1}^{k}\beta^{k-i} \left\| \nabla f(x_i) \right\|^2 \notag \\
  & \leq \left\|x_{k} - x^{\ast} \right\|^2 - \eta_k \left(f(x_k) - f^{\ast} \right) - \frac{1}{L}\frac{f(x_k) - f^{\ast}}{\left\|d_k \right\|^2} \sum_{i=1}^{k}\beta^{k-i} \left\| \nabla f(x_i) \right\|^2 \notag \\
  & \mathop{\leq}^{(a)} \left\|x_{k} - x^{\ast} \right\|^2 - \eta_k \left(f(x_k) - f^{\ast} \right) - \frac{(1-\beta)}{L}\left(f(x_k) - f^{\ast} \right)
\end{align*}
where $(a)$ follows from the fact that
\begin{align}\label{dk:nablaf}
        \left\|d_k \right\|^2 &  = \left\| \beta d_{k-1} + \nabla f(x_k) \right\|^2 = \beta^2\left\|d_{k-1} \right\|^2 + \left\|\nabla f(x_k) \right\|^2 + 2\beta \left\langle d_{k-1}, \nabla f(x_k) \right\rangle  \notag \\
        & \mathop{\leq}^{(a)} \beta^2\left\|d_{k-1} \right\|^2 + \left\|\nabla f(x_k) \right\|^2 + \beta \left( \tau \left\|d_{k-1} \right\|^2 + \frac{1}{\tau}\left\| \nabla f(x_k) \right\|^2 \right) \notag \\
        & = \beta \left\|d_{k-1} \right\|^2 + \frac{1}{1-\beta}\left\|\nabla f(x_k) \right\|^2 \mathop{\leq}^{(b)}  \frac{1}{(1-\beta)}\sum_{i=1}^{k} \beta^{k-i} \left\|\nabla f(x_i) \right\|^2
    \end{align}
    where $(a)$ uses the Cauchy-Schwarz inequality and  we let $\tau = 1-\beta$ and $(b)$ follows from the induction that $\left\|d_i\right\|^2 \leq \beta \left\|d_{i-1}\right\|^2 + \frac{1}{1-\beta}\left\|\nabla f(x_i) \right\|^2$ for all $i=1,\cdots, k$ with $d_0=0$. Then the proof is complete.
\end{proof}  

\begin{theorem}({\bf ALR-MAG on non-smooth problems})\label{thm:mag:convex}
Consider the iterative scheme of MAG defined by (\ref{alg:mag}) and the step size is selected by (\ref{mag:lr}),
we derive the convergence guarantees for MAG in the following cases:
\begin{itemize}
    \item Suppose that the function $f$ is convex and its gradient is bounded, then $f(x_k) - f(x^{\ast}) \rightarrow 0$ ($k \rightarrow \infty)$.
        \item If the objective function $f$ is convex and its gradient (or subgradient) is bounded by $G^2$ (i.e. $\left\|\partial f(x) \right\|^2 \leq G^2$), we get that $  f(\hat{x}_k) - f^{\ast} \leq  \frac{G\left\|x_1 - x^{\ast} \right\|}{(1-\beta)\sqrt{k}}$ where $\hat{x}_k = \frac{1}{k} \sum_{i=1}^{k}x_i$.
    \item If the function $f$ is $\mu$-strongly convex and its gradient is bounded by $G^2$, then $\left\|x_{k} - x^{\ast} \right\|^2  \leq \frac{4G^2}{(1-\beta)^2\mu^2} \frac{1}{k}$.
  \item If the function is a polyhedral convex on $\R^d$ with $\kappa_1 >0$ and satisfies Assumption \ref{assumpt:polyhed}, then $\left\|x_k -x^{\ast}\right\|^2 $ promotes linear convergence with a rate at least $1- \frac{(1-\beta)^2\kappa_1^2}{G^2}$.
\end{itemize}
\end{theorem}

\begin{proof}({\bf of Theorem \ref{thm:mag:convex}})
\begin{itemize}
\item \textbf{Convergence (suppose that convex and gradient is bounded)}: Applying the result of Lemma \ref{lem:mag:xxstar}(i) and summing it from $k=0, \cdots, \infty$ gives
\begin{align}\label{inequ:converge}
    \lim_{k \rightarrow \infty}\sum_{i=1}^{k} \frac{f(x_k) - f(x^{\ast})^2}{\left\|d_k \right\|^2} \leq \left\|x_1 - x^{\ast} \right\|^2
\end{align}
Because the gradient is bounded by $G^2$, from (\ref{dk:nablaf}), we have 
    \begin{align}\label{dk:bounded}
       \left\|d_k \right\|^2  = \left\| \beta d_{k-1} + \nabla f(x_k) \right\|^2 
       & \leq \frac{1}{(1-\beta)}\sum_{i=1}^{k} \beta^{k-i} \left\|\nabla f(x_i) \right\|^2 \leq \frac{G^2}{(1-\beta)^2},
    \end{align}
    then $f(x_k) - f(x^{\ast}) \rightarrow 0$ ($k \rightarrow \infty$).
    \item \textbf{If the function is only convex and gradient is bounded}, then $f(\hat{x}_k) - f(x^{\ast}) \leq \mathcal{O}(1/\sqrt{k})$. Applying the result of Lemma \ref{lem:mag:xxstar}(i) and $\eta_k = \frac{f(x_k) - f^{\ast}}{\left\|d_k \right\|^2}$ with $\left\|d_k\right\|^2 \leq G^2/(1-\beta)^2$, we have
    \begin{align*}
     \left(\frac{1}{k}\sum_{i=1}^{k}(f(x_i) - f(x^{\ast}))\right)^2  & \mathop{\leq}^{(a)} \frac{1}{k}\sum_{i=1}^{k} [f(x_i) - f(x^{\ast})]^2 \leq \frac{G^2}{(1-\beta)^2 k }\left(\left\|x_1-x^{\ast} \right\|^2 - \left\|x_{k+1} - x^{\ast} \right\|^2 \right) \notag \\
     & \leq \frac{G^2}{(1-\beta)^2k}\left\|x_1-x^{\ast} \right\|^2
    \end{align*}
    where $(a)$ uses the Cauchy-Schwarz inequality that $\left(\frac{1}{k}\sum_{i=1}^k \alpha_i\right)^2 \leq \frac{1}{k}\sum_{i=1}^k\alpha_i^2$ for all $\alpha_i \geq 0$.
 By the convexity of $f$, we can obtain that
    \begin{align*}
        f(\hat{x}_k) - f^{\ast} \leq \frac{1}{k}\sum_{i=1}^k \left(f(x_i) - f(x^{\ast})\right) \leq  \frac{G\left\|x_1 - x^{\ast} \right\|}{(1-\beta)\sqrt{k}}.
    \end{align*}
    
    \item \textbf{If the objective function is strongly convex and the gradient is bounded}, then $\left\| x_k - x^{\ast} \right\|^2 \leq \mathcal{O}(1/k)$.
  
    If the objective function is $\mu$-strongly convex, we have $f(x_k) - f(x^{\ast}) \geq \frac{\mu}{2} \left\| x_k - x^{\ast} \right\|^2$. Due to the fact that gradient is bounded by $G^2$, by (\ref{dk:bounded}), we have $\left\|d_k \right\|^2 \leq \frac{G^2}{(1-\beta)^2}$
    and 
    \begin{align*}
      \left\|x_{k+1} - x^{\ast} \right\|^2 \leq \left\|x_{k} - x^{\ast} \right\|^2\left(1-\frac{(1-\beta)^2\mu^2}{4G^2} \left\|x_{k} - x^{\ast} \right\|^2 \right)
    \end{align*}
    We can achieve that $\left\|x_{k} - x^{\ast} \right\|^2  \leq \frac{4G^2}{(1-\beta)^2\mu^2} \frac{1}{k}$ by induction.

 \item {\bf If the function is polyhedral convex and Assumption~\ref{assumpt:polyhed} holds}, in this case, we know that the polyhedral error bound condition holds: there exists a constant $\kappa_1 > 0$ such that 
 \begin{align*}
    \left\|x -x^{\ast} \right\| \leq \frac{1}{\kappa_1} (f(x) - f^{\ast}), \,\, \forall x \in \mathcal{X}
 \end{align*}
Because the gradient (or subgradient) is bounded by $G^2$, that is $\max_{g \in \partial f(x_k)} \left\|g\right\|^2 \leq G^2$, from (\ref{dk:bounded}), we can achieve that
\begin{align*}
\left\| d_k \right\|^2 \leq \frac{G^2}{(1-\beta)^2}.
\end{align*}
Applying the result of Lemma \ref{lem:mag:xxstar} (i) and using the definition of $\eta_k$ in (\ref{mag:lr}), we have
\begin{align*}
    \left\|x_{k+1} -x^{\ast} \right\|^2 & \leq \left\|x_{k} -x^{\ast} \right\|^2 - \eta_k (f(x_k) - f^{\ast})  = \left\|x_{k} -x^{\ast} \right\|^2 -  \frac{(f(x_k) - f^{\ast})^2}{\left\|d_k \right\|^2} \notag \\ 
    & \leq \left\|x_{k} -x^{\ast} \right\|^2 - \frac{(1-\beta)^2}{G^2}(f(x) - f^{\ast})^2 \leq \left\|x_{k} -x^{\ast} \right\|^2 - \frac{(1-\beta)^2\kappa_1^2}{G^2}\left\|x_{k} -x^{\ast} \right\|^2 \notag \\
    & \leq  \left(1-\frac{(1-\beta)^2\kappa_1^2}{G^2}\right)\left\|x_{k} -x^{\ast} \right\|^2.
\end{align*}
In this case, $\left\|x_k -x^{\ast}\right\|^2 $ promotes linear convergence with a rate at least $1- \frac{(1-\beta)^2\kappa_1^2}{G^2}$. We must make sure that $1- \frac{(1-\beta)^2\kappa_1^2}{G^2} > 0$. If not, we can increase $G$ or decrease $\kappa_1$ to make the condition $1- \frac{(1-\beta)^2\kappa_1^2}{G^2} > 0$ hold.
\end{itemize}
\end{proof}

\begin{proof}({\bf Proof of Theorem \ref{thm:mag:smooth}})


 By the convexity and smoothness, we know that Lemma \ref{lem:mag:xxstar}(ii) holds. Then suppose that the objective function $f$ is semi-strongly convex with $\hat{\mu}$, we can achieve that 
\begin{align*}
        \left\|x_{k+1} - x^{\ast} \right\|^2 \leq \left(1-\frac{\hat{\mu}}{2} \left(\eta_k + \frac{(1-\beta)}{L}\right) \right)\left\|x_k -x^{\ast} \right\|^2 \leq \left(1-\frac{\hat{\mu}(1-\beta)}{2L} \right)\left\|x_k -x^{\ast} \right\|^2
\end{align*}
where $\eta_k \geq 0$.
That is $\left\|x_k-x^{\ast} \right\|^2$ promotes globally linear convergence with a rate at least $(1-(1-\beta)(2\kappa)^{-1})$ where $\kappa = L/\hat{\mu}$. 

\end{proof}

\subsection{Proofs of Theorems in Section \ref{sec:thm:sc}}
We provide the essential lemmas and the proofs for the important theorems in Section \ref{sec:thm:sc}. The results of ALR-SMAG for polyhedral convex and non-smooth functions and general convex functions which do not appear in the main content are presented in this part. 

The first lemma follows the result of Lemma \ref{lem:mag:convex} of ALR-MAG in the deterministic case but it is more complicated.
\begin{lemma}\label{lem:smag:convex}
For convex functions, if the step size $\eta_k \leq \frac{f_{S_k}(x_k) - f_{S_k}^{\ast}}{c\left\|d_k \right\|^2}$ for all $k\geq 1$, the iterates of SMAG satisfy that
$\left\langle d_{k-1}, x_k -x^{\ast}\right\rangle \geq \left(1-\frac{1}{c} \right)\sum_{i=1}^{k-1}\beta^{k-1-i}\left(f_{S_i}(x_i) - f_{S_i}^{\ast} \right) +  \sum_{i=1}^{k-1}\beta^{k-1-i}\left(f_{S_i}^{\ast} - f_{S_i}(x^{\ast})\right) $ for all $k\geq 2$.
\end{lemma}
\begin{proof}({\bf of Lemma \ref{lem:smag:convex}})
    For $k = 2$, we have 
    \begin{align*}
        \left\langle d_{1}, x_2 -x^{\ast}\right\rangle & = \left\langle \nabla f_{S_1}(x_1), x_1 - \eta_1 \nabla f_{S_1}(x_1) - x^{\ast}\right\rangle = -\eta_1 \left\|\nabla f_{S_1}(x_1) \right\|^2 + \left\langle \nabla f_{S_1}(x_1), x_1 - x^{\ast}\right\rangle \notag\\
        & \geq -\frac{1}{c}\left(f_{S_1}(x_1) - f_{S_1}^{\ast}\right) + f_{S_1}(x_1) - f_{S_1}(x^{\ast}) = \left(1-\frac{1}{c} \right)\left(f_{S_1}(x_1) - f_{S_1}^{\ast} \right) + f_{S_1}^{\ast} -  f_{S_1}(x^{\ast}).
    \end{align*} 
    For $k > 2$, if the claim $\left\langle d_{k-2}, x_{k-1} -x^{\ast}\right\rangle \geq 
 \left(1-\frac{1}{c} \right)\sum_{i=1}^{k-2}\beta^{k-2-i}\left(f_{S_i}(x_i) - f_{S_i}^{\ast} \right) + \sum_{i=1}^{k-2}\beta^{k-2-i}\left(f_{S_i}^{\ast} - f_{S_i}(x^{\ast})\right)$ holds at $k-2$, we have
       \begin{align*}
        \left\langle d_{k-1}, x_k -x^{\ast}\right\rangle &   = \left\langle d_{k-1}, x_k - x_{k-1} + x_{k-1} 
        -x^{\ast}\right\rangle = \left\langle d_{k-1}, -\eta_{k-1}d_{k-1} + x_{k-1} 
        -x^{\ast}\right\rangle  \notag \\
        & =  - \eta_{k-1}\left\|d_{k-1}\right\|^2 + \left\langle \nabla f_{S_{k-1}}(x_{k-1}) + \beta d_{k-2},  x_{k-1} 
        -x^{\ast}\right\rangle \notag \\
        & = - \eta_{k-1}\left\|d_{k-1}\right\|^2 + \left\langle \nabla f_{S_{k-1}}(x_{k-1}),  x_{k-1} 
        -x^{\ast}\right\rangle + \beta \left\langle d_{k-2}, x_{k-1} 
        -x^{\ast}\right\rangle \notag \\
        & \mathop{\geq}^{(a)} - \eta_{k-1}\left\|d_{k-1}\right\|^2 +  \left(f_{S_{k-1}}(x_{k-1}) - f_{S_{k-1}}(
        x^{\ast}) \right)+ \beta \left\langle d_{k-2}, x_{k-1} 
        -x^{\ast}\right\rangle \notag \\
        & \geq -\frac{1}{c}\left(f_{S_{k-1}}(x_{k-1}) - f_{S_{k-1}}^{\ast} \right)+ \left(f_{S_{k-1}}(x_{k-1}) - f_{S_{k-1}}(
        x^{\ast}) \right) + \beta \sum_{i=1}^{k-2}\beta^{k-2-i}\left(f_{S_i}^{\ast} - f_{S_i}(x^{\ast})\right) \notag \\
        & \quad + \beta  \left(1-\frac{1}{c} \right)\sum_{i=1}^{k-2}\beta^{k-2-i}\left(f_{S_i}(x_i) - f_{S_i}^{\ast} \right) \notag \\
        & = \left(1-\frac{1}{c} \right)\sum_{i=1}^{k-1}\beta^{k-1-i}\left(f_{S_i}(x_i) - f_{S_i}^{\ast} \right) +  \sum_{i=1}^{k-1}\beta^{k-1-i}\left(f_{S_i}^{\ast} - f_{S_i}(x^{\ast})\right) 
     \end{align*}
     where $(a)$ follows from the convexity of $f$ that $\left\langle \nabla f_{S_{k-1}}(x_k), x_k -x^{\ast} \right\rangle \geq f_{S_{k-1}}(x_k) - f_{S_{k-1}}(x^{\ast}) $ and $\eta_{k-1} \leq \frac{f_{S_{k-1}}(x_{k-1}) - f_{S_{k-1}}^{\ast}}{\left\|d_{k-1} \right\|^2}$. That is to say, this claim holds at $k-1$. By induction, we can conclude that $\left\langle d_{k-1}, x_k -x^{\ast}\right\rangle \geq \left(1-\frac{1}{c} \right)\sum_{i=1}^{k-1}\beta^{k-1-i}\left(f_{S_i}(x_i) - f_{S_i}^{\ast} \right) +  \sum_{i=1}^{k-1}\beta^{k-1-i}\left(f_{S_i}^{\ast} - f_{S_i}(x^{\ast})\right) $ for all $k\geq 2$. The proof is complete.
     \end{proof}

\begin{proof}({\bf Proof of Theorem \ref{thm:semi:sc}})

In this case, the formula of step size for the stochastic version of ALR-MAG is
\begin{align*}
\eta_k =  \min \left\lbrace \frac{f_{S_k}(x_k) - f_{S_k}^{\ast}}{c\left\|d_k \right\|^2}, \eta_{\max} \right\rbrace.
\end{align*}
By the definition of step size, we have $\eta_k \leq \frac{f_{S_k}(x_k) - f_{S_k}^{\ast}}{c\left\|d_k \right\|^2}$. By Lemma \ref{lem:smag:convex}, we get that 
\begin{align*}
 \left\langle  d_k, x_k -  x^{\ast} \right\rangle & \geq \left\langle \nabla f_{S_k}(x_k), x_k -x^{\ast} \right\rangle  + \left(1-\frac{1}{c} \right)\sum_{i=1}^{k-1}\beta^{k-i}\left(f_{S_i}(x_i) - f_{S_i}^{\ast}\right)  - \sum_{i=1}^{k-1}\beta^{k-i}\left( f_{S_i}(x^{\ast}) - f_{S_i}^{\ast} \right)\notag \\
 & \geq \left(f_{S_k}(x_k) -  f_{S_k}^{\ast} \right)  + \left(1-\frac{1}{c} \right)\sum_{i=1}^{k-1}\beta^{k-i}\left(f_{S_i}(x_i) - f_{S_i}^{\ast}\right) - \sum_{i=1}^{k}\beta^{k-i}\left( f_{S_i}(x^{\ast}) - f_{S_i}^{\ast} \right).
\end{align*}
To make the analysis to be clear, we define a 0-1 event $X_k$. If $\eta_k = \frac{f_{S_k}(x_k) - f_{S_k}^{\ast}}{c\left\|d_k \right\|^2} \leq \eta_{\max}$, it implies that $X_k$ happens (i.e., $X_k=1$); otherwise, $X_k=0$. Let $P_k = P(X_k=1)$. First, we consider the event $X_k$ happens, then the distance $\left\|x_{k+1} - x^{\ast} \right\|^2$ can be estimated as 
\begin{align}\label{inequ:mag:xxk:sl}
    \left\|x_{k+1} - x^{\ast} \right\|^2 & = \left\|x_{k} - \eta_k d_k -  x^{\ast} \right\|^2 = \left\|x_{k} - x^{\ast} \right\|^2 -  2\eta_k \left\langle  d_k, x_k -  x^{\ast} \right\rangle + \eta_k^2 \left\| d_k \right\|^2  \notag \\
    & \leq \left\|x_{k} - x^{\ast} \right\|^2 -  2\frac{(f_{S_k}(x_k) - f_{S_k}^{\ast})^2}{ c \left\| d_k \right\|^2}  + \frac{(f_{S_k}(x_k) - f_{S_k}^{\ast})^2}{ c^2\left\| d_k \right\|^2} \notag \\
    & - \frac{2}{c}\left(1-\frac{1}{c} \right)\frac{f_{S_k}(x_k) - f_{S_k}^{\ast}}{\left\|d_k \right\|^2}\sum_{i=1}^{k-1}\beta^{k-i}\left(f_{S_i}(x_i) - f_{S_i}^{\ast}\right) + 2\eta_k \sum_{i=1}^{k}\beta^{k-i}\left( f_{S_i}(x^{\ast}) - f_{S_i}^{\ast} \right)\notag \\
    & \leq  \left\|x_{k} - x^{\ast} \right\|^2 -  \frac{1}{c^2}\frac{(f_{S_k}(x_k) - f_{S_k}^{\ast})^2}{\left\| d_k \right\|^2}- \frac{2\left(c-1 \right)}{c^2}\frac{f_{S_k}(x_k) - f_{S_k}^{\ast}}{\left\|d_k \right\|^2}\sum_{i=1}^{k}\beta^{k-i}\left(f_{S_i}(x_i) - f_{S_i}^{\ast}\right) \notag \\
    & \quad + 2\eta_k\sum_{i=1}^{k}\beta^{k-i}\left( f_{S_i}(x^{\ast}) - f_{S_i}^{\ast} \right).
\end{align}
The $L$-smooth property of $f_{S_i}$ for $i=1,\cdots, k$ gives
\begin{align}\label{inequ:mag:fx}
    \sum_{i=1}^{k}\beta^{k-i}\left(f_{S_i}(x_i) - f_{S_i}^{\ast}\right) \geq \frac{1}{2L}\sum_{i=1}^{k}\beta^{k-i} \left\|\nabla f_{S_i}(x_i) \right\|^2.
\end{align}
By (\ref{dk:nablaf}), we know that
$\left\|d_k \right\|^2   \leq \frac{1}{(1-\beta)} \sum_{i=1}^{k} \beta^{k-i}\left\| \nabla f_{S_i}(x_i) \right\|^2 $.
Applying (\ref{inequ:mag:fx}) and (\ref{dk:nablaf}) into (\ref{inequ:mag:xxk:sl}), we can achieve that
\begin{align}\label{inequ:ssc:case1}
\left\|x_{k+1} - x^{\ast} \right\|^2  
    & \leq \left\|x_{k} - x^{\ast} \right\|^2  - (1-\beta)\frac{\left(c-1 \right)}{c^2L} \left(f_{S_k}(x_k) - f_{S_k}^{\ast}\right) + 2\eta_k \sum_{i=1}^{k}\beta^{k-i}\left( f_{S_i}(x^{\ast}) - f_{S_i}^{\ast} \right) \notag \\
    & \leq \left\|x_{k} - x^{\ast} \right\|^2  - (1-\beta)\frac{\left(c-1 \right)}{c^2L} \left(f_{S_k}(x_k) - f_{S_k}^{\ast}\right) + 2\eta_{\max} \sum_{i=1}^{k}\beta^{k-i}\left( f_{S_i}(x^{\ast}) - f_{S_i}^{\ast} \right) \notag \\
    & \leq \left\|x_{k} - x^{\ast} \right\|^2  - (1-\beta)\frac{\left(c-1 \right)}{c^2L} \left(f_{S_k}(x_k) - f_{S_k}(x^{\ast})\right) + 2\eta_{\max} \sum_{i=1}^{k}\beta^{k-i}\left( f_{S_i}(x^{\ast}) - f_{S_i}^{\ast} \right)
\end{align}
where $\eta_k \leq \eta_{\max}$ and $f_{S_k}^{\ast} \leq f_{S_k}(x^{\ast})$.

If $\eta_k = \eta_{\max} < \frac{f_{S_k}(x_k) - f_{S_k}^{\ast}}{c\left\|d_k \right\|^2}$, that is $X_k = 0$, we have
\begin{align}\label{inequ:ssc:case2}
 \left\|x_{k+1} - x^{\ast} \right\|^2 & \leq \left\|x_{k} - x^{\ast} \right\|^2 -  2\eta_{\max}(f_{S_k}(x_k) - f_{S_k}^{\ast}) + \eta_{\max}\frac{(f_{S_k}(x_k) - f_{S_k}^{\ast})}{c\left\|d_k \right\|^2} \left\|d_k \right\|^2 \notag \\
    &  + 2\eta_{\max}\sum_{i=1}^{k}\beta^{k-i}\left( f_{S_i}(x^{\ast}) - f_{S_i}^{\ast} \right) - 2\eta_{\max}\left(1-\frac{1}{c} \right)\sum_{i=1}^{k-1}\beta^{k-i}\left(f_{S_i}(x_i) - f_{S_i}^{\ast}\right)  \notag \\
    & \mathop{\leq}^{(a)} \left\|x_{k} - x^{\ast} \right\|^2 -  \left(2-\frac{1}{c} \right)\eta_{\max} (f_{S_k}(x_k) - f_{S_k}^{\ast}) +2\eta_{\max}\sum_{i=1}^{k}\beta^{k-i}\left( f_{S_i}(x^{\ast}) - f_{S_i}^{\ast} \right) \notag\\
    & \mathop{\leq}^{(b)}  \left\|x_{k} - x^{\ast} \right\|^2 -  \left(2-\frac{1}{c} \right)\eta_{\max}(f_{S_k}(x_k) -  f_{S_k}(x^{\ast})) + 2\eta_{\max}\sum_{i=1}^{k}\beta^{k-i}\left( f_{S_i}(x^{\ast}) - f_{S_i}^{\ast} \right) 
\end{align}
where $(a)$ uses the truth that $c > 1$ and $f_{S_i}(x_i) \geq f_{S_i}^{\ast}$ 
for each $i \geq 1$, and $(b)$ follows from the fact that  $f_{S_k}(x^{\ast}) \geq f_{S_k}^{\ast} = \min f_{S_k}(x)$. 
Overall, no matter $\frac{f_{S_k}(x_k) - f_{S_k}^{\ast}}{c\left\|d_k \right\|^2} \leq \eta_{\max}$ or not, we both have
\begin{align}
 \left\|x_{k+1} - x^{\ast} \right\|^2 & \leq \left\|x_{k} - x^{\ast} \right\|^2  - \min\left((1-\beta)\frac{\left(c-1 \right)}{c^2L}, \left(2-\frac{1}{c} \right)\eta_{\max}\right)(f_{S_k}(x_k) - f_{S_k}^{\ast})  + 2\eta_{\max}\sum_{i=1}^{k}\beta^{k-i}\left( f_{S_i}(x^{\ast}) - f_{S_i}^{\ast} \right). 
\end{align}
We then take conditional expectation w.r.t. $\mathcal{F}_k$\footnote{$\mathcal{F}_k $ is the $\sigma$-algebra of the set $\left\lbrace (x_1, \nabla f_{S_1}(x_1)), \cdots,(x_{k-1}, \nabla f_{S_{k-1}}(x_{k-1})), x_k \right\rbrace$} on  the above inequalities:
\begin{align}
 \E[\left\|x_{k+1} - x^{\ast} \right\|^2 \mid \mathcal{F}_k & \leq \left\|x_{k} - x^{\ast} \right\|^2  - \min\left((1-\beta)\frac{\left(c-1 \right)}{c^2L}, \left(2-\frac{1}{c} \right)\eta_{\max}\right)\E[f_{S_k}(x_k) - f_{S_k}^{\ast} \mid \mathcal{F}_k]  \notag \\
 & + 2\eta_{\max}\sum_{i=1}^{k}\beta^{k-i}\left( \E[f_{S_i}(x^{\ast}) - f_{S_i}^{\ast}] \right) \notag \\
 & \mathop{\leq}^{(a)} \left\|x_{k} - x^{\ast} \right\|^2  - \min\left((1-\beta)\frac{\left(c-1 \right)}{c^2L}, \left(2-\frac{1}{c} \right)\eta_{\max}\right)\frac{\hat{\mu}}{2}\left\|x_k -x^{\ast} \right\|^2 +\frac{2\eta_{\max} \sigma^2}{(1-\beta)} \notag \\
 & \leq \left(1- \rho_1\right)\left\|x_{k} - x^{\ast} \right\|^2  +\frac{2\eta_{\max} \sigma^2}{(1-\beta)}.
\end{align}
where $\rho_1 = \min\left\lbrace \frac{(1-\beta)\left(c-1 \right)\hat{\mu}}{2c^2L} , \frac{(2c-1)\hat{\mu} \eta_{\max}}{2c}\right\rbrace$; and  $(a)$ uses the facts that $\E[f_{S_k}(x_k) - f_{S_k}(x^{\ast})\mid \mathcal{F}_k] = \E[f_{S_k}(x_k) - f_{S_k}(x^{\ast})\mid \mathcal{F}_k] = f(x_k) - f(x^{\ast})  \geq \frac{\hat{\mu}}{2}\left\|x_k -x^{\ast} \right\|^2$ and the assumption on $f_{S_k}^{\ast}$.
Telescoping the above inequality from $k=1$ to $K$ gives that
\begin{align*}
    \E[\left\|x_{K+1} - x^{\ast} \right\|^2  \mid \mathcal{F}_K]  \leq \left(1- \rho_1\right)^K\left\|x_{1} - x^{\ast} \right\|^2  +\frac{2\eta_{\max} \sigma^2}{\rho_1(1-\beta)}.
\end{align*}
The proof is complete. 
\end{proof}
Next, we consider the convergence of SMAG with (\ref{mag:lr:sc}) for the polyhedral convex functions which is a special category of nonsmooth and non-strongly convex functions.
\begin{theorem}({\bf Polyhedral convex and non-smooth functions})\label{thm:polyhedral:c} Under interpolation ($\sigma=0$), we suppose that function $f$ is polyhedral convex with $\hat{\kappa}$ and the gradient of each realization $\nabla f(x;\xi)$ is bounded by $G^2$. Consider the step size (\ref{mag:lr:sc}) with $c > 1/2$ and $\eta_{\max}=\infty$, we get that
\begin{align*}
\E[\left\|x_{k+1} - x^{\ast} \right\|^2] \leq \left(1- \rho_2\right)^k\left\|x_1 - x^{\ast} \right\|^2
\end{align*}
where $\rho_2 = \frac{\hat{\kappa}^2(1-\beta)^2\left(2c-1 \right)}{c^2bG^2}$.
\end{theorem}
For the interpolated functions, Theorem \ref{thm:polyhedral:c} generalizes the linear convergence rate beyond the smooth and semi-strongly convex functions.
\begin{proof}({\bf Proof of Theorem \ref{thm:polyhedral:c}}) Under the interpolation setting, it has $\min f(x;\xi) = f^{\ast} = f(x^{\ast}) $ and all loss function $f_i$ agrees with one common minimizer $x^{\ast}$.
We assume that the function $f(x)$ is polyhedral convex with $\hat{\kappa}>0$, that is $\left\|x - x^{\ast} \right\| \leq \frac{1}{\hat{\kappa}}(f(x) - f(x^{\ast}))$. Each realization function $f(x;\xi)$ is Lipschitz continuous (that is  $\left\|\nabla f(x;\xi) \right\|^2 \leq  G^2$ for all $x$). We consider the SMAG algorithm with the step size defined by (\ref{mag:lr:sc}) and $\eta_{\max} = \infty$. In this case, $L$-smooth property does not hold, that is to say, we can not use the Inequality (\ref{inequ:mag:fx}). Due to that $\left\|\nabla f(x_k;\xi) \right\|^2 \leq G^2$ which induces that $\left\|\nabla f_{S_k}(x_k) \right\|^2 \leq G^2$. By the relationship $\left\|d_k \right\|^2 \leq \frac{1}{1-\beta}\sum_{i=1}^{k}\beta^{k-i}\left\|\nabla f_{S_i}(x_i) \right\|^2 \leq \frac{G^2}{(1-\beta)^2}$, we still can achieve that
\begin{align}
  \left\|x_{k+1} - x^{\ast} \right\|^2 & \leq   \left\|x_{k} - x^{\ast} \right\|^2  - \frac{(1-\beta)^2\left(2c-1 \right)}{c^2} \frac{\left(f_{S_k}(x_k) - f_{S_k}^{\ast}\right)^2}{\left\|d_k \right\|^2} + 2\eta_k \sum_{i=1}^{k}\beta^{k-i}\left( f_{S_i}(x^{\ast}) - f_{S_i}^{\ast} \right) \notag\\
  & \leq  \left\|x_{k} - x^{\ast} \right\|^2  - (1-\beta)^2\frac{\left(2c-1 \right)}{c^2} \frac{\left(f_{S_k}(x_k) - f_{S_k}^{\ast}\right)^2}{G^2} + 2\eta_k \sum_{i=1}^{k}\beta^{k-i}\left( f_{S_i}(x^{\ast}) - f_{S_i}^{\ast} \right) \notag \\
  & \mathop{=}^{(a)} \left\|x_{k} - x^{\ast} \right\|^2  - (1-\beta)^2\frac{\left(2c-1 \right)}{c^2} \frac{\left(f_{S_k}(x_k) - f_{S_k}(x^{\ast})\right)^2}{G^2}
\end{align}
where $(a)$ uses the fact that $f_{S_k}^{\ast} = f_{S_k}(x^{\ast})$ for each $k \geq 1$.
Taking conditional expectation on the both side, we have 
\begin{align*}
   \E[ \left\|x_{k+1} - x^{\ast} \right\|^2 \mid \mathcal{F}_k]  & \leq \left\|x_k - x^{\ast} \right\|^2 - \frac{(1-\beta)^2\left(2c-1 \right)}{c^2 bG^2} \E[\left(f_{S_k}(x_k) - f_{S_k}(x^{\ast})\right)^2 \mid \mathcal{F}_k] \notag \\
   & \leq \left\|x_k - x^{\ast} \right\|^2 - \frac{(1-\beta)^2\left(2c-1 \right)}{c^2G^2} \left(\E[f_{S_k}(x_k) - f_{S_k}(x^{\ast})\mid \mathcal{F}_k]\right)^2  \notag \\
   & = \left\|x_k - x^{\ast} \right\|^2 - \frac{(1-\beta)\left(2c-1 \right)}{c^2G^2} \left(f(x_k) - f^{\ast}]\right)^2  \notag \\
   & \leq \left(1-\frac{\hat{\kappa}^2(1-\beta)^2\left(2c-1 \right)}{c^2G^2} \right)\left\|x_k - x^{\ast} \right\|^2
\end{align*}
For $k=1, \cdots, K$, we can achieve the linear convergence with a rate $\rho = 1- \frac{\hat{\kappa}^2(1-\beta)^2\left(2c-1 \right)}{c^2G^2}$. We now complete the proof.
\end{proof}

\begin{theorem}({\bf General convex functions})\label{thm:convex}
Assume that each individual function $f(x;\xi)$ is convex and $L$-smooth for $\xi \in \Xi$. Consider SMAG under step size  (\ref{mag:lr:sc}) with $c > 1$, we can achieve that 
\begin{align*}
  \E[f(\hat{x}_K) - f^{\ast}] \leq   \frac{1}{Q} \frac{\left\|x_1 - x^{\ast} \right\|^2}{K}  + \frac{2\eta_{\max}\sigma^2}{Q(1-\beta)}
\end{align*}
where $Q = \min\left((2-1/c)\eta_{\max}, (1-\beta)(c-1)/(c^2L)\right) $ and $\hat{x} = \frac{1}{K}\sum_{k=1}^{K} x_k$.
\end{theorem}
The first observation is that the size of the solution's neighborhood is also proportional to $\eta_{\max}$, similar to the semi-strongly convex case of Theorem \ref{thm:semi:sc}. If the interpolation condition holds, SMAG under (\ref{mag:lr:sc}) can achieve an $\mathcal{O}(1/K)$ convergence rate to reach the optimum $f^{\ast}$.

\begin{proof}({\bf Proof of Theorem \ref{thm:convex}}) In this case, we consider the function is convex and $L$-smooth. 

Similar to Theorem \ref{thm:semi:sc}, (\ref{inequ:ssc:case1}) and (\ref{inequ:ssc:case2}) still hold. The only difference from Theorem \ref{thm:semi:sc} is that we do not have $f(x_k) - f(x^{\ast}) \geq \frac{\mu}{2}\left\|x_k - x^{\ast} \right\|^2$.
    Thus
    \begin{align*}
     \E[\left\|x_{k+1} - x^{\ast} \right\|^2  \mid \mathcal{F}_k] & 
     \leq \left\|x_{k} - x^{\ast} \right\|^2 - Q \left(\E[f_{S_k}(x_k) - f_{S_k}(x^{\ast}) \right) + 2\eta_{\max}\sum_{i=1}^k \beta^{k-i}\E[(f_{S_k}(x^{\ast}) - f_{S_k}^{\ast}) \mid \mathcal{F}_k] \notag \\
     & = \left\|x_{k} - x^{\ast} \right\|^2 - Q \left(f(x_k) - f(x^{\ast}) \right) + \frac{2\eta_{\max}\sigma^2}{1-\beta} 
    \end{align*}
    where $Q = \min\left\lbrace \frac{(1-\beta)\left(c-1 \right)}{c^2L} , \frac{(2c-1)\eta_{\max}}{c}\right\rbrace $. Summing the above inequality from $k=1$ to $K$ and dividing $Q$ to both side, we have 
    \begin{align*}
        f(\hat{x}_K) - f^{\ast} = \frac{1}{K}\sum_{k=1}^{K} (f(x_k) - f^{\ast}) & \leq  \frac{1}{K}\sum_{k=1}^{K}\frac{\left\|x_k -x^{\ast} \right\|^2 - \E[\left\|x_{k+1} - x^{\ast}\right\|^2 \mid \mathcal{F}_k]}{Q} + \frac{2\eta_{\max}\sigma^2}{(1-\beta)Q} \notag \\
        & \leq \frac{\left\|x_1 -x^{\ast} \right\|^2 }{K Q} + \frac{2\eta_{\max}\sigma^2}{(1-\beta)Q}.
    \end{align*}
 Now, the proof is complete.
\end{proof}

We now investigate the convergence of  ALR-SMAG  for a class of nonconvex functions. The quasar convex functions with respect to $x^{\ast} \in \mathcal{X}^{\ast}$ is an extension of star convexity~\citep{nesterov2006cubic} and convexity.
\begin{theorem}({\bf Quasar convex functions})\label{thm:quasar:convex}
Under interpolation ($\sigma=0$), we assume that each individual function $f(x;\xi)$ is $\zeta$-quasar-convex and $L$-smooth for $\xi \in \Xi$. Consider ALR-SMAG  with $c > 1/\zeta$, we can achieve that 
\begin{align*}
    \min_{i=1,\cdots, K} f(x_i) - f^{\ast} \leq \frac{Lc^2 }{(1-\beta)(\zeta c -1)} \frac{\left\|x_1 -x^{\ast} \right\|^2}{K}.
\end{align*}
\end{theorem}
Under interpolation, Theorem \ref{thm:quasar:convex} provides an $\mathcal{O}\left(1/K\right)$ convergence guarantee to reach the optimum $f^{\ast}$ for a class of nonconvex functions for ALR-SMAG.
\begin{proof}({\bf Proofs of Theorem \ref{thm:quasar:convex}}) We assume that each individual function $f(x;\xi)$ is $\zeta$-quasar-convex and $L$-smooth. Under interpolation, it implies that each component function $f(x;\xi)$ agrees with a common minimizer $x^{\ast}$. That is to say: the mini-batch functions $f_{S_k}(x)$ is also $\zeta$-quasar-convex and satisfies
\begin{align*}
    \left\langle \nabla f_{S_k}(x_k), x_k -x^{\ast}\right\rangle \geq \zeta \left( f_{S_k}(x_k) - f_{S_k}^{\ast} \right)
\end{align*}
where $\zeta \in (0,1]$ and $k\geq 1$.  In this case, the result of Lemma \ref{lem:smag:convex} is 
\begin{align*}
\left\langle d_{k-1}, x_k -x^{\ast}\right\rangle & \geq  \left(\zeta -\frac{1}{c} \right)\sum_{i=1}^{k-1}\beta^{k-1-i}\left(f_{S_i}(x_i) - f_{S_i}^{\ast} \right) +  \zeta \sum_{i=1}^{k-1}\beta^{k-1-i}\left(f_{S_i}^{\ast} - f_{S_i}(x^{\ast})\right) \notag \\
& = \left(\zeta -\frac{1}{c} \right)\sum_{i=1}^{k-1}\beta^{k-1-i}\left(f_{S_i}(x_i) - f_{S_i}^{\ast} \right)
\end{align*}
where $\zeta > 1/c$.
Then
\begin{align*}
 \left\langle  d_k, x_k -  x^{\ast} \right\rangle & \geq \left\langle \nabla f_{S_k}(x_k), x_k -x^{\ast} \right\rangle  + \left(\zeta -\frac{1}{c} \right)\sum_{i=1}^{k-1}\beta^{k-i}\left(f_{S_i}(x_i) - f_{S_i}^{\ast}\right)  \notag \\
 & \geq \zeta \left(f_{S_k}(x_k) -  f_{S_k}^{\ast} \right)  + \left(\zeta -\frac{1}{c} \right)\sum_{i=1}^{k-1}\beta^{k-i}\left(f_{S_i}(x_i) - f_{S_i}^{\ast}\right).
\end{align*}
We consider the step size (\ref{mag:lr:sc}) and $\eta_{\max} = \infty$. The distance of $\left\|x_{k+1} -x^{\ast} \right\|^2$ can be evaluated as
\begin{align}\label{inequ:smag:quasar:1}
   \left\|x_{k+1} - x^{\ast} \right\|^2   &   = \left\|x_{k} - x^{\ast} \right\|^2    - 2\eta_k \left\langle d_k, x_k -x^{\ast} \right\rangle + \eta_k^2\left\|d_k \right\|^2 \notag \\
   & \leq \left\|x_{k} - x^{\ast} \right\|^2 - \frac{2(f_{S_k}(x_k) - f_{S_k}^{\ast})}{c\left\|d_k \right\|^2} \left(\zeta \left(f_{S_k}(x_k) -  f_{S_k}^{\ast} \right)  + \left(\zeta -\frac{1}{c} \right)\sum_{i=1}^{k-1}\beta^{k-i}\left(f_{S_i}(x_i) - f_{S_i}^{\ast}\right) \right) \notag \\
   & \quad + \eta_k\frac{(f_{S_k}(x_k) - f_{S_k}^{\ast})}{c\left\|d_k \right\|^2}\left\|d_k \right\|^2.
\end{align}
By the smoothness property of each $f(x;\xi)$ and $\left\|d_k \right\|^2 \leq \frac{1}{1-\beta}\sum_{i=1}^{k}\beta^{k-i}\left\|\nabla f_{S_i}(x_i) \right\|^2$, we obtain that
\begin{align*}
   \sum_{i=1}^{k-1}\beta^{k-i}\left(f_{S_i}(x_i) - f_{S_i}^{\ast}\right)  \geq \frac{1}{2L}\sum_{i=1}^k \beta^{k-i}\left\|\nabla f_{S_i}(x_i)\right\|^2 \geq \frac{(1-\beta)}{2L}\left\|d_k \right\|^2.
\end{align*}
Incorporating the above inequality to (\ref{inequ:smag:quasar:1}) gives that
\begin{align}\label{inequ:smag:quasar:2}
  \left\|x_{k+1} - x^{\ast} \right\|^2   &   \leq \left\|x_{k} - x^{\ast} \right\|^2 - \frac{(1-\beta)(\zeta c -1)}{Lc^2}\left( f_{S_k}(x_k) - f_{S_k}^{\ast} \right) - \left(2\zeta - \frac{1}{c} \right)\eta_k \left(f_{S_k}(x_k) - f_{S_k}^{\ast}\right) \notag \\
  & \leq \left\|x_{k} - x^{\ast} \right\|^2 - \frac{(1-\beta)(\zeta c -1)}{Lc^2}\left( f_{S_k}(x_k) - f_{S_k}^{\ast} \right)
\end{align}
where the last inequality holds since $\zeta > 1/c$. Taking conditional expectation w.r.t. $\mathcal{F}_k$ on the both side of (\ref{inequ:smag:quasar:2}), we achieve that
\begin{align*}
    \E[ \left\|x_{k+1} - x^{\ast} \right\|^2 \mid \mathcal{F}_k] \leq \left\|x_{k} - x^{\ast} \right\|^2 - \frac{(1-\beta)(\zeta c -1)}{Lc^2}\left( f(x_k) - f^{\ast}. \right)
\end{align*}
Diving the above inequality by a constant $Q_1 = \frac{(1-\beta)(\zeta c -1)}{Lc^2}$ and summing over $k=1,\cdots, K$ gives that
\begin{align*}
    \min_{i=1,\cdots, K} f(x_i) - f^{\ast} \leq \frac{1}{K}\sum_{i=1}^{K}\left( f(x_i) - f^{\ast}\right) \leq \frac{1}{Q_1}\left(\E[\left\|x_{k} - x^{\ast} \right\|^2] -  \E[ \left\|x_{k+1} - x^{\ast} \right\|^2]\right) \leq \frac{Lc^2 \left\|x_1 -x^{\ast} \right\|^2}{K(1-\beta)(\zeta c -1)}.
\end{align*}
We now complete the proof.
\end{proof}

\subsection{Theoretical Guarantees of ALR-HB on Least-Squares Problems}\label{sec:hb:ls}
In this part, we consider the theoretical convergence of HB under the step size defined by (\ref{mad:hb:lr:1}) or (\ref{mad:hb:lr:L}) and get the fast linear convergence rate for ALR-HB on least-squares problems. 

We recall the step size (\ref{mad:hb:lr:L}) that is ALR-HB(v2):
\begin{align*}
    \eta_k = \frac{1}{2L} + \frac{f(x_k) - f^{\ast}}{\left\|\nabla f(x_k) \right\|^2} + \beta \frac{\left\langle \nabla f(x_k), x_k -x_{k-1} \right\rangle}{\left\|\nabla f(x_k) \right\|^2}.
\end{align*}
In general, the step size (\ref{mad:hb:lr:L}) may be not positive if $\left\langle \nabla f(x_k), x_k -x_{k-1} \right\rangle \ll - \left(f(x_k) - f^{\ast} \right)$. It means that the momentum direction $x_{k} - x_{k-1}$ has an acute angle with $-\nabla f(x_k) $ and it also promotes the reduction on the function values, just as $-\nabla f(x_k)$. In this way, from the formula (\ref{mad:hb:lr:L}), the weight on $-\nabla f(x_k)$  will be reduced. However, we still choose to trust $-\nabla f(x_k)$ more which is the exact descent direction, compared to the momentum direction $x_k -x_{k-1}$. Thus, we truncate the step size to be a constant when the inner product $\left\langle \nabla f(x_k), x_k -x_{k-1} \right\rangle \leq - \left(f(x_k) - f^{\ast} \right)$. 
 We define $$\tilde{\eta}_k = \frac{f(x_k) - f^{\ast}}{\left\|\nabla f(x_k) \right\|^2} + \beta \frac{\left\langle \nabla f(x_k), x_k -x_{k-1} \right\rangle}{\left\|\nabla f(x_k) \right\|^2} - \frac{1-\beta}{2L}.$$ 
 When $\left\langle \nabla f(x_k), x_k -x_{k-1} \right\rangle \geq - \left(f(x_k) - f^{\ast} \right) $, we can see that $\tilde{\eta}_k \geq 0$. Then the step size can be re-written as
\begin{align}\label{truncated:HB;Polyak:lr}
\eta_k =  \frac{1}{2L} + \frac{f(x_k) - f^{\ast}}{\left\|\nabla f(x_k) \right\|^2} + \beta \frac{\left\langle \nabla f(x_k), x_k -x_{k-1} \right\rangle}{\left\|\nabla f(x_k) \right\|^2} = \frac{2-\beta}{2L} +  \tilde{\eta}_k  \tag{Truncated ALR-HB(v2)}
\end{align}
If $\left\langle \nabla f(x_k), x_k -x_{k-1} \right\rangle < - \left(f(x_k) - f^{\ast} \right) $, we set $\tilde{\eta}_k = 0$ and the step size $\eta_k = \frac{2-\beta}{2L}$. In the numerical experiment on least-squares in Section \ref{sec:least:squares}, such a lower bound $(2-\beta)/(2L)$ for ALR-HB(v2) never hits. For the step size defined by (\ref{mad:hb:lr:1}) without $L$, the truncated lower bound is $(1-\beta)/(2L)$. This is a very small number for example when we set $\beta = 0.9$ which is commonly used in practice.
\begin{theorem}({\bf ALR-HB(v2) for least-squares problems})\label{thm:hb:ls}
For the least-squares problem, consider the heavy-ball method defined by (\ref{alg:HB}) and truncated step size by (\ref{truncated:HB;Polyak:lr}), we can derive the following property:
    \begin{align*}
         \left\|  \begin{bmatrix}
        x_{k+1} - x^{\ast} \\
        x_k -  x^{\ast}
        \end{bmatrix}  \right\|^2 
        & = \left\|\begin{bmatrix}
       (1+\beta)\I_d - \hat{\alpha}  A & -\beta \I_d \\
        \I_d & \bf{0}
        \end{bmatrix} 
        \begin{bmatrix}
        x_{k} - x^{\ast} \\
        x_{k-1} -  x^{\ast}
        \end{bmatrix}  \right\|^2 - \tilde{\eta}_k^2 \left\| \nabla f(x_k) \right\|^2.  
    \end{align*}
\end{theorem}
where $\hat{\alpha}=(2-\beta)/(2L)$.
Especially, if the problem is $\mu$-strongly convex and $L$-smooth, we set $\beta = \left( \frac{\sqrt{\kappa}-1}{\sqrt{\kappa}+1}\right)^2$ where $\kappa = L/\mu$ and $\mu=\lambda_{\min}(A), L = \lambda_{\max}(A)$, we can achieve the linear convergence rate at least
    \begin{align*}
         \left\|  \begin{bmatrix}
        x_{k+1} - x^{\ast} \\
        x_k -  x^{\ast}
        \end{bmatrix}  \right\|^2 
        & =\rho^k \left\| 
        \begin{bmatrix}
        x_{2} - x^{\ast} \\
        x_{1} -  x^{\ast}
        \end{bmatrix}  \right\|^2.  
    \end{align*}
where $\rho = 1- \frac{4-\sqrt{15}}{2(\sqrt{\kappa}+1)}$.
\begin{proof}({\bf of Theorem \ref{thm:hb:ls}})
 We consider the least-squares problem,  
\begin{align*}
    f(x)= \frac{1}{2}x^{T}Ax + \left\langle x, b\right\rangle + c = \frac{1}{2}\left\|x-x^{\ast} \right\|_A^2 + f^{\ast}
\end{align*}
where $A \in \R^{d\times d}$ is symmetric and positive definite, $x^{\ast} = -A^{-1}b$ and $f^{\ast} = -\frac{1}{2}b^{T}A^{-1}b + c$,
and its gradient $\nabla f(x) = Ax +  b = A(x-x^{\ast})$. Recall the truncated step size of ALR-HB(v2), we let $\hat{\alpha}= \frac{2-\beta}{2L}$ and $\tilde{\eta}_k = \frac{f(x_k) - f^{\ast}}{\left\|\nabla f(x_k) \right\|^2} + \beta \frac{\left\langle \nabla f(x_k), x_k -x_{k-1} \right\rangle}{\left\|\nabla f(x_k) \right\|^2} - \frac{1-\beta}{2L}$. When $\left\langle \nabla f(x_k), x_k -x_{k-1} \right\rangle \geq - \left(f(x_k) - f^{\ast} \right) $, we can see that $\tilde{\eta}_k \geq 0$ and
\begin{align*}
\eta_k =  \frac{1}{2L} + \frac{f(x_k) - f^{\ast}}{\left\|\nabla f(x_k) \right\|^2} + \beta \frac{\left\langle \nabla f(x_k), x_k -x_{k-1} \right\rangle}{\left\|\nabla f(x_k) \right\|^2} = \frac{2-\beta}{2L} +  \tilde{\eta}_k
\end{align*}
If $\left\langle \nabla f(x_k), x_k -x_{k-1} \right\rangle \leq - \left(f(x_k) - f^{\ast} \right) $, we have $\tilde{\eta}_k = 0$ and $\eta_k = \frac{2-\beta}{2L}$.
The iterative formula of HB can be re-written as
  \begin{align*}
        \begin{bmatrix}
        x_{k+1} - x^{\ast} \\
        x_k -  x^{\ast}
        \end{bmatrix} 
        & = \begin{bmatrix}
       (1+\beta)\I_d & -\beta \I_d \\
        \I_d & \bf{0}
        \end{bmatrix} 
        \begin{bmatrix}
        x_{k} - x^{\ast} \\
        x_{k-1} -  x^{\ast}
        \end{bmatrix}  
        - {\eta}_k   
        \begin{bmatrix}
        \nabla f(x_k) \\
        \bf{0}
        \end{bmatrix}  \\ \notag \\
        & = \begin{bmatrix}
        (1+\beta)\I_d - \hat{\alpha} A  & -\beta \I_d \\
        \I_d & \bf{0}
        \end{bmatrix} 
        \begin{bmatrix}
        x_{k} - x^{\ast} \\
        x_{k-1} -  x^{\ast}
        \end{bmatrix}  
        - \tilde{\eta}_k   
        \begin{bmatrix}
        \nabla f(x_k) \\
        \bf{0}
        \end{bmatrix}.
    \end{align*}
    Then
     \begin{align}\label{inequ:hb:poly}
      \left\|  \begin{bmatrix}
        x_{k+1} - x^{\ast} \\
        x_k -  x^{\ast}
        \end{bmatrix}  \right\|^2 
        & = \left\|\begin{bmatrix}
       (1+\beta)\I_d - \hat{\alpha}  A & -\beta \I_d \\
        \I_d & \bf{0}
        \end{bmatrix} 
        \begin{bmatrix}
        x_{k} - x^{\ast} \\
        x_{k-1} -  x^{\ast}
        \end{bmatrix}  
        - \tilde{\eta}_k    
        \begin{bmatrix}
        \nabla f(x_k) \\
        \bf{0}
        \end{bmatrix} \right\|^2 \notag \\
        & = \left\|\begin{bmatrix}
       (1+\beta)\I_d - \hat{\alpha}  A & -\beta \I_d \\
        \I_d & \bf{0}
        \end{bmatrix} 
        \begin{bmatrix}
        x_{k} - x^{\ast} \\
        x_{k-1} -  x^{\ast}
        \end{bmatrix}  \right\|^2 + \tilde{\eta}_k^2 \left\| \nabla f(x_k) \right\|^2  \notag \\ 
        & \quad - 2\tilde{\eta}_k   \begin{bmatrix}
        \nabla f(x_k)^{T} & 
        \bf{0}^{T}
        \end{bmatrix} \begin{bmatrix}
       (1+\beta)\I_d - \hat{\alpha}  A & -\beta \I_d \\
        \I_d & \bf{0}
        \end{bmatrix} 
        \begin{bmatrix}
        x_{k} - x^{\ast} \\
        x_{k-1} -  x^{\ast}
        \end{bmatrix} \notag \\
        & = \left\|\begin{bmatrix}
       (1+\beta)\I_d - \hat{\alpha}  A & -\beta \I_d \\
        \I_d & \bf{0}
        \end{bmatrix} 
        \begin{bmatrix}
        x_{k} - x^{\ast} \\
        x_{k-1} -  x^{\ast}
        \end{bmatrix}  \right\|^2 + \tilde{\eta}_k^2 \left\| \nabla f(x_k) \right\|^2  \notag \\
        & \quad - 2\tilde{\eta_k} \left( \left\langle \nabla f(x_k), x_k -x^{\ast}\right\rangle + \beta \left\langle \nabla f(x_k), x_k - x_{k-1} \right\rangle - \hat{\alpha} \left\|\nabla f(x_k) \right\|^2\right) \notag \\
        & \mathop{\leq}^{(a)} \left\|\begin{bmatrix}
       (1+\beta)\I_d - \hat{\alpha}  A & -\beta \I_d \\
        \I_d & \bf{0}
        \end{bmatrix} 
        \begin{bmatrix}
        x_{k} - x^{\ast} \\
        x_{k-1} -  x^{\ast}
        \end{bmatrix}  \right\|^2 + \tilde{\eta}_k^2 \left\| \nabla f(x_k) \right\|^2  \notag \\
        & \quad - 2\tilde{\eta}_k\left(f(x_k) - f^{\ast} + \frac{1}{2L}\left\|\nabla f(x_k) \right\|^2 + \beta \left\langle \nabla f(x_k), x_k - x_{k-1} \right\rangle - \frac{2-\beta}{2L}\left\|\nabla f(x_k) \right\|^2 \right) \notag \\
        & \mathop{=}^{(b)} \left\|\begin{bmatrix}
       (1+\beta)\I_d - \hat{\alpha}  A & -\beta \I_d \\
        \I_d & \bf{0}
        \end{bmatrix} 
        \begin{bmatrix}
        x_{k} - x^{\ast} \\
        x_{k-1} -  x^{\ast}
        \end{bmatrix}  \right\|^2 - \tilde{\eta}_k^2 \left\| \nabla f(x_k) \right\|^2 
     \end{align}  
    where $(a)$ follows from $\left\langle \nabla f(x_k), x_k -x^{\ast}\right\rangle \geq f(x_k) - f^{\ast} + \frac{1}{2L}\left\|\nabla f(x_k) \right\|^2$ and $\nabla f(x_k) = A(x_k -x^{\ast})$, and $(b)$ uses the formula of step size $\tilde{\eta}_k = \frac{f(x_k) - f(x^{\ast})}{\left\|\nabla f(x_k) \right\|^2} + \beta \frac{\left\langle \nabla f(x_k), x_k -x_{k-1} \right\rangle}{\left\|\nabla f(x_k) \right\|^2} - \frac{1-\beta}{2L}$. If $\left\langle \nabla f(x_k), x_k -x_{k-1} \right\rangle \leq - \left(f(x_k) - f^{\ast} \right) $, the above result is also correct due to that $\tilde{\eta}_k = 0$. Overall, we can derive that
    \begin{align}\label{main:inequ:hb}
         \left\|  \begin{bmatrix}
        x_{k+1} - x^{\ast} \\
        x_k -  x^{\ast}
        \end{bmatrix}  \right\|^2 
        & = \left\|\begin{bmatrix}
       (1+\beta)\I_d - \hat{\alpha}  A & -\beta \I_d \\
        \I_d & \bf{0}
        \end{bmatrix} 
        \begin{bmatrix}
        x_{k} - x^{\ast} \\
        x_{k-1} -  x^{\ast}
        \end{bmatrix}  \right\|^2 - \tilde{\eta}_k^2 \left\| \nabla f(x_k) \right\|^2. 
    \end{align}
    Let 
    \begin{align*}
       y_k := \begin{bmatrix}
       x_k-x^{\ast}  \\
       x_{k-1}-x^{\ast}
        \end{bmatrix}, \quad  \,\,  D := \begin{bmatrix}
       (1+\beta)\I_d - \hat{\alpha} A & -\beta \I_d \\
        \I_d & \bf{0}
        \end{bmatrix}.
    \end{align*}
 By (\ref{main:inequ:hb}), we obtain the exponential decrease in $\left\|y_k \right\|^2$:
 \begin{align*}
     \left\|y_{k+1} \right\| \leq \left\|D y_k \right\| \leq \left\|D^{k} y_1 \right\| \leq \left\| D^k\right\|_2\left\|y_1 \right\| \leq \left(\rho(D) + o(1)) \right)^{k}\left\|y_1 \right\|
 \end{align*}
     where $\rho(D)$ is the spectrum of $D$. In order to explicitly derive the convergence rate of ALR-HB (v2), we will turn to the eigenvalues of $D$. Furthermore, we can see that $D$ is permutation-similar to a block diagonal matrix with $2\times 2$ block $D_i$, that is
     \begin{align*}
         D \sim
         \begin{bmatrix}
       D_1 & \bf{0} & \cdots &  \bf{0} \\
       \bf{0}  & D_2  &  \cdots & \bf{0} \\
      \vdots &  & \cdots   & \vdots \\
     \bf{0}  & \bf{0}  & \cdots &  D_d
        \end{bmatrix} \, \text{where} \, D_i = 
        \begin{bmatrix}
               1+\beta - \hat{\alpha} \lambda_{i} & -\beta \\
               1 & 0
        \end{bmatrix} \,\text{for}\, \, i=1, 2, \cdots, d_1.
     \end{align*}
     Therefore, to get the eigenvalues of $D$, it is sufficient to compute the eigenvalues for all $D_i$. For any $i \in [d_1]$, the eigenvalues of the $2\times 2$ matrix are the roots of the quadratic function:
     \begin{align}
       L(s):=  s^2 - (1+\beta - \hat{\alpha} \lambda_{i})s + \beta = 0, \,\text{where}\,\, \Delta_i = (1+\beta-\hat{\alpha} \lambda_i)^2 - 4\beta
     \end{align}
     where $\hat{\alpha} = (2-\beta)/(2L)$ and $L=\lambda_{\max}$.       Because $\lambda_i / L \leq 1$, we have  $1+\beta-\hat{\alpha} \lambda_i =  1+\beta- (2-\beta)\lambda_i/(2L) \geq 1+\beta- (2-\beta)/2 = \frac{3\beta}{2} > 0$.  In this case, if $\Delta_i  \leq 0$, it is equivalent to $1+\beta-\hat{\alpha} \lambda_i \leq 2\sqrt{\beta}$. If $\beta \geq \min_{i}  \left(1- \sqrt{\frac{\lambda_i(L+\lambda_i)}{2L^2}} \right)^2  := \left(1- \sqrt{\frac{(1+\kappa)}{2\kappa^2}} \right)^2 $, then $\Delta_i \leq 0$ for each $i$. The best convergence rate is achieved by choosing $\beta = \hat{\beta}:=\left(1- \sqrt{\frac{(1+\kappa)}{2\kappa^2}} \right)^2$. The convergence rate is linear with $\rho(D) = \sqrt{\beta}$. 
     
     In the numerical experiments, we found that $\beta=\beta^{\ast}$ performs better. What is its convergence rate if we set $\beta=\beta^{\ast} < \hat{\beta}$? That is to say: there are $\Delta_i$ for $i \in [d_1]$ such that $\Delta_i > 0$. The quadratic function $L(s)$ must have two solutions denoted by $s_1 < s_2$. Because $L(0)=\beta >0$ and $L(1) = \hat{\alpha} \lambda_i > 0$, $s_1s_2 = \beta >0$, and $s_1 + s_2 = 1+ \beta - \hat{\alpha} \lambda_i > 0$. We can claim that $s_1 \in (0,\beta)$ and $s_2 \in (\beta, 1)$. The worst-case convergence is decided by the value of $s_2$. Next, we show that if $\beta = \beta^{\ast}$, 
     \begin{align*}
        s_2 & = \frac{(1+\beta)- \frac{2-\beta}{2L}\lambda_i + \sqrt{\left((1+\beta)- \frac{2-\beta}{2L}\lambda_i \right)^2 - 4\beta }}{2} \notag \\
        & = \frac{(1+\beta)- \frac{2-\beta}{2L}\lambda_i + \sqrt{(1-\beta)^2 - \frac{1}{\kappa}}}{2} \leq 1 - \frac{4-\sqrt{15}}{2(\sqrt{\kappa}+1)} + o\left(\frac{1}{\sqrt{\kappa}+1}\right)
     \end{align*}
     In this case, the convergence rate is at least $\rho^k$ where $\rho \approx 1 - \frac{4-\sqrt{15}}{2(\sqrt{\kappa}+1)} < 1$. Therefore, we have proved the linear convergence for ALR-HB(v2) with the rate at least $\rho = 1 - \frac{4-\sqrt{15}}{2(\sqrt{\kappa}+1)}$.

\end{proof}

\section{Supplementary Numerical Results and Details}\label{append:numerical}

\subsection{Details of Section \ref{sec:least:squares}  for Least-Squares Problems}
\label{append:least:squares}
In this part, we provide the details of the experiments on the least-squares problem.
The dimension $d_1=d=1000$ and the condition number $\kappa=10^4$. The theoretical optimal momentum parameter $\beta^{\ast} = 0.9606$. The initial point is randomly generated and then fix it for the different test algorithms.  If the step size is not specified, we select it from the set $\left\lbrace 10^{-3}, 10^{-2}, 10^{-1}, 1, 10^1, 10^2\right\rbrace$. For the L$^4$Mom method, we choose the momentum parameter from $\beta \in \left\lbrace 0.5, 0.9, 0.95, \beta^{\ast}, 0.99 \right\rbrace$ and the hyper-parameter $\alpha \in \left\lbrace 0.0001, 0.001, 0.01, 0.015, 0.1, 0.15, 1\right\rbrace$  as the original paper~\citep{L4}.

 If parameters $\mu$ and $L$ are unknown a priori, the details of the parameters are listed below: HB with best-tuned constant step size $\eta = 0.01$ and the best-tuned momentum parameter $\beta=0.99$ (its optimal value $\beta^{\ast}=0.9606$); For ALR-MAG and ALR-HB, we set $\beta = 0.95$; In L$^4$Mom, we choose $\beta = 0.95$ and $\alpha=0.01$.


\subsection{Results on Logistic Regression Problems}\label{sec:experiment:logistic}
  
To illustrate the practical behavior of ALR-SMAG and ALR-SHB in the convex interpolation setting, we perform experiments on logistic regression with both synthetic data and a classification dataset from LIBSVM~\footnote{\url{https://www.csie.ntu.edu.tw/~cjlin/libsvmtools/datasets/}}. We test our algorithms ALR-SHB and ALR-SMAG and compare with SPS\_${\max}$~\citep{SPS}, SGDM under constant step size, AdSGD~\citep{AdSGD}, SAHB~\citep{AHB}, and L$^4$Mom~\citep{L4}. Note that we do not estimate $f_{S_k}^{\ast}$ every iterate but set $f_{S_k}^{\ast}=0$.

First, we follow the experiments described in section 4.1 of SPS~\citep{SPS} on synthetic data for logistic regression. We do the grid search for all the parameters that are not specified and choose the best based on their practical performance.  The details of the parameters in synthetic experiments are given below: (1) SPS\_${\max}$~\citep{SPS} with $c \in \left\lbrace 0.1, 0.2, 0.5, 1, 5, 10, 20\right\rbrace$ and $ \eta_{\max} = 
\left\lbrace 0.01, 0.1, 1, 10, 100 \right\rbrace$: we set $\eta_{\max}=100$ and $c=5$; (2) SGD with momentum (SGDM) with best tuned constant step size: $\eta \in \left\lbrace 0.01, 0.1, 1, 10, 100 \right\rbrace$ and we choose $\eta = 10$ and $\beta = 0.9$; (3) AdSGD~\citep{AdSGD}: $\lambda_0 = 1$ with the pair of the parameters $(\sqrt{1 + 0.01 \theta}, 1/L_k)$; (4) SAHB~\citep{AHB}: we set $\gamma_1 = 1.2, \gamma = 1, C = 10$; (5) L$^4$Mom~\citep{L4}, the main parameter $\alpha \in \left\lbrace 0.001,0.0015, 0.01, 0.015, 0.1 0.15\right\rbrace$ (0.15 is recommended value, but we found that $\alpha=0.01$ works better in this case) and $\beta = 0.9$; (6) Our algorithms: $c \in \left\lbrace 0.1, 0.5, 1,  5,  10 \right\rbrace$, $ \eta_{\max} = \left\lbrace 0.01, 0.1, 1, 10, 100 \right\rbrace$ and $\beta = 0.9$: for ALR-SMAG, we choose $\eta_{\max} =100$ and $c=5$; for ALR-SHB, we set $\eta_{\max} = 100$ and $c=5$. The result is reported in Figure \ref{fig:ls:with:syn}. We observe that for HB and ALR-HB, the function value drops faster than other algorithms at the early stage of the training. After 400 steps, our algorithms ALR-SHB and ALR-SMAG perform better than others. 

\begin{figure}[ht]
\begin{center}
     \begin{subfigure}[b]{0.45\textwidth}
 \includegraphics[width=\textwidth]{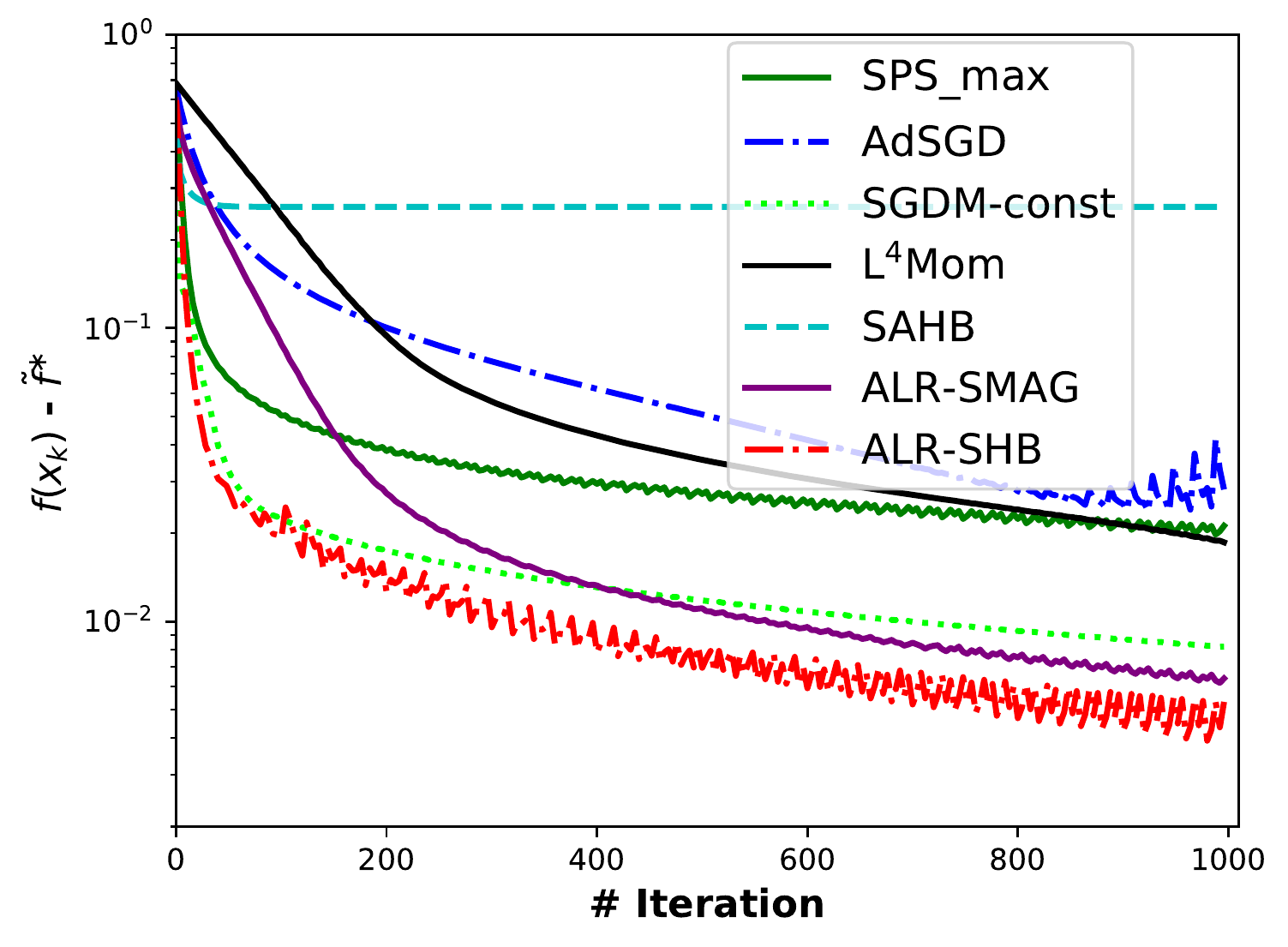}
   \captionsetup{justification=centering}
  \caption{Results on synthetic dataset}
         \label{fig:ls:with:syn}
  \end{subfigure}
 \hfill
  \begin{subfigure}[b]{0.45\textwidth}
\includegraphics[width=\textwidth]{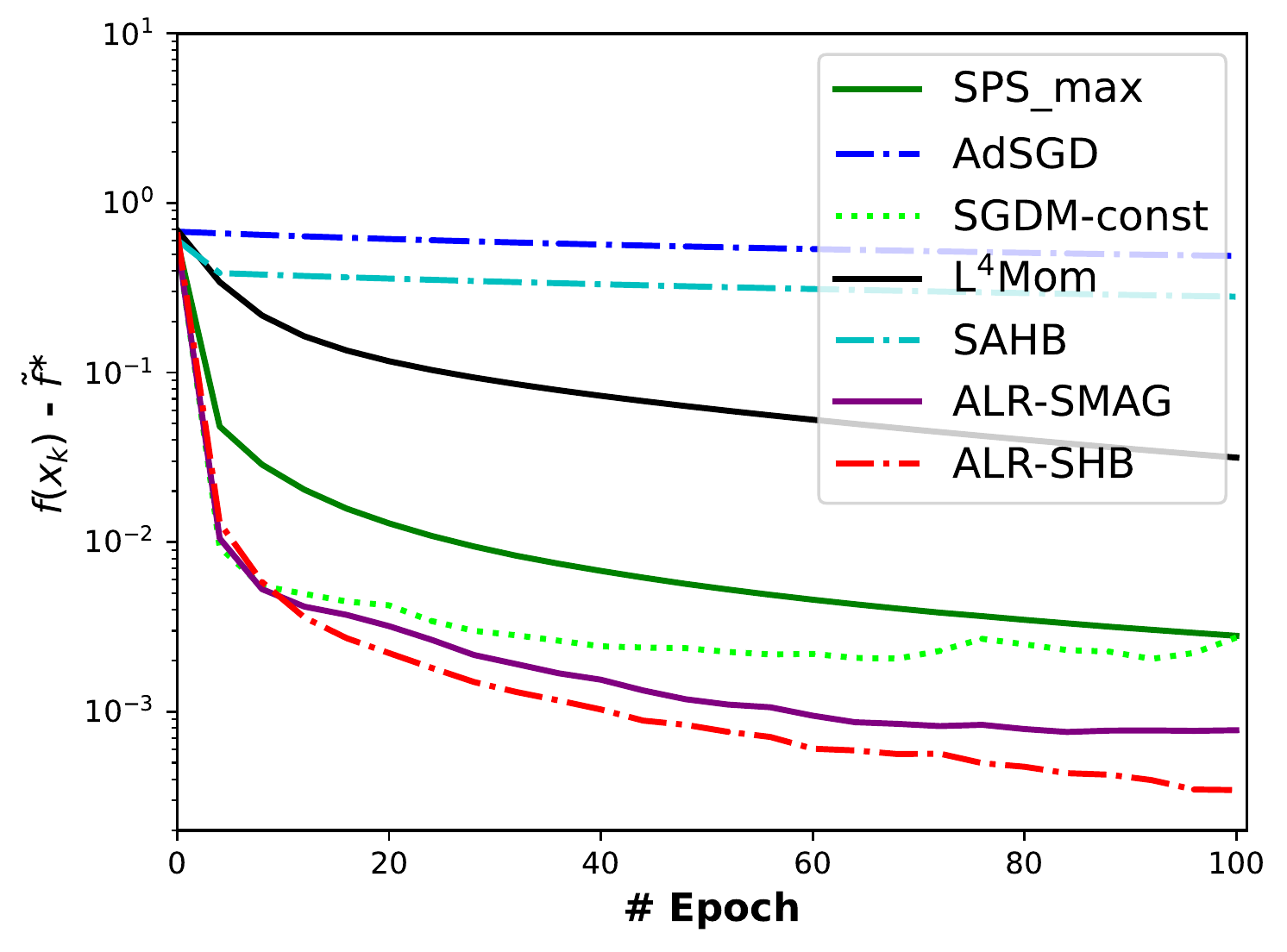}
\captionsetup{justification=centering}
         \caption{Results on rcv1}
         \label{fig:lr:rcv}
       \end{subfigure}
        \caption{Logistic regression}
        \label{fig:lr:st}
                
       \end{center}
\end{figure}


Similar to the synthetic dataset, we test logistic regression on a real binary
classification dataset RCV1 $(n = 20242; d = 47236)$ where a 0.75 partition of the dataset
is used for training, and the rest is for the test.   The batch size $b=100$ and the maximum epoch call is 100.  We can see that the optimality $f(x) - \tilde{f}^{\ast}$\footnote{$\tilde{f}^{\ast}$ is the estimation of $f^{\ast}$ by running heavy-ball for a very long time.} for ALR-SHB and ALR-SMAG 
 drops faster than others.  The details of the algorithms are: for SPS\_${\max}$~\citep{SPS}, we set $\eta_{\max}=100$ and $c=0.5$ (recommended from their paper); We set $\eta = 10$ for SGDM  and momentum parameter $\beta=0.9$; For L$^4$Mom~\citep{L4}, the parameter $\alpha=0.0015$ (0.15 is recommended value, but we found that $\alpha=0.0015$ works better) and $\beta = 0.9$; For AdSGD~\citep{AdSGD}, we set $(\sqrt{(1+0.01\theta)}, 1/L_k)$ . For SAHB~\citep{AHB}, we set $\gamma_1 = 1, \gamma = 0.5, C = 100$.
For our algorithms ALR-SHB and ALR-SMAG, we set $\eta_{\max} =100$ and $c=5$ for ALR-SMAG and  $\eta_{\max} = 100$ and $c=10$ for ALR-SHB.


\subsection{Numerical Results on CIFAR10 and Parameters Details of Section \ref{sec:numerical:dnn}}\label{append:cifar}

First, we provide the results on CIFAR10 with ResNet34~\citep{he2016deep}. In this experiment, we set the parameters for the tested algorithms as below: SGDM under constant step size $\eta = 0.01$; Adam with step size $\eta=0.001$ and $(\beta_1, \beta_2)=(0.9, 0.999)$; L$^4$Mom~\citep{L4} with $\alpha=0.01$; SPS\_${\max}$~\citep{SPS} with $\eta_0=0.1$ and $c=0.2$ with smoothing technique to update $\eta_{\max}$; SLS-acc~\citep{SLS} with $\eta_{0} = 1$ and $c=0.1$; ALR-SHB with $c=0.5$ and $\eta_{\max}=0.01$ (with the warmup, under $\eta_{\max}=0.1\min(10^{-4}k, 1)$ and $c=0.5$); ALR-SMAG with $c=0.1$ and $\eta_{\max}=0.01$ (with warmup under $\eta_{\max}=0.1\min(10^{-4}k, 1)$ and $c=0.1$).
\begin{figure}[ht]
\begin{center}
\includegraphics[width=0.45\textwidth]{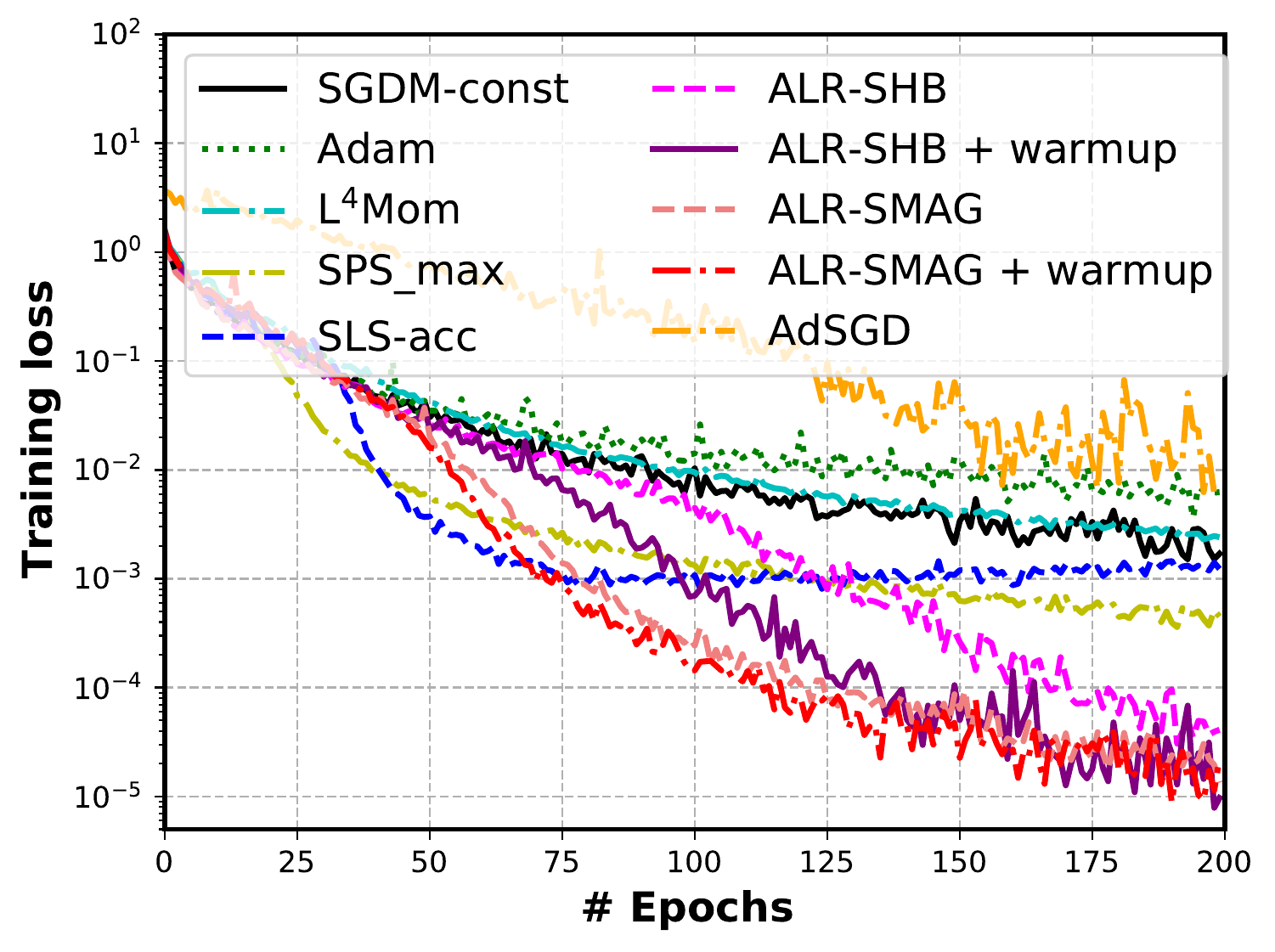}
  \hfill
\includegraphics[width=0.45\textwidth]{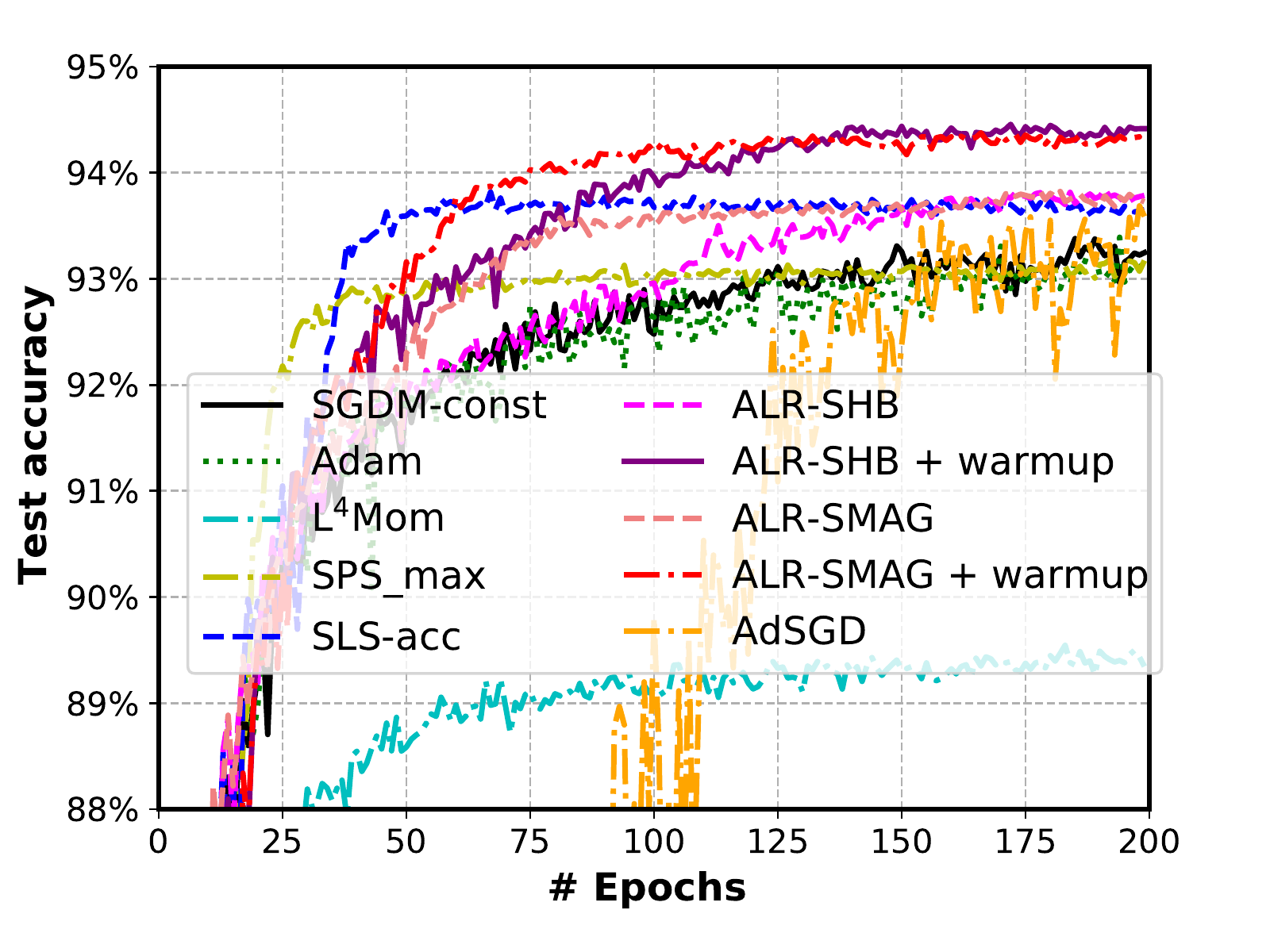}
        \caption{CIFAR10 - ResNet34: training loss (left) and test accuracy (right)}
        \label{fig:lr:cifar10}
            \end{center}
\end{figure}

For the experiments of CIFAR100 on WRN-28-10, the details of the algorithms:  SGDM under constant step size is shown below: $\eta \in \left\lbrace 0.001, 0.01, 0.1, 1 \right\rbrace$ and we set $\eta = 0.1$; SGDM with step-decay $\eta_k = \eta_0/ 10^{\lfloor k/K_0 \rfloor}$ where $K_0=\lceil K/3 \rceil$ where $K$ is the total number of iterations and we set $\eta_0=0.1$; Adam with $\eta = 0.001$ and $(\beta_1, \beta_2) = (0.9, 0.999)$; L$^4$Mom: we set $\alpha = 0.15$; stochastic line search with momentum (SLS-acc)~\citep{SLS} with $c=0.1$; SPS\_${\max}$~\citep{SPS}: we set $c = 0.2$ and $\eta_{\max}=1$ with smoothing technique; AdSGD with parameters $(\sqrt{1+0.02\theta}, 1/L_k)$. For our algorithms: ALR-SHB: $\eta_{\max} = 0.1$ and $c = 0.5$, ALR-SMAG: $\eta_{\max} = 0.1$ and $c = 0.05$ (for warmup, we set $c=0.5$ for ALR-SHB  and $c=0.05$ for ALR-SMAG, and $\eta_{\max}=\min(10^{-4}k, 1)$). In Figure~\ref{fig:smag:lr}, we present the adaptive step sizes of ALR-SMAG with and without warmup. The step size is not stably decreasing but hits the upper bound at the beginning of training, later drops for some iterations, and then hits the upper bound again in a somewhat irregular pattern.
\begin{figure}
 \begin{center}
\includegraphics[width=0.42\textwidth]{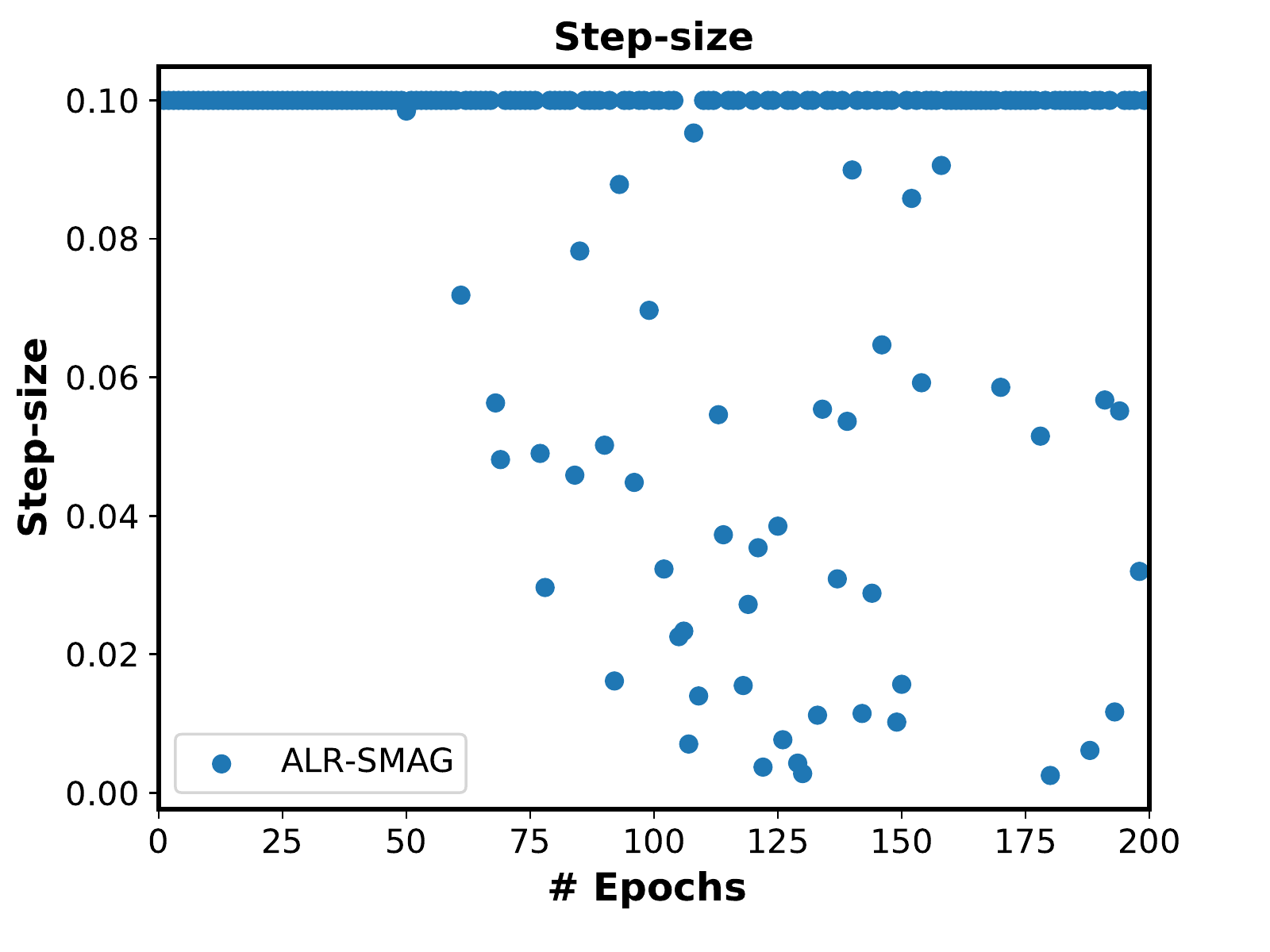} \hfill
\includegraphics[width=0.42\textwidth]{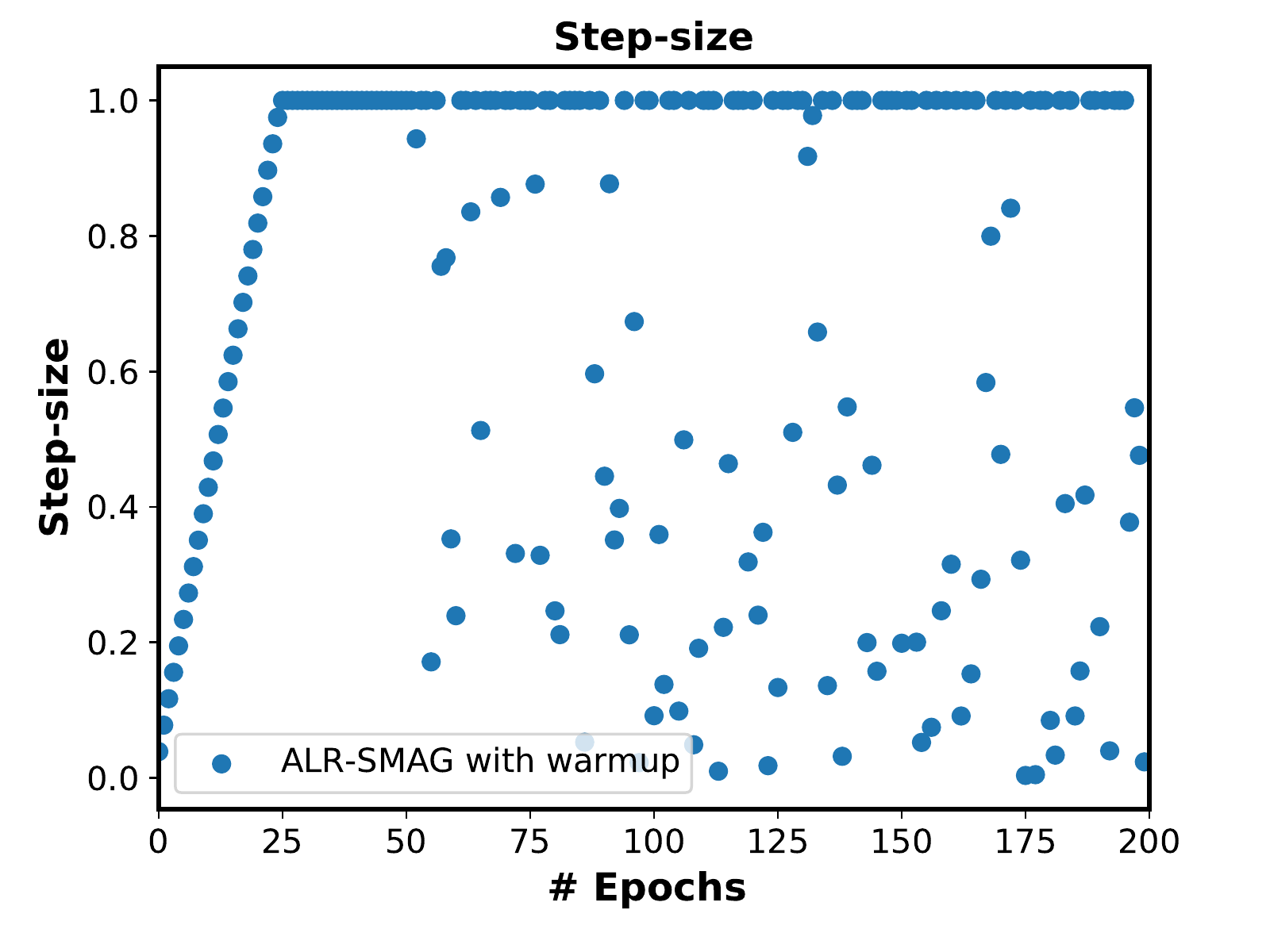}
    \caption{The plot of step sizes of ALR-SMAG and ALR-SMAG under warmup}
    \label{fig:smag:lr}
\end{center}
\end{figure}

The details of the algorithms on the experiment of CIFAR100 on DenseNet121: SGDM under constant step size $\eta = 0.01$; SGDM with step-decay $\eta_k = \eta_0/ 10^{\lfloor k/K_0 \rfloor}$ where $K_0=\lceil K/3 \rceil$ and $\eta_0=0.01$. For our algorithms: ALR-SHB with $c=0.5$ and $\eta_{\max}=0.01$, ALR-SMAG with $c=0.1$ and $\eta_{\max}=0.01$; For warmup, we set $c=0.5$ for ALR-SHB and $c=0.1$ for ALR-SMAG, and $\eta_{\max}=0.1 \min(10^{-4}k, 1)$. For the other algorithms, the parameters are the same as those on WRN-28-10.

Finally, we show how the hyper-parameter $c>0$ is related to the performance of ALR-SMAG. The parameter $c$ is tested from the set $\left\lbrace 0.05, 0.1, 0.2, 0.3, 0.5\right\rbrace$. The result is reported in Figure \ref{fig:mag:c}. We can see that the hyper-parameter $c$ is insensitive to the performance of ALR-SMAG in a small range $c \in [0,1,0.5]$. In this case, the results for $c=0.05, 0.1, 0.2$ are similar. We suggest that we might set the hyper-parameter $c$ to 0.1 in the experiments on CIFARs (CIFAR10 and CIFAR100).  

\begin{figure}
\begin{center}
\includegraphics[width=0.42\textwidth]{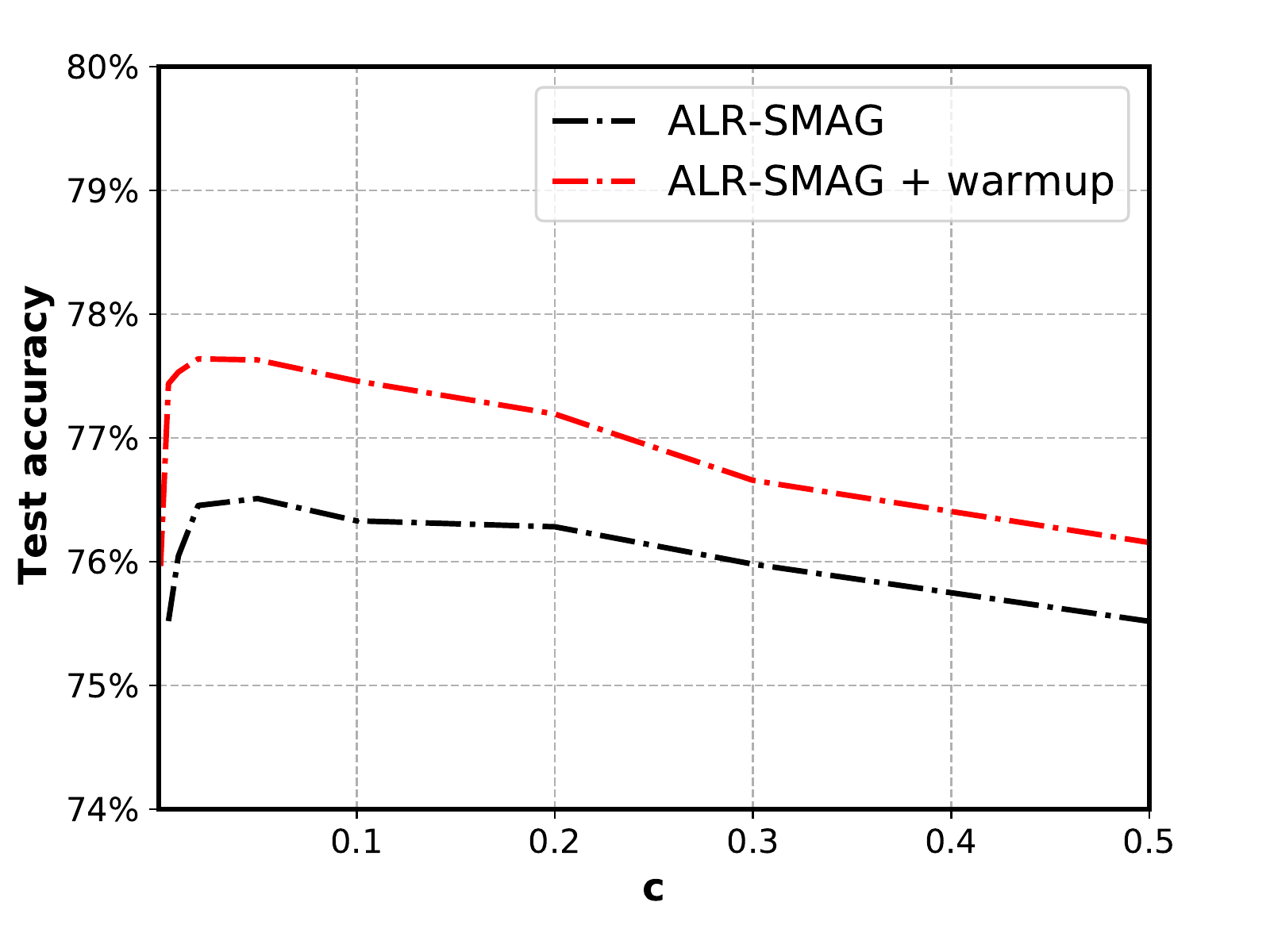}
\caption{The behavior of the parameter $c$ of ALR-SMAG on CIFAR100 - WRN-28-10}
    \label{fig:mag:c}
    \end{center}    
\end{figure}

 \subsection{Results on Tiny-ImageNet200}\label{append:tinyimagenet}

We now turn our attention to Tiny-ImageNet200~\citep{le2015tiny} on ResNet18~\citep{he2016deep} with the pre-trained model. This dataset includes 50000 images (200 classes) for training and 10000 images for the test. In this experiment, we compare our algorithms ALR-SHB and ALR-SMAG against SGD with momentum under constant step size, the popular step-decay~\citep{ge2019step} and cosine decay~\citep{loshchilov2016sgdr} step sizes, L$^4$Mom~\citep{L4} and Adam~\citep{Adam}. The results are reported in Table~\ref{tab:tinyimagenet:polyak-grad}. The maximal epoch call is 200 and the batch size is 256.

The details of the algorithms are shown below: SGD with momentum (SGDM) under constant step size $\eta = 0.01$; step-decay $\eta_k = \eta_0/10^{\lfloor k/K_0 \rfloor}$ where $K_0=\lceil K/3\rceil $; cosine decay step size $\eta_k = 0.5 \eta_0(\cos(k\pi /K)+1)$ where $K$ is the total number of iterations and $\eta_0=0.01$; Adam with constant step size $\eta = 0.001$; L$^4$Mom with $\alpha=0.15$. For our algorithm ALR-SMAG, we set $\eta_0=0.01$ and $c=0.5$; with warmup, we set $\eta_{\max}=\eta_0 \min(10^{-6}k, 1)$ with $\eta_0=0.1$ and $c=0.5$.

\begin{table}[h]
\caption{The result of test accuracy on Tiny-ImageNet200 - ResNet18}
\label{tab:tinyimagenet:polyak-grad}
\vskip 0.1in    
\begin{center}
\begin{small}
\begin{sc}
\begin{tabular}{lcccc}
\toprule  \cline{1-5}
 \multirow{2}{*}{Method}  & \multicolumn{4}{c}{Test accuracy (\%)} \\
& \#60 & \#120 & \#180 & Best\\ \midrule
SGDM-const & 64.88 &  65.05 & 65.17 & 65.87 $\pm$ 1.37 \\ \hline 
ADAM & 58.37 & 58.37 & 58.70 & 59.56 $\pm$ 0.35\\ \hline 
  L$^4$Mom & 65.84 & 65.58 & 65.51 &  66.87 $\pm$ 1.48  \\ \hline 
 SGDM-step  &  65.2 & {\bf 67.08} & {\bf 67.15} & 67.35 $\pm$ 1.24  \\   \hline 
SGDM-cosine & 65.75 & 66.69 & 66.95 & 67.13 $\pm$ 1.13 \\   \hline  
ALR-SMAG &   66.29  &  66.11 &  66.05 & 66.71 $\pm$ 1.69 \\ \hline
ALR-SMAG + Warmup &  {\bf 67.36} & 67 & 67.09 & {\bf 67.66 $\pm$ 1.05} \\
  \bottomrule
\end{tabular}
\end{sc}
\end{small}
\end{center}
\end{table}

For a wide range of problem classes, we can select $c$ from a small range $c \in \left\lbrace 0.1, 0.5\right\rbrace$. If the problem is 'difficult' to solve, i.e., necessitates a small step size, we recommend $c=0.3$ or $c=0.5$. For the problem at the level of training CIFARs, we can use $c=0.1$. If a user does not have any prior information about the problems and does not want to pay any effort to tune $c$, we recommend using $c=0.3$ since it works well for a wide range of problems and does not give significantly worse performance than a better-tuned value.

\subsection{Details of the Experiments in Section~\ref{sec:wd}}
\label{append:wd}
In the experiments for ALR-SMAG with weight-decay, the details of the algorithms are addressed as below: AdamW with step-decay step size:$\eta_0 = 0.001$ and $\eta_k = \eta_0/10^{\lfloor k/K_0 \rfloor}$ where $K_0=\lceil K/3\rceil $; SGDM with warmup: $\eta_k = \eta_0 \min\left(10^{-6}k, \frac{1}{\sqrt{k}} \right)$ and $\eta_0=0.1$;
SGDM under step-decay step size $\eta_k = \eta_0/10^{\lfloor k/K_0 \rfloor}$ with $\eta_0 = 0.1$ and $K_0=\lceil K/3\rceil$; SGDM under cosine step size without restart $\eta_k = 0.5 \eta_0(\cos(k\pi /K)+1)$ where $K$ is the total number of iterations and $\eta_0=0.1$; ALR-SMAG: $\eta_{\max}=0.1$ and $c=0.3$, $\lambda = 0.0005$. 
In the fine-tuning phase, the parameter $c$
 is exponentially increased after $K_{mid}$ steps and  $c = c_0 \exp^{\left(\frac{k-K_{mid}}{K - K_{mid}}\right)\ln(c_{\max}/c_0)}$. In this experiment, $K_{mid} = 0.8 K$ and $c_{\max}=100 c_0$ where $c_0=0.3$.
\begin{algorithm}[ht]
\captionof{algorithm}{ALR-SMAG with weight-decay }\label{alg:mag:wd}
\begin{algorithmic}[1]
  \STATE {\bfseries Input:}  $x_1$, $\beta \in (0,1)$, $c > 0,  \eta_{\max}, \lambda > 0, \epsilon=10^{-5}$
\WHILE{$x_k$ does not converge}
\STATE{$ k \leftarrow k+1 $}
		\STATE $g_k \leftarrow \frac{1}{|S_k|}\sum_{i \in S_k} \nabla f(x_k; \xi_i)$  \\\vspace{0.01in}
		\STATE $f_{S_k}(x_k) \leftarrow  \frac{1}{|S_k|}\sum_{i \in S_k}  f(x_k; \xi_i)$  
		\STATE $d_{k} \leftarrow \beta d_{k-1} + g_k$ 

		\STATE $
		    \eta_k \leftarrow  \min \left\lbrace \eta_{\max}, \frac{f_{S_k}(x_k)}{c \left\|d_{k} \right\|^2 + \epsilon} \right\rbrace $
		\STATE $x_{k+1} \leftarrow  x_k - \eta_k (d_{k} \bm{ + \lambda x_k} )$
\ENDWHILE
\end{algorithmic}
   \end{algorithm}

\end{document}